\documentclass[11pt]{article}

\usepackage{geometry}
\usepackage{latexsym}
\usepackage{mathrsfs}
\usepackage{amsfonts}
\usepackage{amssymb}
\usepackage{subfigure}    
\usepackage{graphicx}
\usepackage{epstopdf}
\usepackage{float}
\usepackage{multirow}
\usepackage{bm}
\usepackage{amsmath}
\usepackage{amsthm}
\usepackage{color}
\usepackage[subnum]{cases}
\usepackage{rotating}
\usepackage{enumitem}
\usepackage{accents}
\usepackage{algorithm}
\usepackage{algorithmic}
\usepackage[title]{appendix}
\usepackage{mathtools}
\usepackage{caption}
\usepackage{url}
\usepackage{booktabs}
\usepackage[numbers,square,sort&compress]{natbib}
\makeatletter
\@addtoreset{equation}{section}
\makeatother

\makeatletter
\def \wideubar{\underaccent{{\cc@style\underline{\mskip15mu}}}}
\def \widebar{\accentset{{\cc@style\underline{\mskip10mu}}}}
\makeatother

\definecolor{blue}{rgb}{0,0,0.9}
\definecolor{red}{rgb}{0.9,0,0}
\definecolor{green}{rgb}{0,0.9,0}
\definecolor{brown}{rgb}{0.6,0.1,0.1}
\definecolor{lightgreen}{rgb}{0.1,0.5,0.1}

\usepackage[colorlinks=true,
breaklinks=true,
bookmarks=true,
urlcolor=blue,
citecolor=lightgreen,
linkcolor=lightgreen,
bookmarksopen=false,
draft=false]{hyperref}

\graphicspath{{images/}}

\newtheorem{proposition}{Proposition}[section]

\newtheorem{definition}{Definition}[section]
\newtheorem{lemma}{Lemma}[section]

\newtheorem{theorem}{Theorem}[section]
\newtheorem{remark}{Remark}[section]

\newtheorem{assumption}{Assumption}

\begin{document}

\title{An Inexact Variable Metric Proximal Gradient-subgradient Algorithm for a Class of Fractional Optimization Problems}

\author{Lei Yang\thanks{School of Computer Science and Engineering, Sun Yat-sen University ({\tt yanglei39@mail.sysu.edu.cn}).},~~
Xiangrui Kong\thanks{School of Computer Science and Engineering, Sun Yat-sen University ({\tt kongxr5@mail2.sysu.edu.cn}).},~~
Min Zhang\thanks{School of Mathematics and Information Science, Guangzhou University ({\tt zhangmin1206@gzhu.edu.cn}).},~~
Yaohua Hu\thanks{(Corresponding author) School of Mathematical Sciences, Shenzhen University ({\tt mayhhu@szu.edu.cn}).}
}


\maketitle

\begin{abstract}
In this paper, we study a class of fractional optimization problems, in which the numerator of the objective is the sum of a convex function and a differentiable function with a Lipschitz continuous gradient, while the denominator is a nonsmooth convex function. This model captures ratio-type formulations arising in sparse learning and related applications. To address this class of problems, we propose an inexact variable metric proximal gradient-subgradient algorithm (iVPGSA), which, to the best of our knowledge, is the first inexact proximal algorithm specifically designed for such type of fractional problems. By incorporating a variable metric proximal term and allowing for approximate subproblem solutions under a flexible error criterion, the proposed algorithm is highly adaptable to a broader range of problems while achieving favorable computational efficiency. Under suitable assumptions, we establish that any accumulation point of the generated sequence is a critical point of the target problem. Moreover, we develop a Kurdyka-{\L}ojasiewicz (KL)-based analysis framework, relying only on the classical KL property and its associated exponent, to prove the global convergence of the entire sequence and characterize its convergence rate, \textit{without} requiring a strict sufficient descent property. Our results clarify how the classical KL exponent and inexactness jointly influence the convergence rate. Finally, numerical experiments on the $\ell_1/\ell_2$ Lasso problem and the constrained $\ell_1/\ell_2$ sparse optimization problem demonstrate the computational advantages of the iVPGSA over existing representative algorithms.

\vspace{5mm}
\end{abstract}

\section{Introduction}

Fractional minimization models arise naturally in a variety of settings, for example, when a normalized or scale-invariant criterion is desired, when a ratio-type regularizer is used to encode structural information such as sparsity, or when the target quantity has an intrinsic ratio form. Such models have found many applications in signal processing \citep{elx2013method,rwdl2019scale,tao2022minimization,tz2023study,tzx2024partly, yex2014ratio,zyp2021analysis}, image reconstruction \citep{jlsw2012image,wtnl2021limited,wtcnl2022minimizing}, and machine learning \citep{rpdcp2015euclid,yuan2023coordinate,yuan2025admm}. In sparse learning and regression, ratio structures are especially useful for promoting scale-invariant sparsity and reducing sensitivity to the magnitude of the decision variable.

Motivated by these fractional models, in this paper, we study the following composite fractional optimization problem:
\begin{equation}\label{mainpro}
\min_{\bm{x} \in \Omega} ~~ \frac{f(\bm{x}) + h(\bm{x})}{g(\bm{x})},
\end{equation}
where $f:\mathbb{R}^n\to(-\infty, \infty]$ is a proper closed convex function, $h: \mathbb{R}^n \rightarrow \mathbb{R}$ is a continuously differentiable function, $g:\mathbb{R}^n\to\mathbb{R}$ is a convex function, and the set $\Omega := \left\{ \bm{x} \in \mathbb{R}^n \colon g(\bm{x}) \neq 0 \right\}$ is assumed to be nonempty. More specific assumptions on model \eqref{mainpro} are summarized in Assumption \ref{assumA} as stated in Section \ref{sec-notation}. To connect model \eqref{mainpro} with sparse learning and regression, we present two concrete examples below, which are also used in our numerical experiments.
\begin{itemize}
\item ({\em Constrained $\ell_1/\ell_2$ sparse optimization}) The use of the $\ell_1/\ell_2$ norm in sparse signal recovery and sparse learning has garnered significant attention due to its scale-invariance and sparsity-inducing properties, especially compared to the traditional $\ell_1$-norm \citep{elx2013method,rwdl2019scale,tao2022minimization,tz2023study,tzx2024partly,yex2014ratio,zyp2021analysis}. The inverse problem of recovering a sparse signal from Gaussian noisy observations, using the $\ell_1/\ell_2$ norm as a measure of sparsity \citep{rwdl2019scale,zyp2021analysis}, can be formulated as
    \begin{equation}\label{ex1}
    \min_{\bm{x}\in\mathbb{R}^n}~~
    \frac{\|\bm{x}\|_1}{\|\bm{x}\|}
    \qquad \mathrm{s.t.} \qquad
    \|A\bm{x}-\bm{b}\| \leq \sigma,
    \end{equation}
    where $A\in\mathbb{R}^{m \times n}$ is a matrix with full row rank, $\bm{b}\in\mathbb{R}^m$ is an observation vector, and $\sigma > 0$ is a parameter controlling the noise level such that $\|\bm{b}\|>\sigma$. Model \eqref{ex1} fits \eqref{mainpro}, with $\Omega := \mathbb{R}^n \setminus \{0\}$, by taking
    \begin{equation*}
    f(\bm{x}) := \|\bm{x}\|_1
    + \iota_{\sigma}(A\bm{x} - \bm{b}),
    \quad h(\bm{x}) \equiv 0,
    \quad g(\bm{x}) := \|\bm{x}\|,
    \end{equation*}
    where $\iota_{\sigma}$ denotes the indicator function on the set $\left\{\bm{y}\in\mathbb{R}^m:\|\bm{y}\|\leq\sigma\right\}$.

\item ({\em Fractional sparse regression}) A broad class of fractional sparse regression models can be written as
    \begin{equation}\label{ex2}
    \min_{\bm{x}\in\mathbb{R}^n}~~
    \frac{\lambda\|\bm{x}\|_1+\sum_{i=1}^{m}
    \ell_i(\bm{a}_i^{\top}\bm{x}-b_i)}{\|\bm{x}\|}
    \qquad \mathrm{s.t.} \qquad
    \bm{\alpha} \leq \bm{x} \leq \bm{\beta},
    \end{equation}
    where $\lambda>0$, $\bm{\alpha},\bm{\beta}\in\mathbb{R}^n$ are finite bound vectors with $\bm{\alpha}\leq\bm{\beta}$, $\bm{b}:=(b_1,\cdots,b_m)\neq\bm{0}$, and $\bm{a}_i^{\top}$ denotes the $i$-th row of $A$. The loss functions $\ell_i$ ($i=1,\cdots,m$) are assumed to be nonnegative and continuously differentiable with Lipschitz continuous derivatives. This formulation includes the $\ell_1/\ell_2$ Lasso problem \citep{lszz2022proximal,zl2022first} by taking the least-squares loss $\ell_i(t):=\frac{1}{2}t^2$. It can also accommodate robust sparse regression models with nonconvex smooth losses, such as the Cauchy-type loss $\ell_i(t):=\log\left(1+\sigma^{-2}t^2\right)$ with $\sigma>0$.
    Model \eqref{ex2} fits \eqref{mainpro}, with $\Omega := \mathbb{R}^n \setminus \{0\}$, by taking
    \begin{equation*}
    f(\bm{x}) := \lambda\|\bm{x}\|_1+\iota_{[\bm{\alpha},\,\bm{\beta}]}(\bm{x}),
    \quad
    h(\bm{x}) := {\textstyle\sum_{i=1}^{m}}\ell_i(\bm{a}_i^{\top}\bm{x}-b_i),
    \quad
    g(\bm{x}) := \|\bm{x}\|,
    \end{equation*}
    where $\iota_{[\bm{\alpha},\,\bm{\beta}]}$ denotes the indicator function on the set $\left\{\bm{x}\in\mathbb{R}^n : \bm{\alpha}\leq\bm{x}\leq\bm{\beta}\right\}$. Moreover, when the loss functions $\ell_i$ ($i=1,\cdots,m$) are convex, model \eqref{ex2} also fits \eqref{mainpro} by taking
    \begin{equation*}
    f(\bm{x}) := \lambda\|\bm{x}\|_1
    + \iota_{[\bm{\alpha},\,\bm{\beta}]}(\bm{x})
    + {\textstyle\sum_{i=1}^{m}}\ell_i(\bm{a}_i^{\top}\bm{x}-b_i),
    \quad
    h(\bm{x}) \equiv 0,
    \quad
    g(\bm{x}) := \|\bm{x}\|.
    \end{equation*}
    Thus, the convex loss case admits two distinct reformulations that lead to different algorithmic frameworks with significantly different numerical behaviors. Numerical comparisons and related discussions are provided in Section \ref{sec-num-l12lasso}.

\end{itemize}

\subsection{Related work}

Fractional optimization problems of the form \eqref{mainpro} and their variants have attracted considerable attention due to their modeling flexibility. Meanwhile, their inherent nonconvexity poses significant algorithmic challenges and has motivated growing efforts to develop efficient solution methods; see, for example, \citep{bc2017proximal,bdl2022extrapolated,bdl2023inertial,blt2023full,lszz2022proximal,yuan2023coordinate,yuan2025admm,zl2022first,zzl2023equivalent}.
Among existing methods, Dinkelbach's parametric algorithm is arguably the most classical one. It employs an iterative scheme to transform a fractional problem into a sequence of non-fractional subproblems; see \citep{dinkelbach1967nonlinear} and \citep[Section 7.2.4]{cp2021modern}. Specifically, given an initial point $\bm{x}^0\in\Omega$, at the $k$-th iteration, the following subproblem is solved:
\begin{equation}\label{DPA}
\bm{x}^{k+1} = \arg\min_{\bm{x}\in\Omega}
\big\{ f(\bm{x}) + h(\bm{x}) - c_k g(\bm{x}) \big\},
\end{equation}
where $c_k := \frac{f(\bm{x}^k) + h(\bm{x}^k)}{g(\bm{x}^k)}$.
However, the resulting subproblem \eqref{DPA} remains nonconvex and can still be challenging to solve.

Recently, by further exploring the structure of the problem, Bo{\c{t}} and Csetnek \citep{bc2017proximal} introduced a proximal-gradient-type method to solve a class of fractional optimization problems in the form of \eqref{mainpro}, where the numerator $f+h$ is treated as a convex function, and the denominator $g$ is a convex and smooth function. Their method iteratively replaced $g$ in \eqref{DPA} with its linear approximation and incorporated a proximal term, resulting in the following subproblem:
\begin{equation}\label{BC-subprob}
\bm{x}^{k+1}
= \arg\min\limits_{\bm{x} \in \Omega} \left\{ f(\bm{x}) + h(\bm{x}) - c_k \langle \nabla g(\bm{x}^k), \,\bm{x} - \bm{x}^k \rangle + \frac{\gamma_k}{2} \|\bm{x} - \bm{x}^k\|^2 \right\},
\end{equation}
where $\gamma_k>0$ is the proximal parameter. Although the subproblem \eqref{BC-subprob} is convex under their assumptions, computing an exact solution remains challenging for a general $f+h$. Subsequently, Zhang and Li \citep{zl2022first} considered problem \eqref{mainpro} under assumptions similar to Assumption \ref{assumA}. 
To make the subproblem more tractable, at each iteration, they not only replaced $g$ in \eqref{DPA} with its linear approximation, but also replaced $h$ with its quadratic majorant approximation. This led to a so-called proximity-gradient-subgradient algorithm (PGSA), whose basic iterative step is given by:
\begin{equation}\label{PGSA}
\bm{x}^{k+1}
= \arg\min\limits_{\bm{x}\in\Omega} \left\{ f(\bm{x})
+ \langle \nabla h(\bm{x}^k) - c_k\bm{y}^k, \,\bm{x} - \bm{x}^k \rangle + \frac{\gamma_k}{2}\|\bm{x} - \bm{x}^k\|^2 \right\},
\end{equation}
where $\bm{y}^k \in \partial g(\bm{x}^k)$ and $\gamma_k>L_h$ can be determined by a (non)monotone line search. Notably, solving subproblem \eqref{PGSA} reduces to computing the proximal mapping of $f$ (rather than  $f+h$ in \eqref{BC-subprob}) and can be computationally efficient for many common choices of $f$. This makes the PGSA a promising solution method for tackling problem \eqref{mainpro}. More recently, inspired by Nesterov's acceleration techniques (see, e.g., \citep{n1983a,n2004introductory,n2013gradient}), Bo{\c{t}} et al. \citep{bdl2022extrapolated} developed an extrapolated proximal subgradient algorithm (ePSA) for solving problem \eqref{mainpro} under suitable assumptions and reported promising numerical performance. Later, Li et al. \citep{lszz2022proximal} introduced a proximal-gradient-subgradient algorithm with backtracked extrapolation (PGSA\_BE), which offers an alternative extrapolated proximal algorithm and has shown improved performance over the ePSA in solving the $\ell_1/\ell_2$ Lasso problem. Finally, we refer readers to \citep{bdl2023inertial,blt2023full,yuan2023coordinate,yuan2025admm,zll2026parameterized,zzl2023equivalent} for other recent advancements in fractional optimization.

\subsection{Our approach and contributions}

Although the aforementioned proximal algorithms (e.g., PGSA, ePSA, PGSA\_BE) provide efficient approaches for solving problem \eqref{mainpro}, they all require computing an \textit{exact} solution to the associated subproblem within their algorithmic frameworks. This, in turn, requires either a closed-form solution or an inner solver that computes a highly accurate solution, which is often unrealistic or computationally expensive in many applications, thereby limiting the applicability of these proximal algorithms. To overcome this inadequacy, it is desirable to allow for approximate subproblem solutions with progressively improving accuracy, while ensuring that the associated error criterion remains practically verifiable. However, to the best of our knowledge, research on inexact proximal algorithms for solving the fractional optimization problem \eqref{mainpro} remains largely unexplored. This gap is particularly relevant when the subproblem must be solved by an iterative method, as in the constrained $\ell_1/\ell_2$ sparse optimization problem considered in Section \ref{sec-num-l12cso}.

To address the lack of inexact proximal algorithms for solving the fractional optimization problem \eqref{mainpro}, in this paper, we develop an inexact variable metric proximal gradient-subgradient algorithm (iVPGSA). Specifically, as presented in Section \ref{sec-iVPGSA}, our inexact proximal algorithm relaxes the requirement of an exact solution to the subproblem by introducing a flexible inexact condition \eqref{inexcondH-x} and its associated error criterion \eqref{stopcritH-x}, which are well-defined and practically verifiable in the two test examples considered in Section \ref{sec-num}. Moreover, inspired by recent advances in the variable metric proximal gradient method (see, e.g., \citep{bpr2016variable,s2017variable,zpl2024vmipg}), we incorporate a variable metric proximal term into our inexact algorithmic framework to enhance its adaptability. A suitable choice of the variable metric allows the algorithm to exploit problem structure and may yield a more tractable subproblem that can be efficiently solved by a second-order method; see the test example in Section \ref{sec-num-l12cso} for more details.

While the proposed inexact proximal algorithm offers potential computational advantages, it also raises additional difficulties in convergence analysis. Indeed, unlike exact proximal algorithms, the iVPGSA does not generally satisfy a strict sufficient descent property for the original objective function. Instead, due to the inexact subproblem solution, it satisfies only an \textit{approximate} sufficient descent property, where the decrease in the objective function is guaranteed only up to a certain error term; see Proposition \ref{prop-obj}(i). Such perturbations destroy the monotonic decrease of the objective function values, and hence the standard KL analyses (see, e.g., \citep{ab2009on,abrs2010proximal,abs2013convergence,bst2014proximal,fgp2015splitting}) for exact descent methods are not directly applicable, especially when deriving convergence rates.

Several recent works have studied KL-based convergence analysis for inexact or stochastic algorithms under \textit{approximate} sufficient-descent-type conditions; see \citep{hl2023convergence,lmq2023convergence,qmlm2024kl,sun2021sequence}. Specifically, Sun \citep{sun2021sequence} established global sequential convergence by introducing the same auxiliary function as in \eqref{defpofun}, but did not address convergence rates. Such rate analysis is more challenging, as it requires a quantitative characterization of the interaction between the KL geometry and the error terms. Subsequent studies \citep{hl2023convergence,lmq2023convergence,qmlm2024kl} established global sequential convergence and convergence-rate estimates by employing \textit{strengthened} variants of the KL property and its associated exponent, in which the function-value differences in Definitions \ref{defKLfun} and \ref{defKLexpo} are replaced by their absolute values, respectively. In contrast, our analysis relies only on the classical KL property and its associated exponent. By combining an auxiliary function with a careful control of the inexactness errors, we establish the global convergence of the entire sequence and characterize how the classical KL exponent and the error decay jointly determine the convergence rate.

The key contributions and findings of this paper are summarized as follows: \vspace{-1mm}
\begin{itemize}
\item We develop an \textit{inexact variable metric proximal gradient-subgradient algorithm} (denoted by iVPGSA), which, to the best of our knowledge, is the first inexact proximal algorithm designed to solve the fractional optimization problem \eqref{mainpro}. The proposed algorithm accommodates inexact subproblem solutions through a flexible and verifiable error criterion, and incorporates a variable metric proximal term that can exploit problem structure and lead to more tractable subproblems.

\vspace{-2mm}
\item We conduct a comprehensive study of the convergence properties of the iVPGSA, including subsequential convergence, global convergence of the entire sequence, and convergence-rate estimates. Specifically, we develop a KL-based analysis framework under an \textit{approximate} sufficient descent property, relying only on the classical KL property and its associated exponent rather than strengthened KL inequalities. The key mechanism is to absorb the accumulated errors into an auxiliary function, which converts the perturbed descent relation for the original objective into an exact descent relation for the auxiliary function. In particular, we establish an explicit relationship between the KL exponents of the original objective and the auxiliary function, which clarifies how the original KL geometry and the error decay jointly determine the convergence rate. Our rate analysis may also provide useful insights for studying other inexact methods with similar perturbed descent relations.

\vspace{-2mm}
\item We conduct numerical experiments on the $\ell_1/\ell_2$ Lasso problem and the constrained $\ell_1/\ell_2$ sparse optimization problem, comparing the iVPGSA with representative existing algorithms. The results demonstrate the computational advantages of the proposed inexact variable-metric framework for solving problem \eqref{mainpro}.

\end{itemize}

The rest of this paper is organized as follows. We present notation and preliminaries in Section \ref{sec-notation}. We then describe our iVPGSA for solving problem \eqref{mainpro} and establish its preliminary convergence properties in Section \ref{sec-iVPGSA}. The extended convergence analysis based on the classical KL property and its associated exponent is presented in Section \ref{sec-conv}. Implementation details and numerical comparisons are reported in Section \ref{sec-num}, with conclusions given in Section \ref{sec-conc}.

\section{Notation and preliminaries}\label{sec-notation}

In this paper, we present scalars, vectors, and matrices in lowercase letters, bold lowercase letters, and uppercase letters, respectively. We use $\mathbb{R}$, $\mathbb{R}^n$ ($\mathbb{R}^n_+$), and $\mathbb{R}^{m\times n}$ to denote the set of real numbers, $n$-dimensional real (nonnegative) vectors, and $m\times n$ real matrices, respectively. For a vector $\bm{x}\in\mathbb{R}^n$, $x_i$ denotes its $i$-th entry, $\|\bm{x}\|$ denotes its Euclidean norm, $\|\bm{x}\|_1:=\sum^n_{i=1}|x_i|$ denotes its $\ell_1$ norm, and $\|\bm{x}\|_H:=\sqrt{\langle\bm{x},\,H\bm{x}\rangle}$ denotes its weighted norm associated with a symmetric positive definite matrix $H$.
For a closed convex set $\mathcal{X}\subseteq\mathbb{R}^n$, its indicator function $\iota_{\mathcal{X}}$ is defined by $\iota_{\mathcal{X}}(\bm{x})=0$ if $\bm{x}\in\mathcal{X}$ and $\iota_{\mathcal{X}}(\bm{x})=+\infty$ otherwise.

For an extended-real-valued function $h: \mathbb{R}^n \rightarrow [-\infty,\infty]$, we say that it is \textit{proper} if $h(\bm{x}) > -\infty$ for all $\bm{x} \in \mathbb{R}^n$ and its domain ${\rm dom}\,h:=\{\bm{x} \in \mathbb{R}^n: h(\bm{x}) < \infty\}$ is nonempty. A proper function $h$ is said to be closed if it is lower semicontinuous. Recall from \citep[Definition 8.3]{rw1998variational} that, for a proper closed function $h$, the regular (or Fr\'{e}chet) subdifferential and the (limiting) subdifferential of $h$ at $\bm{x}\in{\rm dom}\,h$ are defined as
\begin{equation*}
\begin{aligned}
\widehat{\partial} h(\bm{x}) &:= \left\{\bm{d} \in \mathbb{R}^{n} : \liminf\limits_{\bm{y} \rightarrow \bm{x}, \,\bm{y} \neq \bm{x}}\, \frac{h(\bm{y})-h(\bm{x})-\langle \bm{d}, \bm{y}-\bm{x}\rangle}{\|\bm{y}-\bm{x}\|} \geq 0\right\}, \\[3pt]
\partial h(\bm{x}) &:= \left\{ \bm{d} \in \mathbb{R}^{n}: \exists \,\bm{x}^k \xrightarrow{h} \bm{x}, ~\bm{d}^k \rightarrow \bm{d} ~~\mathrm{with}~~\bm{d}^k\in\widehat{\partial} h(\bm{x}^k) \right\},
\end{aligned}
\end{equation*}
respectively. When $h$ is continuously differentiable or convex, the above subdifferential coincides with the classical concept of derivative or convex subdifferential of $h$; see, e.g., \citep[Exercise~8.8]{rw1998variational} and \citep[Proposition~8.12]{rw1998variational}. For a proper closed convex function $h: \mathbb{R}^n \rightarrow (-\infty, \infty]$ and a given $\delta \geq 0$, the $\delta$-subdifferential of $h$ at $\bm{x}\in{\rm dom}\,h$ is defined by $\partial_\delta h(\bm{x}):=\{\bm{d}\in\mathbb{R}^n: h(\bm{y}) \geq h(\bm{x}) + \langle \bm{d}, \,\bm{y}-\bm{x} \rangle -\delta, ~\forall\,\bm{y}\in\mathbb{R}^n\}$, and when $\delta=0$, $\partial_\delta h$ simply reduces to the classical convex subdifferential $\partial h$. For any $\gamma>0$,
the proximal mapping of $\gamma h$ at $\bm{x}\in\mathbb{R}^n$ is defined by $\mathtt{prox}_{\gamma h}(\bm{x}):=\arg\min_{\bm{y}} \big\{h(\bm{y}) + \frac{1}{2\gamma}\|\bm{y} - \bm{x}\|^2\big\}$.

We next recall the Kurdyka-{\L}ojasiewicz (KL) property (see \citep{ab2009on,abrs2010proximal,abs2013convergence,bdl2007the,bst2014proximal} for more details), which is now a common technical condition for establishing the convergence of the whole sequence. A large number of functions such as proper closed semialgebraic functions satisfy the KL property \citep{abrs2010proximal,abs2013convergence}. For notational simplicity, let $\Xi_{\nu}$, where $\nu\in(0,+\infty]$, denote a class of concave functions $\varphi:[0,\nu) \rightarrow \mathbb{R}_{+}$ satisfying: (i) $\varphi(0)=0$; (ii) $\varphi$ is continuously differentiable on $(0,\nu)$ and continuous at $0$; (iii) $\varphi'(t)>0$ for all $t\in(0,\nu)$. Then, the KL property can be described as follows.

\begin{definition}[\textbf{KL property and KL function}]\label{defKLfun}
Let $h: \mathbb{R}^n \rightarrow \mathbb{R} \cup \{+\infty\}$ be a proper closed function.
\begin{itemize}
\item[(i)] For $\tilde{\bm{x}}\in{\rm dom}\,\partial h:=\{\bm{x} \in \mathbb{R}^{n}: \partial h(\bm{x}) \neq \emptyset\}$, if there exist a $\nu\in(0, +\infty]$, a neighborhood $\mathcal{V}$ of $\tilde{\bm{x}}$ and a function $\varphi \in \Xi_{\nu}$ such that for all $\bm{x} \in \mathcal{V} \cap \{\bm{x}\in \mathbb{R}^{n} : h(\tilde{\bm{x}})<h(\bm{x})<h(\tilde{\bm{x}})+\nu\}$, it holds that
    \begin{eqnarray*}
    \varphi'(h(\bm{x})-h(\tilde{\bm{x}}))\,\mathrm{dist}(0, \,\partial h(\bm{x})) \geq 1,
    \end{eqnarray*}
    then $h$ is said to have the \textbf{Kurdyka-{\L}ojasiewicz (KL)} property at $\tilde{\bm{x}}$.

\item[(ii)] If $h$ satisfies the KL property at each point of ${\rm dom}\,\partial h$, then $h$ is called a KL function.
\end{itemize}
\end{definition}

Based on the above definition, we introduce the KL exponent \citep{abrs2010proximal,lp2017calculus}.

\begin{definition}[\textbf{KL exponent}]\label{defKLexpo}
Suppose that $h: \mathbb{R}^n \rightarrow \mathbb{R} \cup \{+\infty\}$ is a proper closed function satisfying the KL property at $\tilde{\bm{x}}\in{\rm dom}\,\partial h$ with $\varphi(t)=\tilde{a} t^{1-\theta}$ for some $\tilde{a}>0$ and $\theta\in[0,1)$, i.e., there exist $a,\,\mu,\,\nu>0$, with $a=1/(\tilde{a}(1-\theta))$, such that
\begin{eqnarray*}
\mathrm{dist}(0, \,\partial h(\bm{x})) \geq a \left(h(\bm{x}) - h(\tilde{\bm{x}})\right)^{\theta}
\end{eqnarray*}
whenever $\bm{x} \in {\rm dom}\,\partial h$, $\|\bm{x}-\tilde{\bm{x}}\| \leq \mu$ and $h(\tilde{\bm{x}})<h(\bm{x})<h(\tilde{\bm{x}})+\nu$. Then, $h$ is said to have the KL property at $\tilde{\bm{x}}$ with an exponent $\theta$. If $h$ is a KL function and has the same exponent $\theta$ at any $\tilde{\bm{x}}\in{\rm dom}\,\partial h$, then $h$ is said to be a KL function with an exponent $\theta$.
\end{definition}

We also recall the uniformized KL property, which was established in \citep[Lemma 6]{bst2014proximal}.

\begin{proposition}[\textbf{Uniformized KL property}]\label{uniKL}
Suppose that $h: \mathbb{R}^n \rightarrow \mathbb{R} \cup \{+\infty\}$ is a proper closed function and $\Gamma$ is a compact set. If $h \equiv\kappa$ on $\Gamma$ for some constant $\kappa$ and satisfies the KL property at each point of $\Gamma$, then there exist $\mu,\,\nu>0$ and $\varphi \in \Xi_{\nu}$ such that
\begin{eqnarray*}
\varphi'(h(\bm{x})-\kappa)\,\mathrm{dist}(0, \,\partial h(\bm{x}))
\geq 1
\end{eqnarray*}
for all $\bm{x} \in \{\bm{x}\in\mathbb{R}^{n}: \mathrm{dist}(\bm{x},\,\Gamma)<\mu\} \cap \{\bm{x}\in \mathbb{R}^{n} : \kappa < h(\bm{x}) < \kappa + \nu\}$.
\end{proposition}

Moreover, we give three supporting lemmas that will be used in the subsequent analysis.

\begin{lemma}[{\citep[Lemma 2.2]{lp2017calculus}}]\label{normineq1}
Let $\alpha > 0$. Then, for any $\bm{w} = (w_1, \cdots, w_n)^{\top}\in\mathbb{R}^n_{+}$, there exist $0 < c_1 \leq c_2$ such that $c_1\|\bm{w}\| \leq \left(w_1^{\alpha}+\cdots+w_n^{\alpha}\right)^{\frac{1}{\alpha}} \leq c_2 \|\bm{w}\|$.
\end{lemma}

\begin{lemma}[{\citep[Section 2.2]{p1987introduction}}]\label{lemseqcond}
Suppose that $\{\alpha_k\}_{k=0}^{\infty}\subseteq\mathbb{R}$ and $\{\beta_k\}_{k=0}^{\infty}\subseteq\mathbb{R}_+$ are two sequences such that $\{\alpha_k\}$ is bounded from below, $\sum_{k=0}^{\infty} \beta_k < \infty$, and $\alpha_{k+1} \leq \alpha_{k} + \beta_k$ holds for all $k$. Then, $\{\alpha_k\}$ is convergent.
\end{lemma}

\begin{lemma} \label{lemma-convrate}
Let $\left\{u_k\right\}_{k \geq 1}$ be a sequence of nonnegative scalars, and let $b \geq 0$, $d,\,p,\,\omega>0$, and $\rho\in(0,1)$ be given constants.
\begin{itemize}
\item[{\rm (i)}] Suppose that $\left\{u_k\right\}_{k \geq 1}$ satisfies $u_{k+1} \leq \rho u_k + \frac{d}{k^p}$ for any $k\geq1$. Then, for any $\epsilon>0$, we have that $u_k \leq \frac{d+\epsilon}{1-\rho} \cdot \frac{1}{k^{p}}$ for all sufficiently large $k$.

\item[{\rm (ii)}] Suppose that $\left\{u_k\right\}_{k \geq 1}$ satisfies
    \begin{equation*}
    u_{k+1} \leq\left(1-\frac{\omega}{k+b}\right) u_k+\frac{d}{(k+b)^{p+1}}, \quad \forall\,k \geq 1.
    \end{equation*}
    Then, if $\omega>p$, we have $u_k \leq \frac{d}{\omega-p} \cdot\frac{1}{(k+b)^{p}}+o\left(\frac{1}{(k+b)^{p}}\right)$ for all sufficiently large $k$.
\end{itemize}
\end{lemma}
\begin{proof}
Statement (i) is adapted from \citep[Lemma 4 in Section 2.2]{p1987introduction} with its detailed proof provided in Appendix \ref{apd-lem-convrate}. Statement (ii) follows directly from \citep[Lemma 3.9(a)]{lmq2023convergence}.
\end{proof}

We further make the following blanket technical assumptions throughout the paper, which will be used in the subsequent convergence analysis. These assumptions are standard and frequently used in the design and convergence analysis of fractional optimization algorithms; see, for example, \citep{bdl2022extrapolated,blt2023full,lszz2022proximal}.

\begin{assumption}\label{assumA}
The functions $f$, $h$, $g$ in problem \eqref{mainpro} satisfy the following assumptions.
\begin{itemize}
    \item[{\bf A1.}] The function $f \colon \mathbb{R}^n \rightarrow \mathbb{R} \cup \{+\infty\}$ is proper, closed, and convex.
    \item[{\bf A2.}] The function $h \colon \mathbb{R}^n \rightarrow \mathbb{R}$ is continuously differentiable (possibly nonconvex) with a Lipschitz continuous gradient. That is, there exists a constant $L_h > 0$ such that
        \begin{align*}
        \|\nabla h(\bm{x}) - \nabla h(\bm{y})\| \leq L_h \|\bm{x} - \bm{y}\|, \quad \forall\,\bm{x}, \bm{y} \in \mathbb{R}^n.
        \end{align*}
    \item[{\bf A3.}] The function $g \colon \mathbb{R}^n \rightarrow \mathbb{R}$ is continuous, convex, and nonnegative over $\mathrm{dom}\,f$.
    \item[{\bf A4.}] The sum $f + h$ is nonnegative on $\mathrm{dom}\,f$, and $f(\bm{x})+h(\bm{x}) > 0$ when $g(\bm{x})=0$.
    \item[{\bf A5.}] The function $F \colon \mathbb{R}^n \rightarrow \mathbb{R} \cup \{+\infty\}$ defined as
        \begin{equation}\label{defFun}
        F(\bm{x}):=\left\{\begin{aligned}
        &\dfrac{f(\bm{x}) + h(\bm{x})}{g(\bm{x})},
        &&\text{if} ~~\bm{x}\in\mathrm{dom}\,f \cap \Omega, \\
        &+\infty, && \text{otherwise},
        \end{aligned}\right.
        \end{equation}
        is level-bounded, which means that for every $\alpha \in \mathbb{R}$, the level set $\left\{ \bm{x} \in \mathbb{R}^n : F(\bm{x}) \leq \alpha \right\}$ is bounded (and possibly empty).
\end{itemize}
\end{assumption}

Finally, we define a critical point of $F$, where $F$ is defined in \eqref{defFun}. Note that, since $f$ is convex in our setting, $\partial f=\widehat{\partial} f$, and hence the following definition is consistent with \citep[Definition 3.4]{zl2022first}. Therefore, one can follow \citep[Section 3]{zl2022first} to show that, if $\bm{x}^*$ is a local minimizer of problem \eqref{mainpro}, then $\bm{x}^*$ is a critical point of $F$.

\begin{definition}[\textbf{Critical point}]\label{definition-critical}
Let $\bm{x}^*\in \mathrm{dom}\,F$ and $c^* =F(\bm{x}^*)$. We say that $\bm{x}^*$ is a critical point of $F$ if
\begin{equation*}
0 \in  \partial f(\bm{x}^*) + \nabla h(\bm{x}^*) - c^*\partial g(\bm{x}^*).
\end{equation*}
\end{definition}

\section{An inexact variable metric proximal gradient-subgradient algorithm}\label{sec-iVPGSA}

In this section, we develop an inexact variable metric proximal gradient-subgradient algorithm (iVPGSA) for solving problem \eqref{mainpro} and study its preliminary convergence properties. The complete framework is presented in Algorithm \ref{algo-iVPGSA}.

\begin{algorithm}[ht]
\caption{An inexact variable metric proximal gradient-subgradient algorithm (iVPGSA) for solving problem \eqref{mainpro}}\label{algo-iVPGSA}
\textbf{Input:} Choose $\bm{x}^0 \in \Omega\cap\mathrm{dom}\,f$, a sequence of nonnegative scalars $\{\varepsilon_k\}_{k=0}^{\infty}$, and a sequence of positive definite matrices $\{H_k\}_{k=0}^{\infty}$. Set $k=0$. \\[2pt]
\textbf{while} a termination criterion is not met, \textbf{do} \vspace{-1mm}
\begin{itemize}[leftmargin=2cm]
\item[\textbf{Step 1}.] Compute $c_k = \frac{f(\bm{x}^k)+h(\bm{x}^k)}{g(\bm{x}^k)}$, and take any $\bm{y}^k \in \partial g(\bm{x}^k)$. If $0\in\partial f(\bm{x}^{k})+\nabla h(\bm{x}^{k})-c_k\bm{y}^{k}$, terminate and return $\bm{x}^{k}$.

\item[\textbf{Step 2}.] Find a point $\bm{x}^{k+1}\in\Omega\cap\mathrm{dom}\,f$ and an associated error pair $(\Delta^k, \delta_k)$ by approximately solving
    \begin{equation}\label{subproH-x}
    \min\limits_{\bm{x}} ~f(\bm{x}) + \langle \nabla h(\bm{x}^{k})-c_k \bm{y}^{k},\,\bm{x}-\bm{x}^{k} \rangle + \frac{1}{2}\| \bm{x}-\bm{x}^{k}\|^2_{H_k}
    \end{equation}
    such that
    \begin{equation}\label{inexcondH-x}
    \Delta^{k} \in \partial_{\delta_k} f(\bm{x}^{k+1}) + \nabla h(\bm{x}^{k})-c_k \bm{y}^{k} + H_k(\bm{x}^{k+1}-\bm{x}^{k}),
    \end{equation}
    satisfying the following error criterion:
    \begin{equation}\label{stopcritH-x}
    \|\Delta^{k}\|^2 + |\langle \Delta^{k},\, \bm{x}^{k+1}-\bm{x}^{k} \rangle| + \delta_k \leq \varepsilon_kg(\bm{x}^{k+1}).
    \end{equation}

\item[\textbf{Step 3}.] Set $k = k+1$ and go to \textbf{Step 1}. \vspace{-1.5mm}
\end{itemize}
\textbf{end while}  \\
\textbf{Output}: $\bm{x}^k$
\end{algorithm}

Note from the iVPGSA in Algorithm \ref{algo-iVPGSA} that, at each iteration, our inexact framework allows one to \textit{approximately} solve the subproblem \eqref{subproH-x} under condition \eqref{inexcondH-x} satisfying the error criterion \eqref{stopcritH-x}. Under the condition $H_k\succ\frac{L_h}{2}I_n$, we first show that \textbf{Step 2} is well-defined in the sense that the exact solution of \eqref{subproH-x}, whenever the critical condition in \textbf{Step 1} does not hold, belongs to $\Omega\cap\mathrm{dom}\,f$ and satisfies the error criterion with zero error. Specifically, starting from $\bm{x}^k\in\Omega\cap\mathrm{dom}\,f$, if $\bm{x}^k$ is not a solution of the subproblem \eqref{subproH-x}, we can always find $\bm{x}^{k+1}\in\Omega\cap\mathrm{dom}\,f$ which approximately solves \eqref{subproH-x}, satisfying both \eqref{inexcondH-x} and \eqref{stopcritH-x}.
If the critical condition in \textbf{Step 1} holds, $\bm{x}^{k}$ is the unique solution of \eqref{subproH-x} and a critical point of $F$ in Definition \ref{definition-critical}, but need not be an optimal solution of \eqref{mainpro}. For the convergence analysis, a sequence terminating at such a point is extended constantly, with $\bm{y}^{j}=\bm{y}^{k}$ and $\Delta^{j}=0$, $\delta_j=0$ for all $j\geq k$.
In fact, since $f$ is convex and $H_k$ is positive definite, the objective function in \eqref{subproH-x} is strongly convex and hence the subproblem \eqref{subproH-x} admits a unique solution $\bm{x}^{k,*}\in\mathrm{dom}\,f$, such that
\begin{equation*}
0 \in \partial f(\bm{x}^{k,*}) + \nabla h(\bm{x}^{k}) - c_k \bm{y}^{k}
+ H_k(\bm{x}^{k,*}-\bm{x}^{k}).
\end{equation*}
Then, using a similar derivation as for \eqref{ineq-VMsuffdes}, one can obtain that
\begin{equation}\label{suffdessubopt}
f(\bm{x}^{k,*}) + h(\bm{x}^{k,*})
+ \frac{1}{2}\|\bm{x}^{k,*}-\bm{x}^k\|_{2H_k-L_hI_n}^2
\leq c_kg(\bm{x}^{k,*}),
\end{equation}
where $L_h$ denotes the Lipschitz constant of $\nabla h$ and $I_n$ denotes the identity matrix in $\mathbb{R}^{n\times n}$. This, together with the nonnegativity of $f+h$ on $\mathrm{dom}\,f$ (by Assumption \ref{assumA}4), implies that, if $H_k\succ\frac{L_h}{2}I_n$ and $\bm{x}^k\in\Omega\cap\mathrm{dom}\,f$ is not the solution of \eqref{subproH-x} satisfying $c_k>0$\footnote{Note that when $c_k=0$, it means that $\bm{x}^k$ is already an optimal solution of the original problem \eqref{mainpro}, as guaranteed by Assumptions \ref{assumA}3 and \ref{assumA}4.}, we have that $g(\bm{x}^{k,*})>0$ and hence $\bm{x}^{k,*}\in\Omega$. Therefore, in this case, condition \eqref{inexcondH-x} and its associated error criterion \eqref{stopcritH-x} always hold at $\bm{x}^{k+1}=\bm{x}^{k,*}\in\Omega\cap\mathrm{dom}\,f$ with $\|\Delta^k\|=\delta_k=0$, ensuring that they are achievable. Moreover, if we set $\varepsilon_k\equiv0$ for all $k\geq0$, it means that $\bm{x}^{k+1}$ is always an exact solution of the subproblem. As a result, the iVPGSA reduces to an exact algorithm, which can serve as a variable metric generalization of the PGSA studied in \citep{zl2022first}.

One can also observe that, in our inexact framework, an inexact solution is characterized by an error term $\Delta^k$ appearing on the left-hand side of the approximate optimality condition \eqref{inexcondH-x} and $\partial_{\delta_k}f$ serving as an approximation of $\partial f$. This makes our inexact framework rather flexible to accommodate different scenarios when approximately solving the subproblem. Additionally, the incorporation of a suitable variable metric $H_k$ can further facilitate the computation of the subproblem in practice, as evidenced by the examples in our numerical section. Therefore, the features of inexactness and variable metric increase the flexibility of the proposed algorithmic framework and broaden its applicability to problems for which the subproblem is solved only approximately.

We now start our convergence analysis by establishing a useful result regarding the function values along the sequence generated by the iVPGSA.

\begin{lemma}\label{lemma-VMsuffdes}
Suppose that Assumption \ref{assumA} holds. Let $\{\bm{x}^k\}$ be the sequence generated by the iVPGSA in Algorithm \ref{algo-iVPGSA}. Then, for any $\bm{x}\in \Omega\cap\mathrm{dom}\,f$, we have
\begin{equation}\label{ineq-VMsuffdesgen}
\begin{aligned}
&\quad f(\bm{x}^{k+1}) + h(\bm{x}^{k+1})-c_kg(\bm{x}^{k+1})  \\
&\leq f(\bm{x}) + h(\bm{x})-c_kg(\bm{x})
+ \langle \Delta^k,\, \bm{x}^{k+1} - \bm{x}\rangle +\delta_k
+ L_h\|\bm{x}^{k+1}-\bm{x}\|\|\bm{x}^{k}-\bm{x}\|    \\
&\qquad
+ \frac{1}{2}\|\bm{x}^{k}-\bm{x}\|_{H_k}^2
- \frac{1}{2}\|\bm{x}^{k+1}-\bm{x}\|_{H_k-L_hI_n}^2
- \frac{1}{2}\|\bm{x}^{k+1}-\bm{x}^k\|_{H_k}^2 \\
&\qquad
+ c_k\big(g(\bm{x})-g(\bm{x}^k)+\langle \bm{y}^k, \,\bm{x}^k-\bm{x}\rangle\big),
\end{aligned}
\end{equation}
where $L_h$ denotes the Lipschitz constant of $\nabla h$ and $I_n$ denotes the identity matrix in $\mathbb{R}^{n\times n}$. Moreover, when taking $\bm{x}=\bm{x}^k$, we have
\begin{equation}\label{ineq-VMsuffdes}
f(\bm{x}^{k+1}) + h(\bm{x}^{k+1})
+ \frac{1}{2}\|\bm{x}^{k+1}-\bm{x}^k\|_{2H_k-L_hI_n}^2
\leq c_kg(\bm{x}^{k+1}) + \langle \Delta^k, \,\bm{x}^{k+1}-\bm{x}^{k}\rangle
+ \delta_k.
\end{equation}
\end{lemma}
\begin{proof}
First, from condition \eqref{inexcondH-x}, there exists a $\bm{d}^{k+1}\in \partial_{\delta_k}f(\bm{x}^{k+1})$ such that $\Delta^k = \bm{d}^{k+1} + \nabla h(\bm{x}^k)- c_k\bm{y}^k + H_k (\bm{x}^{k+1}-\bm{x}^k)$. Then, from the definition of $\partial_{\delta_k}f$, we have that, for any $\bm{x}\in\Omega\cap\mathrm{dom}\,f$,
\begin{equation*}
\begin{aligned}
f(\bm{x})
&\geq f(\bm{x}^{k+1})+\langle \bm{d}^{k+1}, \, \bm{x}-\bm{x}^{k+1} \rangle - \delta_k \\
& = f(\bm{x}^{k+1}) +\langle \Delta^k-\nabla h(\bm{x}^k)+c_k\bm{y}^k-H_k (\bm{x}^{k+1}-\bm{x}^k), \,\bm{x}-\bm{x}^{k+1} \rangle - \delta_k,
\end{aligned}
\end{equation*}
which implies that
\begin{equation}\label{ineq-VMf}
\begin{aligned}
f(\bm{x}^{k+1})
&\leq f(\bm{x}) - \langle \nabla h(\bm{x}^k),\, \bm{x}^{k+1} - \bm{x}\rangle + c_k\langle \bm{y}^k,\, \bm{x}^{k+1}-\bm{x} \rangle + \langle \Delta^k,\, \bm{x}^{k+1} - \bm{x}\rangle +\delta_k \\
&\qquad + \langle H_k (\bm{x}^{k+1}-\bm{x}^k), \, \bm{x}-\bm{x}^{k+1} \rangle  \\
&= f(\bm{x}) - \langle \nabla h(\bm{x}^k),\, \bm{x}^{k+1} - \bm{x}\rangle + c_k\langle \bm{y}^k,\, \bm{x}^{k+1}-\bm{x} \rangle + \langle \Delta^k,\, \bm{x}^{k+1} - \bm{x}\rangle +\delta_k \\
&\qquad + \frac{1}{2} \|\bm{x}^{k}-\bm{x}\|_{H_k}^2
- \frac{1}{2} \|\bm{x}^{k+1}-\bm{x}\|_{H_k}^2
- \frac{1}{2} \|\bm{x}^{k+1}-\bm{x}^k\|_{H_k}^2,
\end{aligned}
\end{equation}
where the last equality follows from $\langle H(\bm{a}-\bm{b}), \,\bm{c}-\bm{a} \rangle
= \frac{1}{2}\|\bm{b}-\bm{c}\|_H^2
- \frac{1}{2}\|\bm{a}-\bm{c}\|_H^2
- \frac{1}{2}\|\bm{a}-\bm{b}\|_H^2$.

Moreover, due to the convexity of $g$ and $\bm{y}^k \in \partial g(\bm{x}^k)$, we see that
\begin{equation*}
\begin{aligned}
g(\bm{x}^{k+1})
\geq g(\bm{x}^k) + \langle \bm{y}^k, \,\bm{x}^{k+1}-\bm{x}^k \rangle
= g(\bm{x}) + \langle \bm{y}^k, \,\bm{x}^{k+1}-\bm{x}\rangle
+ \big(g(\bm{x}^k)-g(\bm{x})+\langle \bm{y}^k, \,\bm{x}-\bm{x}^k\rangle\big),
\end{aligned}
\end{equation*}
which, together with the nonnegativity of $c_k$, implies that
\begin{equation}\label{ineq-VMg}
c_k\langle \bm{y}^k, \,\bm{x}^{k+1}-\bm{x}\rangle
\leq c_k\big(g(\bm{x}^{k+1}) - g(\bm{x})\big)
+ c_k\big(g(\bm{x})-g(\bm{x}^k)+\langle \bm{y}^k, \,\bm{x}^k-\bm{x}\rangle\big).
\end{equation}
On the other hand, using the fact that $\nabla h$ is Lipschitz continuous with a Lipschitz constant $L_h$ (by Assumption \ref{assumA}2), we have from
\citep[Lemma 1.2.3]{n2004introductory} that
\begin{equation*}
\begin{aligned}
&h(\bm{x}^{k+1})\leq h(\bm{x})+ \langle \nabla h(\bm{x}),\,\bm{x}^{k+1}-\bm{x} \rangle +\frac{L_h}{2} \Vert \bm{x}^{k+1}-\bm{x} \Vert^2,  \\
\Longrightarrow & \quad - \langle \nabla h(\bm{x}),\, \bm{x}^{k+1} - \bm{x}\rangle \leq h(\bm{x})-h(\bm{x}^{k+1}) + \frac{L_h}{2}\|\bm{x}^{k+1}-\bm{x}\|^2,
\end{aligned}
\end{equation*}
and hence
\begin{equation}\label{ineq-VMh}
\begin{aligned}
&- \langle \nabla h(\bm{x}^k),\, \bm{x}^{k+1} - \bm{x}\rangle
= - \langle \nabla h(\bm{x}),\, \bm{x}^{k+1} - \bm{x}\rangle
+ \langle \nabla h(\bm{x})-\nabla h(\bm{x}^k), \,\bm{x}^{k+1}-\bm{x}\rangle  \\
&\leq h(\bm{x}) - h(\bm{x}^{k+1})
+ \frac{L_h}{2}\|\bm{x}^{k+1}-\bm{x}\|^2
+ L_h\|\bm{x}^{k+1}-\bm{x}\|\|\bm{x}^{k}-\bm{x}\|.
\end{aligned}
\end{equation}

By summing \eqref{ineq-VMf}, \eqref{ineq-VMg}, \eqref{ineq-VMh} and rearranging the resulting terms, we can obtain \eqref{ineq-VMsuffdesgen}. Moreover, when taking $\bm{x}=\bm{x}^k$, we can further obtain \eqref{ineq-VMsuffdes} by using $c_k g(\bm{x}^k) = f(\bm{x}^k) + h(\bm{x}^k)$ and rearranging the resulting terms.
\end{proof}

Based on the above Lemma, we further have the following results.

\begin{proposition}\label{prop-iVPGSA}
Suppose that Assumption \ref{assumA} holds and there exist $\underline{\gamma}, \,\overline{\gamma}>0$ such that $\frac{L_h}{2}<\underline{\gamma}\leq\overline{\gamma}$ and $\underline{\gamma}I_n \preceq H_k \preceq \overline{\gamma}I_n$ for all $k\geq0$. Let $\{\bm{x}^k\}$ be the sequence generated by the iVPGSA in Algorithm \ref{algo-iVPGSA}. Then, for any $k\geq0$, we have
\begin{equation}\label{decreaseFVMeps}
F(\bm{x}^{k+1}) \leq F(\bm{x}^{k})
- \frac{\|\bm{x}^{k+1}-\bm{x}^k\|_{2H_k-L_hI_n}^2}{2g(\bm{x}^{k+1})}
+ \varepsilon_k.
\end{equation}
Moreover, if $\sum_{k=0}^{\infty}\varepsilon_k<+\infty$, the following statements hold.
\begin{itemize}
\item[(i)] $c^*:=\lim\limits_{k\rightarrow\infty}{c_k}= \lim\limits_{k\rightarrow\infty}{F(\bm{x}^k)}$ exists;

\item[(ii)] The sequence $\{\bm{x}^k\}$ is bounded;

\item[(iii)]  $\lim\limits_{k\rightarrow\infty}\cfrac{\|\bm{x}^{k+1}-\bm{x}^k\|^2}{g(\bm{x}^{k+1})}=0$ and $\lim\limits_{k\rightarrow\infty}\|\bm{x}^{k+1}-\bm{x}^k\|=0$;

\item[(iv)] Any accumulation point $\bm{x}^*$ of $\{\bm{x}^k\}$ satisfies that $\bm{x}^*\in \mathrm{dom}\,F$ and $F(\bm{x}^*)=c^*$;

\item[(v)] There exist $0<\underline{g}<\overline{g}$ such that $\underline{g}\leq g(\bm{x}^k)\leq \overline{g}$ for all $k\geq0$.
\end{itemize}
\end{proposition}
\begin{proof}
Due to Assumption \ref{assumA}3 and $\bm{x}^{k+1}\in\Omega\cap\mathrm{dom}\,f$ for all $k\geq0$, we can divide $g(\bm{x}^{k+1})$ on both sides of \eqref{ineq-VMsuffdes} and obtain that
\begin{equation*}
F(\bm{x}^{k+1}) + \frac{\|\bm{x}^{k+1}-\bm{x}^k\|_{2H_k-L_hI_n}^2}{2g(\bm{x}^{k+1})}
\leq  F(\bm{x}^{k}) + \frac{1}{g(\bm{x}^{k+1})}\left(\langle \Delta^k,\,\bm{x}^{k+1} - \bm{x}^{k}\rangle+\delta_k\right).
\end{equation*}
This, together with criterion \eqref{stopcritH-x} and the positivity of $g(\bm{x}^{k+1})$ (by Assumption \ref{assumA}3), readily yields \eqref{decreaseFVMeps}. In the following, we assume that $\sum_{k=0}^{\infty}\varepsilon_k<+\infty$.

\textit{Statement (i)}. We see from \eqref{decreaseFVMeps} and $H_k\succeq\underline{\gamma}I_n$ with $\underline{\gamma}>\frac{L_h}{2}$ for all $k\geq0$ that $F(\bm{x}^{k+1}) \leq F(\bm{x}^{k}) + \varepsilon_k$. Since $F(\bm{x}^{k})$ is nonnegative for all $k\geq0$ (by Assumptions \ref{assumA}3 and \ref{assumA}4), it then follows from Lemma \ref{lemseqcond} that $\{F(\bm{x}^k)\}$ is convergent and hence $c^*:=\lim\limits_{k\rightarrow\infty}c_k= \lim\limits_{k\rightarrow\infty}F(\bm{x}^k)$ exists.

\textit{Statement (ii)}. It is clear from statement (i) that $\{F(\bm{x}^k)\}$ is bounded. This, together with Assumption \ref{assumA}5, readily leads to the boundedness of $\{\bm{x}^k\}$.

\textit{Statement (iii)}. Summing \eqref{decreaseFVMeps} from $k=0$ to $k=K$, we obtain that
\begin{equation*}
F(\bm{x}^{K+1})
+ \sum_{k=0}^{K}\frac{\|\bm{x}^{k+1}-\bm{x}^k\|_{2H_k-L_hI_n}^2}{2g(\bm{x}^{k+1})}
\leq  F(\bm{x}^{0}) + \sum_{k=0}^{K}\varepsilon_k,
\end{equation*}
which, together with the nonnegativity of $\{F(\bm{x}^{k})\}$ (by Assumptions \ref{assumA}3 and \ref{assumA}4) and $H_k\succeq\underline{\gamma}I_n$ (with $\underline{\gamma}>\frac{L_h}{2}$) for all $k\geq0$, implies that
\begin{equation*}
\left(\underline{\gamma}-\frac{L_h}{2}\right)\sum_{k=0}^{K}
\frac{\|\bm{x}^{k+1}-\bm{x}^k\|^2}{g(\bm{x}^{k+1})}
\leq \sum_{k=0}^{K}\frac{\|\bm{x}^{k+1}-\bm{x}^k\|_{2H_k-L_hI_n}^2}{2g(\bm{x}^{k+1})}
\leq F(\bm{x}^{0}) + \sum_{k=0}^{K}\varepsilon_k.
\end{equation*}
Thus, when $\sum_{k=0}^{\infty}\varepsilon_k<+\infty$, we have that $\sum_{k=0}^{\infty}\frac{\|\bm{x}^{k+1}-\bm{x}^k\|^2}{g(\bm{x}^{k+1})}<+\infty$ and hence $\lim\limits_{k\rightarrow\infty}{\frac{\|\bm{x}^{k+1}-\bm{x}^k\|^2}{g(\bm{x}^{k+1})}}=0$. Moreover, from the boundedness of $\{\bm{x}^k\}$ and the continuity of $g$, we see that $g(\bm{x}^k)\leq \bar{g}$ for some $\bar{g}>0$. This together with $\lim\limits_{k\rightarrow\infty}{\frac{\|\bm{x}^{k+1}-\bm{x}^k\|^2}{g(\bm{x}^{k+1})}}=0$ further implies that $\lim\limits_{k\to\infty}\|\bm{x}^{k+1}-\bm{x}^k\|=0$.

\textit{Statement (iv)}. Since $\{\bm{x}^k\}$ is bounded, there exists at least one accumulation point. Let $\bm{x}^*$ be an arbitrary accumulation point of $\{\bm{x}^k\}$ and $\{\bm{x}^{k_j}\}$ be a subsequence such that $\lim\limits_{j\rightarrow \infty}{{\bm{x}}^{k_j}}=\bm{x}^*$. We first prove $\bm{x}^*\in \mathrm{dom}\,F$, i.e., $\bm{x}^*\in \mathrm{dom}\,f\cap\Omega$. From \eqref{ineq-VMsuffdes}, we see that
\begin{equation*}
\begin{aligned}
&\quad f(\bm{x}^{k_j+1}) + h(\bm{x}^{k_j+1})
+ \frac{1}{2}\|\bm{x}^{k_j+1}-\bm{x}^{k_j}\|_{2H_{k_j}-L_hI_n}^2  \\
&\leq c_{k_j}g(\bm{x}^{k_j+1})
+ \langle \Delta^{k_j}, \,\bm{x}^{k_j+1}-\bm{x}^{k_j}\rangle
+ \delta_{k_j}
\leq c_{k_j}g(\bm{x}^{k_j+1})
+ \varepsilon_{k_j} g(\bm{x}^{k_j+1}),
\end{aligned}
\end{equation*}
where the last inequality follows from criterion \eqref{stopcritH-x}. Passing to the limit in the above relation along $\{\bm{x}^{k_j}\}$ and recalling that $f$ is lower semicontinuous (by Assumption \ref{assumA}1), $c_k \to c^*$ (by statement (i)), $\|\bm{x}^{k+1}-\bm{x}^k\|\to0$ (by statement (iii)), $h$ and $g$ are continuous (by Assumptions \ref{assumA}2 and \ref{assumA}3), $g(\bm{x}^k)\leq \bar{g}$ for some $\bar{g}>0$ and $\varepsilon_k\to0$, we obtain that $f(\bm{x}^*) + h(\bm{x}^*) \leq c^*g(\bm{x}^*)$, which indicates that $\bm{x}^*\in\mathrm{dom}\,f$.
It then follows from the nonnegativity of $f+h$ on $\mathrm{dom}\,f$ (by Assumption \ref{assumA}4) that
\begin{equation*}
0 \leq f(\bm{x}^*) + h(\bm{x}^*) \leq c^*g(\bm{x}^*).
\end{equation*}
This together with the fact that $f(\bm{x}) + h(\bm{x}) > 0$ when $g(\bm{x})=0$ (by Assumption \ref{assumA}4) implies that $g(\bm{x}^*)>0$ and hence $\bm{x}^*\in\Omega$. Consequently, we have that $\bm{x}^*\in\mathrm{dom}\,f\cap\Omega$.

Next, we prove that $F(\bm{x}^*)=c^*$ by showing $F(\bm{x}^*)\leq c^*$ and $F(\bm{x}^*)\geq c^*$. On one hand, we have that
\begin{equation*}
\begin{aligned}
F(\bm{x}^*)
= \frac{f(\bm{x}^*)+h(\bm{x}^*)}{g(\bm{x}^*)}
\leq \liminf\limits_{j\to\infty}\,\frac{f(\bm{x}^{k_j})+h(\bm{x}^{k_j})}{g(\bm{x}^{k_j})}
= \liminf\limits_{j\to\infty}\,F(\bm{x}^{k_j})
= c^*,
\end{aligned}
\end{equation*}
where the inequality follows from Assumption \ref{assumA} that $f$ is lower semicontinuous, and $h$ and $g$ are continuous. On the other hand, by setting $\bm{x}=\bm{x}^*$ in \eqref{ineq-VMsuffdesgen}, we see that
\begin{equation*}
\begin{aligned}
&\quad f(\bm{x}^{k+1})+h(\bm{x}^{k+1})-c_kg(\bm{x}^{k+1})  \\
&\leq f(\bm{x}^*)+h(\bm{x}^*)-c_kg(\bm{x}^*)
+ \langle \Delta^k,\, \bm{x}^{k+1} - \bm{x}^*\rangle +\delta_k
+ L_h\|\bm{x}^{k+1}-\bm{x}^*\| \|\bm{x}^{k}-\bm{x}^*\|    \\
&\qquad
+ \frac{1}{2}\|\bm{x}^{k}-\bm{x}^*\|_{H_k}^2
- \frac{1}{2}\|\bm{x}^{k+1}-\bm{x}^*\|_{H_k-L_hI_n}^2
- \frac{1}{2}\|\bm{x}^{k+1}-\bm{x}^k\|_{H_k}^2 \\
&\qquad
+ c_k\Big(g(\bm{x}^*)-g(\bm{x}^k)+\langle \bm{y}^k, \,\bm{x}^k-\bm{x}^*\rangle\Big),
\end{aligned}
\end{equation*}
which implies that
\begin{equation}\label{refsuffdesgen-VM}
\begin{aligned}
F(\bm{x}^{k+1}) - c_k
&\leq \frac{g(\bm{x}^*)}{g(\bm{x}^{k+1})}
\left(\frac{f(\bm{x}^*)+h(\bm{x}^*)}{g(\bm{x}^*)}-c_k\right)  \\
&\qquad + \frac{1}{g(\bm{x}^{k+1})}\Big( \langle \Delta^k,\, \bm{x}^{k+1} - \bm{x}^*\rangle + \delta_k
+ L_h\|\bm{x}^{k+1}-\bm{x}^*\|\|\bm{x}^{k}-\bm{x}^*\|  \\[3pt]
&\qquad + \frac{1}{2}\|\bm{x}^{k}-\bm{x}^*\|_{H_k}^2
- \frac{1}{2}\|\bm{x}^{k+1}-\bm{x}^*\|_{H_k-L_hI_n}^2
- \frac{1}{2}\|\bm{x}^{k+1}-\bm{x}^k\|_{H_k}^2\Big) \\
&\qquad + \frac{c_k}{g(\bm{x}^{k+1})}\left(g(\bm{x}^*)-g(\bm{x}^k)+\langle \bm{y}^k, \,\bm{x}^k-\bm{x}^*\rangle\right).
\end{aligned}
\end{equation}
Moreover, since $g(\bm{x}^k)\leq \bar{g}$ for some $\bar{g}>0$ and $\varepsilon_k\to0$, we have from the error criterion \eqref{stopcritH-x} that $\Delta^k\to0$ and $\delta_k\to0$. Additionally, since $g$ is a continuous convex function and $\{\bm{x}^{k}\}$ is bounded, we know from
\citep[Theorem 23.4]{r1970convex} that $\{\bm{y}^{k}\}$ is bounded. Thus, passing to the limit in \eqref{refsuffdesgen-VM} along $\{\bm{x}^{k_j-1}\}$, and invoking $\lim\limits_{k\rightarrow\infty}{c_k}= \lim\limits_{k\rightarrow\infty}{F(\bm{x}^{k})}=c^*$ (by statement (i)), $\lim\limits_{k\rightarrow\infty}\|\bm{x}^{k+1}-\bm{x}^k\|=0$ (by statement (iii)) and $g(\bm{x}^*)>0$, we can obtain that $F(\bm{x}^*) \geq c^*$. 

\textit{Statement (v)}. It is easy to see from the boundedness of $\{\bm{x}^k\}$ and the continuity of $g$ that $g(\bm{x}^k)\leq\bar{g}$ for some $\bar{g}>0$. On the other hand, note that $g(\bm{x}^k)>0$ (since $\bm{x}^k\in\Omega$) for all $k\geq0$ and any accumulation point $\bm{x}^*$ of $\{\bm{x}^k\}$ satisfies that $\bm{x}^*\in\Omega$ (by statement (iv)) and thus $g(\bm{x}^*)>0$. Using these facts together with the boundedness of $\{\bm{x}^k\}$ and the continuity of $g$, one can verify that $g(\bm{x}^k)\geq\underline{g}$ for some $\underline{g}>0$. This completes the proof.
\end{proof}

With the above preparations, we are now ready to establish the subsequential convergence for the iVPGSA.

\begin{theorem}[\textbf{Subsequential convergence}]\label{Thm-subseqconv}
Suppose that Assumption \ref{assumA} holds, $\sum_{k=0}^{\infty}\varepsilon_k<+\infty$, and there exist $\underline{\gamma}, \,\overline{\gamma}>0$ such that $\frac{L_h}{2}<\underline{\gamma}\leq\overline{\gamma}$ and $\underline{\gamma}I_n \preceq H_k \preceq \overline{\gamma}I_n$ for all $k\geq0$. Let $\{\bm{x}^k\}$ be the sequence generated by the iVPGSA in Algorithm \ref{algo-iVPGSA}. Then, any accumulation point of $\{\bm{x}^k\}$ is a critical point of $F$.
\end{theorem}
\begin{proof}
By Proposition \ref{prop-iVPGSA}(ii), $\{\bm{x}^k\}$ is bounded. Hence, there exists at least one accumulation point. Let $\bm{x}^*$ be an accumulation point of $\{\bm{x}^k\}$ and $\{\bm{x}^{k_j}\}$ be a subsequence such that $\lim\limits_{j\to \infty}{{\bm{x}}^{k_j}}=\bm{x}^*$.  Then, $\bm{x}^*\in \mathrm{dom}\,F$ and $F(\bm{x}^*)=c^*$ based on Proposition \ref{prop-iVPGSA}(iv). Since $g$ is a continuous convex function and $\{\bm{x}^{k_j}\}$ is bounded, we know from \citep[Theorem 23.4]{r1970convex} that $\{\bm{y}^{k_j}\}$ is bounded. Without loss of generality, let $\lim\limits_{j\rightarrow \infty}{\bm{y}^{k_j}}  = \bm{y}^*\in\partial g(\bm{x}^*)$ due to the closedness of operator $\partial g$. Recalling condition \eqref{inexcondH-x}, we have that
\begin{equation}\label{optsubseq-VM}
\Delta^{k_j} \in \partial_{\delta_{k_j}} f(\bm{x}^{k_j+1}) + (\nabla h(\bm{x}^{k_j})-c_{k_j} \bm{y}^{k_j}) + H_{k_j} (\bm{x}^{k_j+1}-\bm{x}^{k_j}).
\end{equation}
Then, passing to the limit in \eqref{optsubseq-VM} along $\{\bm{x}^{k_j}\}$ and invoking $\Delta^{k_j}\to0$, $\delta_{k_j}\to0$, $\lim\limits_{k\to\infty}{c_k}=c^*$ (by Proposition \ref{prop-iVPGSA}(i)), the outer semicontinuity of $\partial_{\delta_{k_j}}{f}$ (see, e.g.,
\citep[Proposition 4.1.1]{hl1993convex}) with $\delta_{k_j}\to0$, $\bm{x}^{k_j+1}-\bm{x}^{k_j}\to0$ (by Proposition \ref{prop-iVPGSA}(iii)), the Lipschitz continuity of $\nabla h$, and $\underline{\gamma}I_n \preceq H_k \preceq \overline{\gamma}I_n$ (with some $\frac{L_h}{2}<\underline{\gamma}\leq\overline{\gamma}$) for all $k\geq0$, we see that
\begin{equation*}
0 \in \partial f(\bm{x}^*) + \nabla h(\bm{x}^*) - c^* \bm{y}^*,
\end{equation*}
which implies that $\bm{x}^*$ is a critical point of $F$ and completes the proof.
\end{proof}

\section{The KL-based convergence analysis}\label{sec-conv}

In this section, under the Kurdyka-{\L}ojasiewicz (KL) property and its associated exponent, together with additional suitable conditions, we study the global convergence of the entire sequence $\{\bm{x}^k\}$ generated by our iVPGSA and establish its convergence rate. While the KL property has become a fundamental tool in analyzing the global sequential convergence of nonconvex algorithms, standard analyses rely on a strict sufficient descent property to ensure that the objective function decreases by a sufficient amount at each iteration. This property plays a crucial role in applying the KL inequality in Definition \ref{defKLfun} or \ref{defKLexpo}; see, for example, \citep{ab2009on,abrs2010proximal,abs2013convergence,bst2014proximal,fgp2015splitting}. However, for many inexact algorithms, including our iVPGSA and those studied in \citep{hl2023convergence,sun2021sequence}, this sufficient descent property may fail because of the errors introduced by inexact subproblem solutions. Instead, these inexact algorithms only satisfy an \textit{approximate} sufficient descent property, for example, as characterized in Proposition \ref{prop-obj}(i), where the decrease in the objective function is only guaranteed up to a certain error term. The presence of such errors introduces additional challenges in global convergence analysis, particularly in deriving the convergence rate.

The KL-based convergence analysis for inexact or stochastic algorithms under such \textit{approximate} sufficient-descent-type conditions has been explored only in several recent works \citep{hl2023convergence,lmq2023convergence,qmlm2024kl,sun2021sequence}. Specifically, Sun \citep{sun2021sequence} developed an auxiliary-function-based framework for establishing global sequential convergence, but did not address convergence rates. Such rate analysis is more challenging because it requires a quantitative characterization of the interaction between the KL geometry and the error terms. Subsequent studies \citep{hl2023convergence,lmq2023convergence,qmlm2024kl} established global sequential convergence and convergence-rate estimates by employing \textit{strengthened} variants of the KL property and its associated exponent, in which the function-value differences in Definitions \ref{defKLfun} and \ref{defKLexpo} are replaced by their absolute values, respectively. These works show that, under suitable assumptions on the errors, desirable convergence properties can still be obtained despite the absence of a strict sufficient descent property, although the rate analyses rely on strengthened KL inequalities.

In the analysis below, we build on these insights and introduce the auxiliary function $\Phi_{\tau}$ in \eqref{defpofun} to absorb the accumulated errors caused by inexact subproblem solutions. This converts the approximate descent relation \eqref{appdeFval} for the original objective function $F$ into the exact descent relation \eqref{Hvalineq} for $\Phi_{\tau}$. Based on this construction, we establish the global convergence of the entire sequence $\{\bm{x}^k\}$ in Section \ref{sec-gloconv} and derive convergence-rate estimates in Section \ref{sec-convrate} by relying only on the classical KL property and its associated exponent in Definitions \ref{defKLfun} and \ref{defKLexpo}, rather than on the strengthened variants used in \citep{hl2023convergence,lmq2023convergence,qmlm2024kl}. Consequently, our analysis provides a new KL-based route for deriving convergence rates under the approximate sufficient descent property and may also offer useful insights for studying other inexact methods with similar perturbed descent relations.

We now start the KL-based convergence analysis by introducing additional assumptions.

\begin{assumption}\label{assumB}
The function $f$ is locally Lipschitz continuous on $\mathrm{dom}\,f$, i.e., for any $\bm{x}\in\mathrm{dom}\,f$, there exists a neighborhood $\mathcal{O}$ of $\bm{x}$ and a constant $L_x>0$ such that $|f(\bm{u})-f(\bm{v})|\leq L_x\|\bm{u}-\bm{v}\|$ holds for any $\bm{u},\,\bm{v}\in\mathcal{O}\cap\mathrm{dom}\,f$.
\end{assumption}

\begin{assumption}\label{assumC}
The function $g$ is continuously differentiable on $\Omega$ with a locally Lipschitz continuous gradient.
\end{assumption}

\begin{assumption}\label{assumD}
Assume that there exists an integer $K_0>0$ such that $\delta_{k}\equiv0$ for all $k \geq K_0$ in \eqref{inexcondH-x}.
\end{assumption}

Assumptions \ref{assumB} and \ref{assumC} are common technical assumptions used in the global convergence analysis of proximal algorithms for solving fractional optimization problems; see, for example, \citep{bdl2022extrapolated,lszz2022proximal,zl2022first}. Note that, since $f$ is convex in our setting, it follows from \citep[Theorem 10.4]{r1970convex} that Assumption \ref{assumB} is satisfied when $\mathrm{dom}\,f=\mathbb{R}^n$. Additionally, under Assumption \ref{assumD}, no error $\delta_k$ is allowed in the computation of $\partial f$ after finitely many iterations, while a non-zero error term $\Delta^k$ remains permissible in \eqref{inexcondH-x}. Consequently, even with this assumption, Algorithm \ref{algo-iVPGSA} retains its inexact nature. While Assumption \ref{assumD} looks restrictive, it is indeed achievable in, for example, the $\ell_1/\ell_2$ Lasso problem in Section \ref{sec-num-l12lasso}. Nevertheless, Assumption \ref{assumD} could also limit the applicability of our iVPGSA when addressing a problem with complex constraints, for example, the constrained $\ell_1/\ell_2$ sparse optimization problem in Section \ref{sec-num-l12cso}. Investigating the possibility of eliminating this requirement when establishing the convergence of the whole sequence remains an interesting direction for our future research.

We also summarize key results concerning the objective function $F$ and its subdifferential as follows, which will be useful in the subsequent analysis.

\begin{proposition}\label{prop-obj}
Suppose that Assumption \ref{assumA} holds and there exist $\underline{\gamma}, \,\overline{\gamma}>0$ such that $\frac{L_h}{2}<\underline{\gamma}\leq\overline{\gamma}$ and $\underline{\gamma}I_n \preceq H_k \preceq \overline{\gamma}I_n$ for all $k\geq0$. Let $\{\bm{x}^k\}$ be the sequence generated by the iVPGSA in Algorithm \ref{algo-iVPGSA}. If $\sum_{k=0}^{\infty}\varepsilon_k<+\infty$, then the following statements hold.
\begin{itemize}
\item[{\rm (i)}] \textbf{(Approximate sufficient descent property)} There exists a positive constant $a>0$ such that
    \begin{equation}\label{appdeFval}
    F(\bm{x}^{k+1}) \leq F(\bm{x}^{k})
    - \frac{a}{2}\|\bm{x}^{k+1}-\bm{x}^k\|^2
    + \varepsilon_k, \quad \forall\,k\geq0.
    \end{equation}

\item[{\rm(ii)}] If, in addition, Assumptions \ref{assumB}, \ref{assumC}, \ref{assumD} hold, there exist a positive constant $b>0$ and $\bm{w}^k\in\partial F(\bm{x}^k)$ such that
    \begin{equation}\label{subdiffbd}
    \|\bm{w}^k\| \leq b\left(\|{\bm{x}}^{k}-{\bm{x}}^{k-1}\|
    + \sqrt{\varepsilon_{k-1}}\right), \quad \forall\,k \geq K_0+1,
    \end{equation}
    where $K_0$ is given in Assumption \ref{assumD}.
\end{itemize}
\end{proposition}
\begin{proof}
\textit{Statement (i)}. The desired result readily follows from \eqref{decreaseFVMeps}, $H_k\succeq\underline{\gamma}I_n$ with $\underline{\gamma}>\frac{L_h}{2}$ for all $k\geq0$, and Proposition \ref{prop-iVPGSA}(v).

\textit{Statement (ii)}. We first recall from Proposition \ref{prop-iVPGSA}(ii) that $\{\bm{x}^k\}$ is bounded. This, together with Assumptions \ref{assumB} and \ref{assumC}, implies that $\nabla g$ and $F$ are globally Lipschitz continuous on $\mathcal{X}$, where $\mathcal{X}\subseteq\mathrm{dom}\,F$ is a closed bounded set that contains $\{\bm{x}^k\}$. Let $\ell_g$ and $\ell_F$ be the Lipschitz constants of $\nabla g$ and $F$ on $\mathcal{X}$, respectively.

Next, from condition \eqref{inexcondH-x} and Assumptions \ref{assumC}, \ref{assumD}, we have that, for any $k \geq K_0+1$,
\begin{equation*}
\Delta^{k-1} \in \partial f(\bm{x}^{k}) + \nabla h(\bm{x}^{k-1}) - c_{k-1}\nabla g(\bm{x}^{k-1}) + H_{k-1}(\bm{x}^{k}-\bm{x}^{k-1}).
\end{equation*}
This, together with $g(\bm{x}^k)\neq0$ (since $\bm{x}^k\in\Omega$) for all $k\geq0$, implies that
\begin{equation}\label{subdiffeq1}
-\frac{H_{k-1}(\bm{x}^{k}-\bm{x}^{k-1})}{g(\bm{x}^k)}
-\frac{\nabla h(\bm{x}^{k-1})}{g(\bm{x}^k)}
+\frac{c_{k-1}\nabla g(\bm{x}^{k-1})}{g(\bm{x}^k)}
+\frac{\Delta^{k-1}}{g(\bm{x}^k)}
\in\frac{\partial f(\bm{x}^{k})}{g(\bm{x}^k)}.
\end{equation}
On the other hand, note from the chain rule given in \citep[Proposition 2.2]{zl2022first} together with Assumptions \ref{assumB} and \ref{assumC} that $\widehat{\partial} F = \frac{g(\widehat{\partial}f+\nabla h)-(f+h)\nabla g}{g^2}$. Using this relation, \eqref{subdiffeq1} and $\partial f(\bm{x}^{k})=\widehat{\partial} f(\bm{x}^{k})$ (due to the convexity of $f$), one can verify that $\bm{w}^k\in\widehat{\partial} F(\bm{x}^k)$ with
\begin{equation*}
\bm{w}^k :=
- \frac{H_{k-1}(\bm{x}^{k}-\bm{x}^{k-1})}{g(\bm{x}^k)}
- \frac{\nabla h(\bm{x}^{k-1})-\nabla h(\bm{x}^{k})}{g(\bm{x}^k)}
+ \frac{c_{k-1}\nabla g(\bm{x}^{k-1})-c_{k}\nabla g(\bm{x}^{k})}{g(\bm{x}^k)}
+ \frac{\Delta^{k-1}}{g(\bm{x}^k)}.
\end{equation*}
This, together with $\underline{\gamma}I_n \preceq H_k \preceq \overline{\gamma}I_n$ for all $k\geq0$, Proposition \ref{prop-iVPGSA}(i)\&(v), the global Lipschitz continuity of $\nabla h$, $\nabla g$ and $F$ on $\mathcal{X}$, $\|\Delta^{k-1}\|\leq\sqrt{\varepsilon_{k-1}g(\bm{x}^{k})}$, and $\widehat{\partial} F(\bm{x}^k)\subseteq\partial F(\bm{x}^k)$ (see
\citep[Theorem 8.6]{rw1998variational}), gives \eqref{subdiffbd} and completes the proof.
\end{proof}

\subsection{Global sequential convergence}\label{sec-gloconv}

Following Sun \citep{sun2021sequence}, our analysis uses the following auxiliary function with some $\tau>1$:
\begin{equation}\label{defpofun}
\Phi_{\tau}(\bm{x},\,t) := F(\bm{x}) + \frac{t^{\tau}}{\tau}, \quad\forall\,(\bm{x},\,t)\in\mathrm{dom}\,F \times \mathbb{R}_+,
\end{equation}
where $F$ is defined in \eqref{defFun}. Moreover, we define the accumulation of errors as
\begin{equation}\label{defaccerr}
s_k := \tau\cdot\sum^{\infty}_{i=k}\varepsilon_i, \quad \forall\,k\geq0.
\end{equation}
Clearly, $\{s_k\}$ is non-increasing and $s_k\to0$ if $\sum_{k=0}^{\infty}\varepsilon_k<+\infty$. With all the above preparations, we are now ready to establish the global sequential convergence of our iVPGSA. It is worth noting that the summable error condition $\sum_{k=0}^{\infty}\varepsilon_k<+\infty$ in \eqref{epssumcond} is redundant, as it automatically follows from $\sum_{k=0}^{\infty}\sqrt{\varepsilon_k}<+\infty$. However, for ease of future reference, we retain it in \eqref{epssumcond}.

\begin{theorem}\label{thm-gloconv}
Suppose that Assumptions \ref{assumA}, \ref{assumB}, \ref{assumC}, \ref{assumD} hold and there exist $\underline{\gamma}, \,\overline{\gamma}>0$ such that $\frac{L_h}{2}<\underline{\gamma}\leq\overline{\gamma}$ and $\underline{\gamma}I_n \preceq H_k \preceq \overline{\gamma}I_n$ for all $k\geq0$. Let $\{\bm{x}^k\}$ be the sequence generated by the iVPGSA in Algorithm \ref{algo-iVPGSA} and let $\tau>1$. If $\Phi_{\tau}$ defined in \eqref{defpofun} is a KL function and
\begin{equation}\label{epssumcond}
\sum_{k=0}^{\infty}\varepsilon_k<+\infty, \quad
\sum_{k=0}^{\infty}\sqrt{\varepsilon_k}<+\infty, \quad
\sum_{k=0}^{\infty}\left(\sum^{\infty}_{i=k}\varepsilon_i\right)^{1-\frac{1}{\tau}}
< +\infty,
\end{equation}
then $\{\bm{x}^k\}$ converges to a critical point of $F$.
\end{theorem}
\begin{proof}
In view of Theorem \ref{Thm-subseqconv}, we only need to show that $\{\bm{x}^k\}$ is convergent. We start by noting from Proposition \ref{prop-obj}(i) that
\begin{equation*}
F(\bm{x}^{k+1}) + \sum^{\infty}_{i=k+1}\varepsilon_i
\leq F(\bm{x}^{k}) + \sum^{\infty}_{i=k}\varepsilon_i
- \frac{a}{2}\|\bm{x}^{k+1}-\bm{x}^k\|^2, \quad \forall\,k\geq0,
\end{equation*}
which, by using the definitions of $\Phi_{\tau}$ and $s_k$ in \eqref{defpofun} and \eqref{defaccerr}, respectively, can be expressed as
\begin{equation}\label{Hvalineq}
\Phi_{\tau}\big(\bm{x}^{k+1}, \,\sqrt[\tau]{s_{k+1}}\big)
\leq \Phi_{\tau}\big(\bm{x}^{k}, \,\sqrt[\tau]{s_{k}}\big)
- \frac{a}{2}\|\bm{x}^{k+1}-\bm{x}^k\|^2, \quad \forall\,k\geq0.
\end{equation}
This implies that $\left\{\Phi_{\tau}\big(\bm{x}^{k}, \,\sqrt[\tau]{s_{k}}\big)\right\}$ is non-increasing. Note from Proposition \ref{prop-iVPGSA}(i) that $c^*:=\lim\limits_{k\rightarrow\infty}{c_k}= \lim\limits_{k\rightarrow\infty}{F(\bm{x}^k)}$ exists. This then together with $s_k\to0$ implies that  $\left\{\Phi_{\tau}\big(\bm{x}^{k}, \,\sqrt[\tau]{s_{k}}\big)\right\}$ is also convergent and converges to $c^*$. In the following, we will consider two cases.

\vspace{1mm}
\textit{Case 1}. Suppose first that $\Phi_{\tau}\big(\bm{x}^{\bar{k}}, \,\sqrt[\tau]{s_{\bar{k}}}\big)=c^*$ for some $\bar{k}\geq0$. Since $\left\{\Phi_{\tau}\big(\bm{x}^{k}, \sqrt[\tau]{s_{k}}\big)\right\}$ is non-increasing, we must have that $\Phi_{\tau}\big(\bm{x}^{k}, \,\sqrt[\tau]{s_{k}}\big)\equiv c^*$ for all $k\geq\bar{k}$. This together with \eqref{Hvalineq} implies that $\bm{x}^{\bar{k}+t}=\bm{x}^{\bar{k}}$ for all $t\geq0$, and hence the sequence $\{\bm{x}^k\}$ converges finitely.

\vspace{1mm}
\textit{Case 2}. We consider the case where $\Phi_{\tau}\big(\bm{x}^{k}, \,\sqrt[\tau]{s_{k}}\big)>c^*$ for all $k\geq0$. In this case, we will divide the proof into three steps: (1) we first prove that $\Phi_{\tau}$ is constant on the set of cluster points of the sequence $\left\{\left(\bm{x}^k,\sqrt[\tau]{s_{k}}\right)\right\}$ and then apply the uniformized KL property; (2) we bound the distance from 0 to $\partial \Phi_{\tau}\big(\bm{x}^k,\sqrt[\tau]{s_{k}}\big)$; (3) we show that the sequence $\{\bm{x}^k\}$ is a Cauchy sequence and hence is convergent. The complete proof is presented as follows.

\textit{Step 1}. Since $\{s_k\}$ is non-increasing and $s_k\to0$, we see that the set of accumulation points of  $\left\{\left(\bm{x}^k,\sqrt[\tau]{s_{k}}\right)\right\}$ can be described as $\Gamma:=\big\{(\bm{x},0):\bm{x}\in\Lambda\big\}$ with $\Lambda$ being the set of accumulation points of $\{\bm{x}^k\}$. Since $\{\bm{x}^k\}$ is bounded as shown in Proposition \ref{prop-iVPGSA}(ii) and $s_k\to0$, the set $\Gamma$ is nonempty and compact. Moreover, for any $(\bm{x}^*,0) \in \Gamma$ with $\bm{x}^*\in\Lambda$, we have that $\Phi_{\tau}(\bm{x}^*,0)=F(\bm{x}^*)$. Thus, we can conclude from Proposition \ref{prop-iVPGSA}(iv) that $\Phi_{\tau}\equiv c^*$ on $\Gamma$. This fact together with our assumption that $\Phi_{\tau}$ is a KL function and Proposition \ref{uniKL} implies that, there exist $\mu, \eta>0$, and $\varphi\in\Xi_\eta$ such that
\begin{equation*}
\varphi'\left(\Phi_{\tau}(\bm{x},t)-c^*\right)\mathrm{dist}\left(0,\,\partial \Phi_{\tau}(\bm{x},t)\right)\geq1,
\end{equation*}
for all $(\bm{x},t)$ satisfying $\mathrm{dist}((\bm{x},t),\,\Gamma)<\mu$ and $c^*<\Phi_{\tau}(\bm{x},t)<c^*+\eta$. On the other hand, since $\lim\limits_{k\to\infty}\mathrm{dist}\big((\bm{x}^k,\sqrt[\tau]{s_{k}}),\,\Gamma\big)=0$ and $\left\{\Phi_{\tau}\big(\bm{x}^{k}, \sqrt[\tau]{s_{k}}\big)\right\}$ is non-increasing and converges to $c^*$, then for such $\mu$ and $\eta$, there exists $K_1\geq0$ such that $\mathrm{dist}\left((\bm{x}^k,\sqrt[\tau]{s_{k}}),\,\Gamma\right)<\mu$ and $c^*<\Phi_{\tau}\big(\bm{x}^{k}, \sqrt[\tau]{s_{k}}\big)<c^*+\eta$ for all $k\geq K_1$. Thus, we have
\begin{equation}\label{KLIC2}
\varphi'\big(\Phi_{\tau}\big(\bm{x}^{k}, \sqrt[\tau]{s_{k}}\big)-c^*\big)\mathrm{dist}\big(0,\, \partial \Phi_{\tau}\big(\bm{x}^{k}, \sqrt[\tau]{s_{k}}\big)\big)\geq1, \quad \forall\,k \geq K_1.
\end{equation}

\textit{Step 2}. We next consider the subdifferential of $\Phi_{\tau}$ at $(\bm{x}^k, \sqrt[\tau]{s_{k}})$ for any $k\geq \max\{K_0,K_1\}+1$, where $K_0$ is given in Assumption \ref{assumD}. Note from Proposition \ref{prop-obj}(ii) that, there exist a constant $b>0$ and $\bm{w}^k\in\partial F(\bm{x}^k)$ such that $\|\bm{w}^k\| \leq b\left(\|{\bm{x}}^{k}-{\bm{x}}^{k-1}\| + \sqrt{\varepsilon_{k-1}}\right)$ for any $k \geq K_0+1$. Then, for any $k \geq \max\{K_0,K_1\}+1$,
\begin{equation*}
\partial \Phi_{\tau}(\bm{x}^k,\sqrt[\tau]{s_{k}})
=\left[\partial F(\bm{x}^k),
\,s_k^{1-\frac{1}{\tau}} \right]
\ni\left[\bm{w}^k, \,s_k^{1-\frac{1}{\tau}}\right],
\end{equation*}
which implies that, for any $k \geq \max\{K_0,K_1\}+1$,
\begin{equation}\label{subdiffbdIC2}
\mathrm{dist}\big(0,\,\partial \Phi_{\tau}(\bm{x}^k,\sqrt[\tau]{s_{k}})\big)
\leq \|\bm{w}^k\| + s_k^{1-\frac{1}{\tau}}
\leq b\left(\|{\bm{x}}^{k}-{\bm{x}}^{k-1}\|
+ \sqrt{\varepsilon_{k-1}}\right)
+ s_k^{1-\frac{1}{\tau}}.
\end{equation}

\textit{Step 3}. We are now ready to prove the sequential convergence. For notational simplicity, let $K_2:=\max\{K_0,K_1\}+1$ and
\begin{equation*}
\Delta_{\varphi}^k
:=\varphi\left(\Phi_{\tau}\big(\bm{x}^{k}, \sqrt[\tau]{s_{k}}\big)-c^*\right)
-\varphi\left(\Phi_{\tau}\big(\bm{x}^{k+1}, \sqrt[\tau]{s_{k+1}}\big)-c^*\right).
\end{equation*}
Then, we have that, for any $k\geq K_2$,
\begin{equation*}
\begin{aligned}
&\quad \Delta_{\varphi}^k\cdot\mathrm{dist}\big(0,\,\partial \Phi_{\tau}(\bm{x}^k,\sqrt[\tau]{s_{k}})\big)  \\
&\geq \varphi'\big(\Phi_{\tau}\big(\bm{x}^{k}, \sqrt[\tau]{s_{k}}\big)-c^*\big)
\,\mathrm{dist}\big(0,\,\partial \Phi_{\tau}(\bm{x}^k,\sqrt[\tau]{s_{k}})\big)
\cdot\left(\Phi_{\tau}\big(\bm{x}^{k}, \sqrt[\tau]{s_{k}}\big)
-\Phi_{\tau}\big(\bm{x}^{k+1}, \sqrt[\tau]{s_{k+1}}\big)\right)  \\
&\geq \Phi_{\tau}\big(\bm{x}^{k}, \sqrt[\tau]{s_{k}}\big)
-\Phi_{\tau}\big(\bm{x}^{k+1}, \sqrt[\tau]{s_{k+1}}\big)
\geq \frac{a}{2}\|\bm{x}^{k+1} -\bm{x}^{k}\|^2,
\end{aligned}
\end{equation*}
where the first inequality follows from the concavity of $\varphi$, the second inequality follows from \eqref{KLIC2} and the non-increasing property of $\{\Phi_{\tau}\big(\bm{x}^{k}, \sqrt[\tau]{s_{k}}\big)\}$, and the last inequality follows from \eqref{Hvalineq}. Combining the above inequality and \eqref{subdiffbdIC2}, we further obtain that,
\begin{equation}\label{fklIC2}
\|\bm{x}^{k+1} -\bm{x}^{k}\|^2
\leq \frac{2b}{a}\Delta_{\varphi}^k\cdot
\left(\|{\bm{x}}^{k}-{\bm{x}}^{k-1}\|
+ \sqrt{\varepsilon_{k-1}}
+ \frac{1}{b}s_k^{1-\frac{1}{\tau}}\right), \quad \forall\,k \geq K_2.
\end{equation}
Taking the square root of \eqref{fklIC2} and using the inequality $\sqrt{uv}\leq\frac{u+v}{2}$ for $u,v\geq0$, we see that
\begin{equation}\label{diffbdeps}
\begin{aligned}
\|\bm{x}^{k+1} -\bm{x}^{k}\|
&\leq \sqrt{\frac{2b}{a}\Delta_{\varphi}^k} \cdot
\sqrt{\|{\bm{x}}^{k}-{\bm{x}}^{k-1}\|
+ \sqrt{\varepsilon_{k-1}}
+ \frac{1}{b}s_k^{1-\frac{1}{\tau}}}  \\
&\leq \frac{b}{a}\Delta_{\varphi}^k
+ \frac{1}{2}\|{\bm{x}}^{k}-{\bm{x}}^{k-1}\|
+ \frac{1}{2}\sqrt{\varepsilon_{k-1}}
+ \frac{1}{2b}s_k^{1-\frac{1}{\tau}},
\quad \forall\,k \geq K_2,
\end{aligned}
\end{equation}
which, by simple reformulations, implies that
\begin{equation*}
\begin{aligned}
\|\bm{x}^{k+1} -\bm{x}^{k}\|
\leq \frac{2b}{a}\Delta_{\varphi}^k
+ \|{\bm{x}}^{k}-{\bm{x}}^{k-1}\|-\|\bm{x}^{k+1} -\bm{x}^{k}\|
+ \sqrt{\varepsilon_{k-1}}
+ \frac{1}{b}s_k^{1-\frac{1}{\tau}}.
\end{aligned}
\end{equation*}
Then, summing the above relation from $k=K_2$ to $\infty$, we have
\begin{equation*}
\begin{aligned}
\sum_{k=K_2}^{\infty}\|\bm{x}^{k+1} -\bm{x}^{k}\|
\leq \frac{2b}{a}\,\varphi\left(\Phi_{\tau}\big(\bm{x}^{K_2}, \sqrt[\tau]{s_{K_2}}\big)-c^*\right)
+ \|\bm{x}^{K_2}-\bm{x}^{K_2-1}\|
+ \sum_{k=K_2}^{\infty}\sqrt{\varepsilon_{k-1}}
+ \frac{1}{b}\sum_{k=K_2}^{\infty}s_k^{1-\frac{1}{\tau}}.
\end{aligned}
\end{equation*}
Using this relation, together with the summable error conditions in \eqref{epssumcond}, we conclude that $\sum_{k=0}^{\infty}\|\bm{x}^{k+1}-\bm{x}^{k}\|<\infty$ and hence $\{\bm{x}^{k}\}$ is convergent. This completes the proof.
\end{proof}

\begin{remark}[\textbf{Comments on the tolerance parameter sequence $\{\varepsilon_k\}$}]\label{remark-epsk}
Note that the global sequential convergence in Theorem \ref{thm-gloconv} requires suitable summable error conditions on the tolerance parameter sequence $\{\varepsilon_k\}$, as given in \eqref{epssumcond}. While these conditions may seem restrictive, they can be met by appropriately selecting $\varepsilon_k$ using either an exponential or a polynomial function. Specifically, let $\tau>1$ and $\varepsilon_0>0$. The tolerance parameter sequence $\{\varepsilon_k\}$ satisfies the summable error conditions in \eqref{epssumcond} when it is specified using one of the following ways:
\begin{itemize}
\item[{\rm (i)}] $\varepsilon_k=\varepsilon_0 \zeta^k$ with $0<\zeta<1$;
\item[{\rm (ii)}] $\varepsilon_k=\frac{\varepsilon_0}{(k+1)^q}$ with $q>\frac{2\tau-1}{\tau-1}$.
\end{itemize}
Here, $\zeta$ denotes the geometric decay factor in the exponential case, while $q$ denotes the decay exponent in the polynomial case. The detailed verification is relegated to Appendix \ref{apd-remark-sumcond} for self-containedness. Moreover, if $\Phi_{\tau}$ is a KL function for any $\tau>1$, as established in Section \ref{sec-KLpofun}, then any $q>2$ is valid in the polynomial case.
\end{remark}

\subsection{Convergence rate analysis}\label{sec-convrate}

In this subsection, we analyze the local convergence rate of $\{\bm{x}^k\}$ under the assumption that the auxiliary function $\Phi_{\tau}$, for some $\tau>1$, is a KL function with an exponent $\theta_{\tau}\in[0,1)$ (see Definition \ref{defKLexpo}). Local convergence rate results of this type have been extensively studied for various algorithms; see, for example, \citep{ab2009on,fgp2015splitting,wcp2017a,y2024proximal}. While our analysis builds upon these existing insights, it introduces additional complexities due to the presence of errors in the iterations. Indeed, unlike exact algorithms, where the descent in the objective function is well-controlled, inexact algorithms like our iVPGSA involve perturbations at each iteration, making the KL-based convergence rate analysis more challenging. To overcome these difficulties, we carefully construct suitable recursive inequalities and regulate the error terms to ensure that they remain manageable throughout the iterations. By leveraging the auxiliary function $\Phi_{\tau}$, our analysis provides detailed insights into how the KL exponent $\theta_{\tau}$ influences the convergence rate in the presence of inexactness. 

To proceed, we first gather four technical lemmas that will be used in our convergence rate analysis. These lemmas pertain to two specific choices of the tolerance parameter sequence $\{\varepsilon_k\}$, as described in Remark \ref{remark-epsk}. Since their proofs follow directly from elementary arguments, we omit them for brevity.

\begin{lemma}\label{lemma-orderofseries}
Let $\tau>1$, $\varepsilon_0>0$ and $s_k := \tau\sum^{\infty}_{i=k}\varepsilon_i$.
\begin{itemize}
\item [{\rm (i)}] If $\varepsilon_k = \varepsilon_0 \zeta^k$ with $0<\zeta<1$, then
    \begin{equation*}
    s_k = l_1 \zeta^k,  \quad
    {\textstyle\sum^{\infty}_{i=k}\sqrt{\varepsilon_{i-1}} = l_2 \zeta^{\frac{k}{2}}},  \quad
    {\textstyle\sum^{\infty}_{i=k}s_i^{1-\frac{1}{\tau}} = l_3 \zeta^{\left(1-\frac{1}{\tau}\right)k}},
    \end{equation*}
    where $l_1:=\frac{\tau\varepsilon_0}{1-\zeta}$,
    $l_2:=\frac{\sqrt{\varepsilon_0}}{\sqrt{\zeta}\left(1-\sqrt{\zeta}\right)}$, and $l_3:= l_1^{1-\frac{1}{\tau}} \left(1-\zeta^{1-\frac{1}{\tau}}\right)^{-1}$.

\item [{\rm (ii)}] If $\varepsilon_k = \frac{\varepsilon_0}{(k+1)^q}$ with $q>\frac{2\tau-1}{\tau-1}$, then, for every $k\geq1$,
    \begin{equation*}
    s_k \leq \frac{e_1}{k^{q-1}},   \quad
    {\textstyle\sum^{\infty}_{i=k}}\sqrt{\varepsilon_{i-1}}
    \leq \frac{e_2}{k^{\frac{q}{2}-1}},  \quad
    {\textstyle\sum^{\infty}_{i=k}}s_i^{1-\frac{1}{\tau}} \leq  \frac{e_3}{k^{\left(1-\frac{1}{\tau}\right)(q-1)-1}},
    \end{equation*}
    where $e_1:=\frac{\tau\varepsilon_0}{q-1}$,  $e_2:=\frac{q\sqrt{\varepsilon_0}}{q-2}$, and
    $e_3:=e_1^{1-\frac{1}{\tau}}\frac{r}{r-1}$ with $r:=\left(1-\frac{1}{\tau}\right)(q-1)>1$.
\end{itemize}
\end{lemma}

\begin{lemma}\label{lem-closed-psi1}
Let $\tau>1$ and $q>\frac{2\tau-1}{\tau-1}$. Define $\Xi(\tau,q)
:=\frac{\tau-1}{\tau}(q-1)$ and
\begin{equation}\label{defpsi1}
\psi_1(q,\tau)
:=\min\left\{\,\frac{q}{2},~\Xi(\tau,q)\,\right\}.
\end{equation}
Then, $\Xi(\tau,q)>1$ and $\psi_1(q,\tau)$ has the following closed-form expression:
\begin{equation*}
\psi_1(q,\tau)
=\left\{\begin{aligned}
&\textstyle{\Xi(\tau,q)},
&&\text{if}\quad 1<\tau\leq2~~\text{and}~~\textstyle{q>\frac{2\tau-1}{\tau-1}}, \\[3pt]
&\textstyle{\Xi(\tau,q)},
&&\text{if}\quad\tau>2~~\text{and}~~\textstyle{\frac{2\tau-1}{\tau-1}<q<\frac{2(\tau-1)}{\tau-2}}, \\[3pt]
&\textstyle{\frac{q}{2}},
&&\text{if}\quad\tau>2~~\text{and}~~\textstyle{q\geq\frac{2(\tau-1)}{\tau-2}}.
\end{aligned}\right.
\end{equation*}
\end{lemma}

\begin{lemma}\label{lem-closed-psi2}
Let $\theta_{\tau}\in\left(\frac{1}{2}, 1\right)$, $\theta_{\tau}\leq\vartheta<1$, $\tau>1$ and $q>\frac{2\tau-1}{\tau-1}$. Define
\begin{equation}\label{defpsi2}
\psi_2(q,\tau,\vartheta)
:=\min\left\{
\,\frac{q}{2},~\Xi(\tau,q),
~\frac{\vartheta}{1-\vartheta}\left(\frac{q}{2}-1\right),
~\frac{\vartheta}{1-\vartheta}\,\big(\Xi(\tau,q)-1\big)
\,\right\}.
\end{equation}
where $\Xi(\tau,q):=\frac{\tau-1}{\tau}(q-1)$. Then, $\psi_2(q,\tau,\vartheta)$ has the following closed-form expression.
\begin{itemize}
\item[{\rm (i)}] If $1<\tau\leq2$ and $q>\frac{2\tau-1}{\tau-1}$; or $\tau>2$ and $\frac{2\tau-1}{\tau-1}<q<\frac{2(\tau-1)}{\tau-2}$, we have that
    \begin{equation*}
    \begin{aligned}
    \psi_2(q,\tau,\vartheta)
    &=\left\{\begin{aligned}
    &{\textstyle\frac{\vartheta}{1-\vartheta}\left(\Xi(\tau,q)-1\right)}, &&\text{if}\quad\theta_{\tau}\leq\vartheta
    <{\textstyle\frac{\Xi(\tau,q)}{2\Xi(\tau,q)-1}}, \\[3pt]
    &\Xi(\tau,q),
    &&\text{if}\quad{\textstyle\frac{\Xi(\tau,q)}{2\Xi(\tau,q)-1}} \leq\vartheta<1.
    \end{aligned}\right.
    \end{aligned}
    \end{equation*}


\item[{\rm (ii)}] If $\tau>2$ and $q\geq\frac{2(\tau-1)}{\tau-2}$, we have that
    \begin{equation*}
    \psi_2(q,\tau,\vartheta)
    =\left\{\begin{aligned}
    &{\textstyle\frac{\vartheta}{1-\vartheta}\left(\frac{q}{2}-1\right)}, &&\text{if}\quad\theta_{\tau}\leq\vartheta<{\textstyle\frac{q}{2q-2}}, \\[3pt]
    &{\textstyle\frac{q}{2}},
    &&\text{if}\quad{\textstyle\frac{q}{2q-2}}\leq\vartheta<1.
    \end{aligned}\right.
    \end{equation*}
\end{itemize}
\end{lemma}

\begin{lemma}\label{lem-closed-psi3}
Let $\theta_{\tau}\in\left(\frac{1}{2}, 1\right)$, $\theta_{\tau}\leq\vartheta<1$, $\tau>1$ and $q>\frac{2\tau-1}{\tau-1}$. Define
\begin{equation}\label{defpsi3}
\begin{aligned}
\psi_3(q,\tau,\vartheta)
&:=\min\left\{\psi_2(q,\tau,\vartheta),
~\frac{\vartheta}{2\vartheta-1}\right\}, \\
\psi_3^{(q,\tau,*)}
&:=\max\limits_{\vartheta}\big\{\psi_3(q,\tau,\vartheta):\theta_{\tau}\leq\vartheta<1\big\}, \\
\vartheta^{(q,\tau,*)}
&:=\arg\max\limits_{\vartheta}\big\{\psi_3(q,\tau,\vartheta):\theta_{\tau}\leq\vartheta<1\big\}
\end{aligned}
\end{equation}
Then, $\psi_3(q,\tau,\vartheta)$, $\psi_3^{(q,\tau,*)}$, and $\vartheta^{(q,\tau,*)}$ have the following closed-form expressions.
\begin{itemize}
\item[{\rm (i)}] If $1<\tau\leq2$ and $q>\frac{2\tau-1}{\tau-1}$; or $\tau>2$ and $\frac{2\tau-1}{\tau-1}<q<\frac{2(\tau-1)}{\tau-2}$, we have that
    \begin{equation*}
    \begin{aligned}
    &\psi_3(q,\tau,\vartheta)
    =\left\{\begin{aligned}
    &{\textstyle\frac{\vartheta}{1-\vartheta}\left(\Xi(\tau,q)-1\right)}, &&\text{if}\quad\theta_{\tau}\leq\vartheta
    <{\textstyle\frac{\Xi(\tau,q)}{2\Xi(\tau,q)-1}}, \\[3pt]
    &{\textstyle\frac{\vartheta}{2\vartheta-1}},
    &&\text{if}\quad{\textstyle\frac{\Xi(\tau,q)}{2\Xi(\tau,q)-1}} \leq\vartheta<1,
    \end{aligned}\right. \\[5pt]
    &\psi_3^{(q,\tau,*)}
    =\left\{\begin{aligned}
    &\Xi(\tau,q),
    &&\text{if}\quad q<{\textstyle\frac{\tau}{\tau-1}\cdot\frac{\theta_{\tau}}{2\theta_{\tau}-1}+1}, \\
    &{\textstyle\frac{\theta_{\tau}}{2\theta_{\tau}-1}},
    &&\text{otherwise},
    \end{aligned}\right.
    ~~\vartheta^{(q,\tau,*)}
    =\left\{\begin{aligned}
    &{\textstyle\frac{\Xi(\tau,q)}{2\Xi(\tau,q)-1}},
    &&\text{if}\quad q<{\textstyle\frac{\tau}{\tau-1}\cdot\frac{\theta_{\tau}}{2\theta_{\tau}-1}+1}, \\[3pt]
    &\theta_{\tau},
    &&\text{otherwise}.
    \end{aligned}\right.
    \end{aligned}
    \end{equation*}

\item[{\rm (ii)}] If $\tau>2$ and $q\geq\frac{2(\tau-1)}{\tau-2}$, we have that
    \begin{equation*}
    \begin{aligned}
    &\psi_3(q,\tau,\vartheta)
    =\left\{\begin{aligned}
    &{\textstyle\frac{\vartheta}{1-\vartheta}\left(\frac{q}{2}-1\right)},
    &&\text{if}\quad\theta_{\tau}\leq\vartheta<{\textstyle\frac{q}{2q-2}},  \\[3pt]
    &{\textstyle\frac{\vartheta}{2\vartheta-1}},
    &&\text{if}\quad{\textstyle\frac{q}{2q-2}}\leq\vartheta<1,
    \end{aligned}\right.  \\[5pt]
    &\psi_3^{(q,\tau,*)}
    =\left\{\begin{aligned}
    &{\textstyle\frac{q}{2}},
    &&\text{if}\quad q<{\textstyle\frac{2\theta_{\tau}}{2\theta_{\tau}-1}}, \\
    &{\textstyle\frac{\theta_{\tau}}{2\theta_{\tau}-1}},
    &&\text{otherwise},
    \end{aligned}\right.
    ~~\vartheta^{(q,\tau,*)}
    =\left\{\begin{aligned}
    &{\textstyle\frac{q}{2q-2}},
    &&\text{if}\quad q<{\textstyle\frac{2\theta_{\tau}}{2\theta_{\tau}-1}}, \\[3pt]
    &\theta_{\tau},
    &&\text{otherwise}.
    \end{aligned}\right.
    \end{aligned}
    \end{equation*}

\end{itemize}
\end{lemma}

We further derive key recursive inequalities under two specific choices of the tolerance parameter sequence $\{\varepsilon_k\}$, as described in Remark \ref{remark-epsk}. These recursions play a crucial role in our subsequent convergence rate analysis.

\begin{lemma}\label{lemma-midkeyineq}
Suppose that Assumptions \ref{assumA}, \ref{assumB}, \ref{assumC}, \ref{assumD} hold, there exist $\underline{\gamma}, \,\overline{\gamma}>0$ such that $\frac{L_h}{2}<\underline{\gamma}\leq\overline{\gamma}$ and $\underline{\gamma}I_n \preceq H_k \preceq \overline{\gamma}I_n$ for all $k\geq0$, and the summable error conditions in \eqref{epssumcond} hold. Moreover, let $\tau>1$ and suppose that $\Phi_{\tau}$ defined in \eqref{defpofun} is a KL function with an exponent $\theta_{\tau}$. Let $\{\bm{x}^k\}$ be the sequence generated by the iVPGSA in Algorithm \ref{algo-iVPGSA}, and let
\begin{equation*}
\Upsilon_k:={\textstyle\sum^{\infty}_{i=k}}\,\|\bm{x}^{i+1}-\bm{x}^i\|,
\end{equation*}
which is well-defined due to Theorem \ref{thm-gloconv}. Then, for all sufficiently large $k$, the following statements hold.
\begin{itemize}
\item[{\rm (i)}] If $\theta_{\tau}\in\left(0,\frac{1}{2}\right]$, there exist $\rho\in(0,1)$ and $\alpha_1,\alpha_2>0$ such that
    \begin{equation}\label{KLineqrate2case1}
    \Upsilon_k \leq \rho \Upsilon_{k-1}
    + \alpha_1{\textstyle\sum^{\infty}_{i=k}}\sqrt{\varepsilon_{i-1}}
    + \alpha_2{\textstyle\sum^{\infty}_{i=k}}\,s_i^{1-\frac{1}{\tau}}.
    \end{equation}

\item[{\rm (ii)}] If $\theta_{\tau}\in\left(\frac{1}{2},1\right)$, then for any $\theta_{\tau}\leq\vartheta<1$, there exist $\beta_1,\beta_2,\beta_3,\beta_4,\beta_5>0$ such that
    \begin{equation}\label{KLineqrate2case2}
    \begin{aligned}
    \Upsilon_k^{\frac{\vartheta}{1-\vartheta}}
    &\leq \beta_1\left(\Upsilon_{k-1}-\Upsilon_k\right)
    + \beta_2 \sqrt{\varepsilon_{k-1}}
    + \beta_3 s_k^{1-\frac{1}{\tau}}  \\
    &\qquad + \beta_4 \Big({\textstyle\sum^{\infty}_{i=k}}\sqrt{\varepsilon_{i-1}} \Big)^{\frac{\vartheta}{1-\vartheta}}
    + \beta_5 \Big({\textstyle\sum^{\infty}_{i=k}}\,s_i^{1-\frac{1}{\tau}} \Big)^{\frac{\vartheta}{1-\vartheta}}.
    \end{aligned}
    \end{equation}
\end{itemize}
\end{lemma}
\begin{proof}
First, we know from \eqref{Hvalineq} and the subsequent arguments that $\{\Phi_{\tau}\big(\bm{x}^{k}, \sqrt[\tau]{s_{k}}\big)\}$ is non-increasing and converges to $c^*$. Thus, if $\Phi_{\tau}\big(\bm{x}^{\bar{k}}, \sqrt[\tau]{s_{\bar{k}}}\big)=c^*$ holds for some $\bar{k}\geq0$, $\{\bm{x}^k\}$ is finitely convergent and the desired conclusions hold trivially. In the following, we only need to consider the case when $\Phi_{\tau}\big(\bm{x}^{k}, \,\sqrt[\tau]{s_{k}}\big)>c^*$ for all $k\geq0$. Using \eqref{diffbdeps}, we have that, for any $k \geq K_2$,
\begin{equation}\label{Upineq1}
\begin{aligned}
\Upsilon_k
&\leq
{\textstyle\frac{1}{2}\sum^{\infty}_{i=k}\|\bm{x}^{i}-\bm{x}^{i-1}\|
+ \frac{b}{a}\sum^{\infty}_{i=k}\Delta_{\varphi}^i
+ \frac{1}{2}\sum^{\infty}_{i=k}\sqrt{\varepsilon_{i-1}}
+ \frac{1}{2b}\sum^{\infty}_{i=k}s_i^{1-\frac{1}{\tau}}} \\
&\leq
{\textstyle\frac{1}{2}\Upsilon_{k-1} + \frac{b}{a}\varphi(U_k)
+ \frac{1}{2}\sum^{\infty}_{i=k}\sqrt{\varepsilon_{i-1}}
+ \frac{1}{2b}\sum^{\infty}_{i=k}s_i^{1-\frac{1}{\tau}}},
\end{aligned}
\end{equation}
where $U_k:=\Phi_{\tau}\big(\bm{x}^{k}, \,\sqrt[\tau]{s_{k}}\big)-c^*$. On the other hand, it follows from the KL inequality \eqref{KLIC2} with $\varphi(t)=\tilde{a}t^{1-\theta_{\tau}}$ for some $\tilde{a}>0$ that, for all $k \geq K_1$,
\begin{equation*}
\tilde{a}(1-\theta_{\tau})(U_k)^{-\theta_{\tau}}\mathrm{dist}\big(0,\, \partial \Phi_{\tau}\big(\bm{x}^{k}, \sqrt[\tau]{s_{k}}\big)\big)\geq1.
\end{equation*}
Moreover, using \eqref{subdiffbdIC2} and the definition of $\Upsilon_k$, we see that for all $k \geq \max\{K_0,K_1\}+1$,
\begin{equation*}
\mathrm{dist}\big(0,\,\partial \Phi_{\tau}(\bm{x}^k,\sqrt[\tau]{s_{k}})\big)
\leq b\big(\Upsilon_{k-1}-\Upsilon_k\big)
+ b\sqrt{\varepsilon_{k-1}}
+ s_k^{1-\frac{1}{\tau}}.
\end{equation*}
Then, combining the above two inequalities yields that, for all $k \geq \max\{K_0,K_1\}+1$,
\begin{equation*}
(U_k)^{\theta_{\tau}} \leq
\tilde{a}b(1-\theta_{\tau})\big(\Upsilon_{k-1}-\Upsilon_k\big)
+ \tilde{a}b(1-\theta_{\tau})\sqrt{\varepsilon_{k-1}}
+ \tilde{a}(1-\theta_{\tau})s_k^{1-\frac{1}{\tau}}.
\end{equation*}
Raising to a power of $\frac{1-\theta_{\tau}}{\theta_{\tau}}$ to both sides of the above inequality, scaling both sides by $\tilde{a}$ and recalling that $\varphi(U_k)=\tilde{a}(U_k)^{1-\theta_{\tau}}$, we obtain that
\begin{equation}\label{KLineqrate2}
\varphi(U_k)
\leq \tilde{a}\,\Big(
\tilde{a}b(1-\theta_{\tau})\big(\Upsilon_{k-1}-\Upsilon_k\big)
+ \tilde{a}b(1-\theta_{\tau})\sqrt{\varepsilon_{k-1}}
+ \tilde{a}(1-\theta_{\tau})s_k^{1-\frac{1}{\tau}}\Big)^{\frac{1-\theta_{\tau}}{\theta_{\tau}}}.
\end{equation}
Next, we consider two cases: $\theta_{\tau}\in\left(0,\frac{1}{2}\right]$ and $\theta_{\tau}\in\left(\frac{1}{2},1\right)$.

\textit{Case (i)}. Suppose that $\theta_{\tau}\in\left(0,\frac{1}{2}\right]$ and hence $\frac{1-\theta_{\tau}}{\theta_{\tau}}\geq1$. Since $\|\bm{x}^{k}-{\bm{x}}^{k-1}\|\to0$ (and hence $\Upsilon_{k-1}-\Upsilon_k\to0$), $\sqrt{\varepsilon_{k-1}}\to0$ and $ s_k^{1-\frac{1}{\tau}}\to0$, we see from \eqref{KLineqrate2} that, for all sufficiently large $k$,
\begin{equation*}
\begin{aligned}
\varphi(U_k)
&\leq
\tilde{a}^2b(1-\theta_{\tau})\big(\Upsilon_{k-1}-\Upsilon_k\big)
+ \tilde{a}^2b(1-\theta_{\tau})\sqrt{\varepsilon_{k-1}}
+ \tilde{a}^2(1-\theta_{\tau})s_k^{1-\frac{1}{\tau}}  \\
&\leq \tilde{a}^2b(1-\theta_{\tau})\big(\Upsilon_{k-1}-\Upsilon_k\big)
+ \tilde{a}^2b(1-\theta_{\tau}){\textstyle\sum^{\infty}_{i=k}}\sqrt{\varepsilon_{i-1}}
+ \tilde{a}^2(1-\theta_{\tau}){\textstyle\sum^{\infty}_{i=k}}\,s_i^{1-\frac{1}{\tau}}.
\end{aligned}
\end{equation*}
Combining this inequality with \eqref{Upineq1} induces that, for all sufficiently large $k$,
\begin{equation*}
\begin{aligned}
\Upsilon_k
&\leq{\textstyle\frac{1}{2}}\Upsilon_{k-1}
+ {\textstyle\frac{\tilde{a}^2b^2(1-\theta_{\tau})}{a}}
\big(\Upsilon_{k-1}-\Upsilon_k\big) \\
&\qquad + {\textstyle\left(\frac{\tilde{a}^2b^2(1-\theta_{\tau})}{a}+\frac{1}{2}\right)
\sum^{\infty}_{i=k}\sqrt{\varepsilon_{i-1}}
+ \left(\frac{\tilde{a}^2b(1-\theta_{\tau})}{a}+\frac{1}{2b}\right)
\sum^{\infty}_{i=k}s_i^{1-\frac{1}{\tau}}},
\end{aligned}
\end{equation*}
which, by simple reformulations, implies that
\begin{equation*}
\Upsilon_k \leq \rho \Upsilon_{k-1}
+ \alpha_1{\textstyle\sum^{\infty}_{i=k}}\sqrt{\varepsilon_{i-1}}
+ \alpha_2{\textstyle\sum^{\infty}_{i=k}}\,s_i^{1-\frac{1}{\tau}},
\end{equation*}
where $\alpha_0:=1+\frac{\tilde{a}^2b^2(1-\theta_{\tau})}{a}>0$,  $\alpha_1:=\alpha_0^{-1}\left(\frac{\tilde{a}^2b^2(1-\theta_{\tau})}{a}+\frac{1}{2}\right)>0$, $\alpha_2:=\alpha_0^{-1}\left(\frac{\tilde{a}^2b(1-\theta_{\tau})}{a}+\frac{1}{2b}\right)>0$, and $\rho:=\alpha_0^{-1}\left(\frac{1}{2}+\frac{\tilde{a}^2b^2(1-\theta_{\tau})}{a}\right)\in(0,1)$. This proves statement (i).

\textit{Case (ii)}. Suppose that $\theta_{\tau}\in\left(\frac{1}{2},1\right)$. In this case, it is easy to verify that $\frac{1-\vartheta}{\vartheta}\leq\frac{1-\theta_{\tau}}{\theta_{\tau}}<1$ for any $\theta_{\tau}\leq\vartheta<1$. Since $\|\bm{x}^{k}-{\bm{x}}^{k-1}\|\to0$ (and hence $\Upsilon_{k-1}-\Upsilon_k\to0$), $\sqrt{\varepsilon_{k-1}}\to0$ and $ s_k^{1-\frac{1}{\tau}}\to0$, one can see from \eqref{KLineqrate2} that, for all sufficiently large $k$, it holds that
\begin{equation*}
\begin{aligned}
\varphi(U_k)
&\leq\tilde{a}\,\Big(
\tilde{a}b\big(\Upsilon_{k-1}-\Upsilon_k\big)
+ \tilde{a}b\sqrt{\varepsilon_{k-1}}
+ \tilde{a}s_k^{1-\frac{1}{\tau}}\Big)^{\frac{1-\theta_{\tau}}{\theta_{\tau}}} \\
&\leq\tilde{a}\,\Big(
\tilde{a}b\big(\Upsilon_{k-1}-\Upsilon_k\big)
+ \tilde{a}b\sqrt{\varepsilon_{k-1}}
+ \tilde{a}s_k^{1-\frac{1}{\tau}}\Big)^{\frac{1-\vartheta}{\vartheta}} \\
&\leq \tilde{a}\left(
\tilde{a}b\big(\Upsilon_{k-1}-\Upsilon_k\big)\right)^{\frac{1-\vartheta}{\vartheta}}
+ \tilde{a}\left(\tilde{a}b\sqrt{\varepsilon_{k-1}}\right)^{\frac{1-\vartheta}{\vartheta}}
+ \tilde{a}\Big(\tilde{a}s_k^{1-\frac{1}{\tau}}\Big)^{\frac{1-\vartheta}{\vartheta}}
\end{aligned}
\end{equation*}
for any $\theta_{\tau}\leq\vartheta<1$. Combining this inequality with \eqref{Upineq1}, we obtain that, for all sufficiently large $k$,
\begin{equation*}
\begin{aligned}
\Upsilon_k
&\leq {\textstyle\frac{1}{2}\Upsilon_{k-1}
+ \frac{b}{a}\tilde{a}\left(\tilde{a}b\right)^{\frac{1-\vartheta}{\vartheta}}
\left(\Upsilon_{k-1}-\Upsilon_k\right)^{\frac{1-\vartheta}{\vartheta}}
+ \left(\frac{b}{a}\tilde{a}\left(\tilde{a}b\right)^{\frac{1-\vartheta}{\vartheta}}\right)   \left(\sqrt{\varepsilon_{k-1}}\right)^{\frac{1-\vartheta}{\vartheta}}} \\
&\qquad\qquad~
+ {\textstyle\left(\frac{b}{a}\tilde{a}^{\frac{1}{\vartheta}}\right)   \Big(s_k^{1-\frac{1}{\tau}}\Big)^{\frac{1-\vartheta}{\vartheta}}
+ \frac{1}{2}\sum^{\infty}_{i=k}\sqrt{\varepsilon_{i-1}}
+ \frac{1}{2b}\sum^{\infty}_{i=k}s_i^{1-\frac{1}{\tau}}},
\end{aligned}
\end{equation*}
which further implies that
\begin{equation*}
\begin{aligned}
\Upsilon_k
&\leq \left(\Upsilon_{k-1}-\Upsilon_k\right)
+ 2\bar{\beta}_1\left(\Upsilon_{k-1}-\Upsilon_k\right)^{\frac{1-\vartheta}{\vartheta}}
+ 2\bar{\beta}_1\left(\sqrt{\varepsilon_{k-1}}\right)^{\frac{1-\vartheta}{\vartheta}}
+ 2\bar{\beta}_3\Big(s_k^{1-\frac{1}{\tau}}\Big)^{\frac{1-\vartheta}{\vartheta}}
+ {\textstyle\sum^{\infty}_{i=k}\sqrt{\varepsilon_{i-1}}
+ \frac{1}{b}\sum^{\infty}_{i=k}s_i^{1-\frac{1}{\tau}}} \\
&\leq \left(2\bar{\beta}_1+1\right)\left(\Upsilon_{k-1}-\Upsilon_k\right)^{\frac{1-\vartheta}{\vartheta}}
+ 2\bar{\beta}_1 \left(\sqrt{\varepsilon_{k-1}}\right)^{\frac{1-\vartheta}{\vartheta}}
+ 2\bar{\beta}_3 \Big(s_k^{1-\frac{1}{\tau}}\Big)^{\frac{1-\vartheta}{\vartheta}}
+ {\textstyle\sum^{\infty}_{i=k}\sqrt{\varepsilon_{i-1}}
+ \frac{1}{b}\sum^{\infty}_{i=k}s_i^{1-\frac{1}{\tau}}},
\end{aligned}
\end{equation*}
where $\bar{\beta}_1:=\frac{b}{a}\tilde{a}\left(\tilde{a}b\right)^{\frac{1-\vartheta}{\vartheta}}$ and $\bar{\beta}_3:=\frac{b}{a}\tilde{a}^{\frac{1}{\vartheta}}$, and the last inequality follows from $\Upsilon_{k-1}-\Upsilon_k\leq\left(\Upsilon_{k-1}-\Upsilon_k\right)^{\frac{1-\vartheta}{\vartheta}}$ when $\Upsilon_{k-1}-\Upsilon_k\leq1$ for all sufficiently large $k$. Then, raising to a power of $\frac{\vartheta}{1-\vartheta}$ (which is larger than 1 for any $\theta_{\tau}\leq\vartheta<1$ in this case) to both sides of the above inequality, we see further that, for all sufficiently large $k$,
\begin{equation*}
{\small\begin{aligned}
\Upsilon_k^{\frac{\vartheta}{1-\vartheta}}
&\leq \left(\left(2\bar{\beta}_1+1\right)\left(\Upsilon_{k-1}
-\Upsilon_k\right)^{\frac{1-\vartheta}{\vartheta}}
+ 2\bar{\beta}_1 \left(\sqrt{\varepsilon_{k-1}}\right)^{\frac{1-\vartheta}{\vartheta}}
+ 2\bar{\beta}_3 \Big(s_k^{1-\frac{1}{\tau}}\Big)^{\frac{1-\vartheta}{\vartheta}}
+ \sum^{\infty}_{i=k}\sqrt{\varepsilon_{i-1}}
+ \frac{1}{b}\sum^{\infty}_{i=k}s_i^{1-\frac{1}{\tau}} \right)^{\frac{\vartheta}{1-\vartheta}} \\
&\leq \beta_1\left(\Upsilon_{k-1}-\Upsilon_k\right)
+ \beta_2 \sqrt{\varepsilon_{k-1}}
+ \beta_3 s_k^{1-\frac{1}{\tau}}
+ \beta_4 \Big({\textstyle\sum^{\infty}_{i=k}}\sqrt{\varepsilon_{i-1}} \Big)^{\frac{\vartheta}{1-\vartheta}}
+ \beta_5 \left({\textstyle\sum^{\infty}_{i=k}}s_i^{1-\frac{1}{\tau}} \right)^{\frac{\vartheta}{1-\vartheta}},
\end{aligned}}
\end{equation*}
where the last inequality follows from Jensen's inequality with $\beta_1:= 5^{\frac{2\vartheta-1}{1-\vartheta}}\left(2\bar{\beta}_1+1\right)^{\frac{\vartheta}{1-\vartheta}}$,
$\beta_2:=5^{\frac{2\vartheta-1}{1-\vartheta}} \left(2\bar{\beta}_1\right)^{\frac{\vartheta}{1-\vartheta}}$,
$\beta_3:=5^{\frac{2\vartheta-1}{1-\vartheta}} \left(2\bar{\beta}_3\right)^{\frac{\vartheta}{1-\vartheta}}$,
$\beta_4:=5^{\frac{2\vartheta-1}{1-\vartheta}}$,
and $\beta_5:=5^{\frac{2\vartheta-1}{1-\vartheta}}b^{-\frac{\vartheta}{1-\vartheta}}$. This proves (ii).
\end{proof}

With the above preparations in place, we are now ready to establish explicit convergence rate estimates for the sequence $\{\bm{x}^k\}$ generated by our iVPGSA, based on the KL exponent $\theta_{\tau}$ of the auxiliary function $\Phi_{\tau}$ for some $\tau>1$. These estimates provide a precise characterization of the algorithm's convergence behavior under the given assumptions.

\begin{theorem}\label{thm-convrate}
Suppose that Assumptions \ref{assumA}, \ref{assumB}, \ref{assumC}, \ref{assumD} hold, there exist $\underline{\gamma}, \,\overline{\gamma}>0$ such that $\frac{L_h}{2}<\underline{\gamma}\leq\overline{\gamma}$ and $\underline{\gamma}I_n \preceq H_k \preceq \overline{\gamma}I_n$ for all $k\geq0$, and the summable error conditions in \eqref{epssumcond} hold. Moreover, let $\tau>1$ and suppose that $\Phi_{\tau}$ defined in \eqref{defpofun} is a KL function with an exponent $\theta_{\tau}$. Let $\{\bm{x}^k\}$ be the sequence generated by the iVPGSA in Algorithm \ref{algo-iVPGSA} and let $\bm{x}^*$ be the limit point. Then, the following statements hold.
\begin{itemize}
\item[{\rm (i)}] If $\theta_{\tau}=0$, the sequence $\{\bm{x}^k\}$ converges to $\bm{x}^*$ finitely.

\item[{\rm (ii)}] If $\theta_{\tau}\in\left(0,\frac{1}{2}\right]$, we have the following results.
    \begin{itemize}
    \item[(a)] For $\varepsilon_k=\varepsilon_0\zeta^k$ with $0<\zeta<1$ and $\varepsilon_0>0$, there exists $\varrho\in(0,1)$ such that
        \begin{equation*}
        \|\bm{x}^k-\bm{x}^*\|\leq\mathcal{O}\big(\varrho^k\big)
        \end{equation*}
        for all sufficiently large $k$.

    \vspace{1mm}
    \item[(b)] For $\varepsilon_k=\frac{\varepsilon_0}{(k+1)^q}$ with $q>\frac{2\tau-1}{\tau-1}$ and $\varepsilon_0>0$, it holds for all sufficiently large $k$ that
        \begin{equation*}
        \|\bm{x}^k-\bm{x}^*\|\leq
        \left\{\begin{aligned}
        &\mathcal{O}\left(k^{-\left[\frac{\tau-1}{\tau}(q-1)-1\right]}\right),
        &&\text{if}~~1<\tau\leq2~~\text{and}~~\textstyle{q>\frac{2\tau-1}{\tau-1}}, \\
        &\mathcal{O}\left(k^{-\left[\frac{\tau-1}{\tau}(q-1)-1\right]}\right),
        &&\text{if}~~\tau>2~~\text{and}~~\textstyle{\frac{2\tau-1}{\tau-1}<q<\frac{2(\tau-1)}{\tau-2}}, \\
        &\mathcal{O}\left(k^{-\left(\frac{q}{2}-1\right)}\right),
        &&\text{if}~~\tau>2~~\text{and}~~\textstyle{q\geq\frac{2(\tau-1)}{\tau-2}}.
        \end{aligned}\right.
        \end{equation*}
    \end{itemize}

\item[{\rm (iii)}] If $\theta_{\tau}\in\left(\frac{1}{2}, 1\right)$, we have the following results for all sufficiently large $k$.
    \begin{itemize}
    \item[(a)] For $\varepsilon_k=\varepsilon_0\zeta^k$ with $0<\zeta<1$ and $\varepsilon_0>0$, it holds that
        \begin{equation*}
        \|\bm{x}^k-\bm{x}^*\|\leq \mathcal{O}\left(k^{-\frac{1-\theta_{\tau}}{2\theta_{\tau}-1}}\right).
        \end{equation*}

    \item[(b)] For $\varepsilon_k=\frac{\varepsilon_0}{(k+1)^q}$ with $q>\frac{2\tau-1}{\tau-1}$ and $\varepsilon_0>0$, it holds that \vspace{1mm}
        \begin{itemize}
        \item[(b1)] if $1<\tau\leq2$ and $q>\frac{2\tau-1}{\tau-1}$; or $\tau>2$ and $\frac{2\tau-1}{\tau-1}<q<\frac{2(\tau-1)}{\tau-2}$, we have that
            \begin{equation*}
            \hspace{-0.5cm}
            \|\bm{x}^k-\bm{x}^*\|
            \leq\left\{\begin{aligned}
            &\mathcal{O}\left(k^{-\left(\frac{\tau-1}{\tau}(q-1)-1\right)}\right),
            &&\text{if}\quad q<{\textstyle\frac{\tau}{\tau-1}\cdot\frac{\theta_{\tau}}{2\theta_{\tau}-1}+1},  \\
            &\mathcal{O}\left(k^{-\frac{1-\theta_{\tau}}{2\theta_{\tau}-1}}\right),
            &&\text{otherwise}.
            \end{aligned}\right.
            \end{equation*}

        \item[(b2)] if $\tau>2$ and $q\geq\frac{2(\tau-1)}{\tau-2}$, we have that
            \begin{equation*}
            \hspace{-0.5cm}
            \|\bm{x}^k-\bm{x}^*\|
            \leq\left\{\begin{aligned}
            &\mathcal{O}\left(k^{-\left(\frac{q}{2}-1\right)}\right),
            &&\text{if}\quad q<{\textstyle\frac{2\theta_{\tau}}{2\theta_{\tau}-1}},  \\
            &\mathcal{O}\left(k^{-\frac{1-\theta_{\tau}}{2\theta_{\tau}-1}}\right),
            &&\text{otherwise}.
            \end{aligned}\right.
            \end{equation*}
        \end{itemize}
    \end{itemize}
\end{itemize}
\end{theorem}
\begin{proof}
\textit{Statement (i)}. First, based on \eqref{Hvalineq} and the subsequent arguments, we know that the sequence $\{\Phi_{\tau}\big(\bm{x}^{k}, \sqrt[\tau]{s_{k}}\big)\}$ is non-increasing and converges to $c^*$. We next claim that, if $\theta_{\tau}=0$, there must exist $\bar{k}\geq0$ such that $\Phi_{\tau}\big(\bm{x}^{\bar{k}}, \sqrt[\tau]{s_{\bar{k}}}\big)=c^*$. Suppose to the contrary that $\Phi_{\tau}\big(\bm{x}^{k}, \,\sqrt[\tau]{s_{k}}\big)>c^*$ for all $k\geq0$. We then have from the definition of the KL exponent (see Definition \ref{defKLexpo}) that, for all sufficiently large $k$,
\begin{equation*}
\mathrm{dist}\left(0,\, \partial \Phi_{\tau}\big(\bm{x}^{k}, \sqrt[\tau]{s_{k}}\big)\right) \geq a,
\end{equation*}
which contradicts \eqref{subdiffbdIC2} due to $\|\bm{x}^{k}-{\bm{x}}^{k-1}\|\to0$, $\sqrt{\varepsilon_{k-1}}\to0$ and $ s_k^{1-\frac{1}{\tau}}\to0$. Therefore, we have that $\Phi_{\tau}\big(\bm{x}^{\bar{k}}, \,\sqrt[\tau]{s_{\bar{k}}}\big)=c^*$ for some $\bar{k}\geq 0$. Since $\{\Phi_{\tau}\big(\bm{x}^{k}, \sqrt[\tau]{s_{k}}\big)\}$ is non-increasing, we must have $\Phi_{\tau}\big(\bm{x}^{k}, \,\sqrt[\tau]{s_{k}}\big) \equiv c^*$ for all $k\geq\bar{k}$. This together with \eqref{Hvalineq} implies that $\bm{x}^{\bar{k}+t}=\bm{x}^{\bar{k}}$ for all $t\geq0$, and hence $\{\bm{x}^k\}$ must converge to $\bm{x}^*$ finitely.

\vspace{2mm}
\textit{Statement (ii)}. Suppose that $\theta_{\tau}\in\left(0,\frac{1}{2}\right]$. We consider the following two cases.

We first consider $\varepsilon_k=\varepsilon_0\zeta^k$ with $0<\zeta<1$ and $\varepsilon_0>0$. Substituting $\varepsilon_k=\varepsilon_0 \zeta^k$ in \eqref{KLineqrate2case1} and applying Lemma \ref{lemma-orderofseries}(i), we have that
\begin{equation*}
\Upsilon_k
\leq \rho \Upsilon_{k-1}
+ \alpha_1{\textstyle\sum^{\infty}_{i=k}}\sqrt{\varepsilon_{i-1}}
+ \alpha_2{\textstyle\sum^{\infty}_{i=k}}\,s_i^{1-\frac{1}{\tau}}
= \rho \Upsilon_{k-1}
+ \alpha_1 l_2 \zeta^{\frac{k}{2}}
+ \alpha_2 l_3 \zeta^{\left(1-\frac{1}{\tau}\right)k}
\end{equation*}
for all $k \geq k_1$ with some $k_1\geq0$. Further, by recursion, we obtain that,
\begin{equation*}
\begin{aligned}
\Upsilon_k
&\leq \rho^{k-k_1} \Upsilon_{k_1-1}
+ \alpha_1 l_2 {\textstyle\sum_{i=k_1}^k} \rho^{k-i} (\sqrt{\zeta})^{i}
+ \alpha_2 l_3 {\textstyle\sum_{i=k_1}^k} \rho^{k-i} \big(\zeta^{1-\frac{1}{\tau}}\big)^i    \\
&\leq \big(\rho^{-k_1}\Upsilon_{k_1-1}\big)\rho^{k}
+ \alpha_1 l_2 k \big(\max\{\rho, \sqrt{\zeta}\}\big)^k
+ \alpha_2 l_3 k \big(\max\{\rho, \zeta^{1-\frac{1}{\tau}}\}\big)^k \\
&\leq \big(\rho^{-k_1}\Upsilon_{k_1-1}\big)\rho^{k}
+ \alpha_1 l_2 \left(\max\{\rho, \sqrt{\zeta}\}\right)^{\frac{k}{2}}
+ \alpha_2 l_3 \left(\max\{\rho, \zeta^{1-\frac{1}{\tau}}\}\right)^{\frac{k}{2}},
\end{aligned}
\end{equation*}
where the last inequality holds for all sufficiently large $k$, because $0<\max\{\rho, \sqrt{\zeta}\}<1$ and $0<\max\{\rho, \zeta^{1-\frac{1}{\tau}}\}<1$, and hence both $k\left(\max\{\rho, \sqrt{\zeta}\}\right)^{\frac{k}{2}}<1$ and $k\big(\max\{\rho, \zeta^{1-\frac{1}{\tau}}\}\big)^{\frac{k}{2}}<1$ hold for all sufficiently large $k$. Thus, using the above relation, there exists $\varrho\in(0,1)$ such that $\Upsilon_k\leq\mathcal{O}\big(\varrho^k\big)$ for all sufficiently large $k$. This then together with $\|\bm{x}^k-\bm{x}^*\|\leq\Upsilon_k$ gives the desired result in statement (ii)(a).

\vspace{1mm}
\textit{Statement (ii)(b)}. We next consider $\varepsilon_k=\frac{\varepsilon_0}{(k+1)^q}$ with $q>\frac{2\tau-1}{\tau-1}$ and $\varepsilon_0>0$. Substituting $\varepsilon_k = \frac{\varepsilon_0}{(k+1)^q}$ in \eqref{KLineqrate2case1} and applying Lemma \ref{lemma-orderofseries}(ii), we have
\begin{equation*}
\begin{aligned}
\Upsilon_k
\leq \rho \Upsilon_{k-1}
+ \frac{\alpha_1e_2}{k^{\frac{q}{2}-1}}
+ \frac{\alpha_2e_3}{k^{\frac{\tau-1}{\tau}(q-1)-1}}
\leq\rho\Upsilon_{k-1}+C\frac{1}{k^{\psi_1(q,\tau)-1}},
\end{aligned}
\end{equation*}
where $C:=\alpha_1e_2+\alpha_2e_3$ and $\psi_1(q,\tau)$ is defined in \eqref{defpsi1}. Using the above relation and applying Lemma \ref{lemma-convrate}(i) and Lemma \ref{lem-closed-psi1}, we conclude that
\begin{equation*}
\Upsilon_k \leq \mathcal{O}\left(\frac{1}{k^{\psi_1(q,\tau)-1}}\right)
=\left\{\begin{aligned}
&\mathcal{O}\left(k^{-\left[\frac{\tau-1}{\tau}(q-1)-1\right]}\right),
&&\text{if}~~1<\tau\leq2~~\text{and}~~\textstyle{q>\frac{2\tau-1}{\tau-1}}, \\[3pt]
&\mathcal{O}\left(k^{-\left[\frac{\tau-1}{\tau}(q-1)-1\right]}\right),
&&\text{if}~~\tau>2~~\text{and}~~\textstyle{\frac{2\tau-1}{\tau-1}<q<\frac{2(\tau-1)}{\tau-2}}, \\[3pt]
&\mathcal{O}\left(k^{-\left(\frac{q}{2}-1\right)}\right),
&&\text{if}~~\tau>2~~\text{and}~~\textstyle{q\geq\frac{2(\tau-1)}{\tau-2}}.
\end{aligned}\right.
\end{equation*}
This then, together with $\|\bm{x}^k-\bm{x}^*\|\leq\Upsilon_k$, gives the desired result in statement (ii)(b).

\vspace{2mm}
\textit{Statement (iii)}. Suppose that $\theta_{\tau}\in\left(\frac{1}{2},1\right)$. We consider the following two cases.

\vspace{1mm}
\textit{Statement (iii)(a)}. We first consider $\varepsilon_k=\varepsilon_0\zeta^k$ with $0<\zeta<1$ and $\varepsilon_0>0$. Since $\sqrt{\varepsilon_{k-1}} \leq \sum^{\infty}_{i=k}\sqrt{\varepsilon_{i-1}}\to0$, $s_k^{1-\frac{1}{\tau}}\leq \sum^{\infty}_{i=k}s_i^{1-\frac{1}{\tau}}\to0$ and $\frac{\theta_{\tau}}{1-\theta_{\tau}}>1$, one can further obtain from \eqref{KLineqrate2case2} with $\vartheta=\theta_{\tau}$ that
\begin{equation*}
\Upsilon_k^{\frac{\theta_{\tau}}{1-\theta_{\tau}}}
\leq \beta_1\left(\Upsilon_{k-1}-\Upsilon_k\right)
+ (\beta_2+\beta_4){\textstyle\sum^{\infty}_{i=k}}\sqrt{\varepsilon_{i-1}}
+ (\beta_3+\beta_5){\textstyle\sum^{\infty}_{i=k}}\,s_i^{1-\frac{1}{\tau}}
\end{equation*}
for all sufficiently large $k$. Substituting $\varepsilon_k=\varepsilon_0 \zeta^k$ in the above relation and applying Lemma \ref{lemma-orderofseries}(i), we have that
\begin{equation}\label{ine-rate-large-expo-1}
\begin{aligned}
\Upsilon_k^{\frac{\theta_{\tau}}{1-\theta_{\tau}}}
&\leq \beta_1\left(\Upsilon_{k-1}-\Upsilon_k\right)
+ (\beta_2+\beta_4)l_2\,\zeta^{\frac{k}{2}}
+ (\beta_3+\beta_5)l_3\,\zeta^{\frac{\tau-1}{\tau}k} \\
&\leq \beta_1\left(\Upsilon_{k-1}-\Upsilon_k\right)
+ \beta_6\,\zeta^{\min\left\{\frac{1}{2},\,\frac{\tau-1}{\tau}\right\}k},
\end{aligned}
\end{equation}
where $\beta_6:=(\beta_2+\beta_4)l_2+(\beta_3+\beta_5)l_3$. Now, we define a function $\chi_{\theta_{\tau}}(t):=t^{\frac{\theta_{\tau}}{1-\theta_{\tau}}}$ for $t>0$. Since $\frac{\theta_{\tau}}{1-\theta_{\tau}}>1$, $\chi_{\theta_{\tau}}$ is convex on $(0,+\infty)$ and hence it holds that
\begin{equation}\label{hcvxineq}
\chi_{\theta_{\tau}}(v) - \chi_{\theta_{\tau}}(t)
\geq {\textstyle\frac{\theta_{\tau}}{1-\theta_{\tau}}}
\,t^{\frac{2 \theta_{\tau}-1}{1-\theta_{\tau}}}(v-t),
\quad \forall\,t,v>0.
\end{equation}
Let $\xi>0$ be an arbitrary positive scalar. Then, using \eqref{hcvxineq} with $\Upsilon_{k}$ and $\xi k^{-\frac{1-\theta_{\tau}}{2\theta_{\tau}-1}}$ in place of $v$ and $t$, respectively, we obtain that
\begin{equation}\label{ine-rate-large-expo-2}
\Upsilon_{k}^{\frac{\theta_{\tau}}{1-\theta_{\tau}}}-\left(\xi k^{-\frac{1-\theta_{\tau}}{2 \theta_{\tau}-1}}\right)^{\frac{\theta_{\tau}}{1-\theta_{\tau}}}
\geq \frac{\theta_{\tau} \xi^{\frac{2 \theta_{\tau}-1}{1-\theta_{\tau}}}}{(1-\theta_{\tau}) k}\left(\Upsilon_{k}-\xi k^{-\frac{1-\theta_{\tau}}{2 \theta_{\tau}-1}}\right).
\end{equation}
Thus, combining \eqref{ine-rate-large-expo-1} and \eqref{ine-rate-large-expo-2}, we deduce that
\begin{equation*}
\beta_1\left(\Upsilon_{k-1}-\Upsilon_k\right)
+ \beta_6\,\zeta^{\min\left\{\frac{1}{2},\,\frac{\tau-1}{\tau}\right\}k}
\geq \Upsilon_k^{\frac{\theta_{\tau}}{1-\theta_{\tau}}}
\geq \frac{\theta_{\tau} \xi^{\frac{2 \theta_{\tau}-1}{1-\theta_{\tau}}}}{(1-\theta_{\tau}) k}\left(\Upsilon_{k}-\xi k^{-\frac{1-\theta_{\tau}}{2 \theta_{\tau}-1}}\right)
+ \left(\xi k^{-\frac{1-\theta_{\tau}}{2 \theta_{\tau}-1}}\right)^{\frac{\theta_{\tau}}{1-\theta_{\tau}}},
\end{equation*}
which further implies that
\begin{equation*}
\begin{aligned}
{\textstyle\left(\frac{\theta_{\tau}\xi^{\frac{2\theta_{\tau}-1}{1-\theta_{\tau}}}}{(1-\theta_{\tau}) k}+\beta_1\right)}\Upsilon_k
\leq \beta_1\Upsilon_{k-1}
+ \beta_6\,\zeta^{\min\left\{\frac{1}{2},\,\frac{\tau-1}{\tau}\right\}k}
+ \beta_7\,k^{-\frac{\theta_{\tau}}{2\theta_{\tau}-1}}
\leq \beta_1\Upsilon_{k-1}
+ (\beta_6+\beta_7)\,k^{-\frac{\theta_{\tau}}{2\theta_{\tau}-1}},
\end{aligned}
\end{equation*}
where $\beta_7:=\frac{2\theta_{\tau}-1}{1-\theta_{\tau}} \xi^{\frac{\theta_{\tau}}{1-\theta_{\tau}}}$ and the last inequality holds for all sufficiently large $k$. Dividing both sides of the above inequality by $\frac{\theta_{\tau}\xi^{\frac{2\theta_{\tau}-1}{1-\theta_{\tau}}}}{(1-\theta_{\tau}) k}+\beta_1$ and rearranging the terms, we then get that
\begin{equation*}
\begin{aligned}
\Upsilon_k
&\leq {\textstyle\left(\frac{\theta_{\tau}\xi^{\frac{2\theta_{\tau}-1}{1-\theta_{\tau}}}}{(1-\theta_{\tau}) k}+\beta_1\right)^{-1}}\beta_1\Upsilon_{k-1}
+ {\textstyle\left(\frac{\theta_{\tau}\xi^{\frac{2\theta_{\tau}-1}{1-\theta_{\tau}}}}{(1-\theta_{\tau}) k}+\beta_1\right)^{-1}}(\beta_6+\beta_7)\,k^{-\frac{\theta_{\tau}}{2\theta_{\tau}-1}} \\
&\leq \left(1-\frac{\beta_8}{\beta_8+k}\right) \Upsilon_{k-1}
+ \beta_1^{-1}(\beta_6+\beta_7)\,k^{-\frac{\theta_{\tau}}{2\theta_{\tau}-1}},
\end{aligned}
\end{equation*}
where $\beta_8:=\frac{\theta_{\tau}\xi^{\frac{2\theta_{\tau}-1}{1-\theta_{\tau}}}}{(1-\theta_{\tau})\beta_1}$. Note that $\xi>0$ can be an arbitrary positive scalar. Therefore, one can choose a $\xi>0$ such that $\beta_8>\frac{\theta_{\tau}}{2\theta_{\tau}-1}-1=\frac{1-\theta_{\tau}}{2\theta_{\tau}-1}$. Then, by applying Lemma \ref{lemma-convrate}(ii), we conclude that $\Upsilon_k \leq \mathcal{O}\left(k^{-\frac{1-\theta_{\tau}}{2\theta_{\tau}-1}}\right)$ for all sufficiently large $k$. This then together with $\|\bm{x}^k-\bm{x}^*\|\leq\Upsilon_k$ gives the desired result in statement (iii)(a).

\vspace{1mm}
\textit{Statement (iii)(b)}. We next consider $\varepsilon_k=\frac{\varepsilon_0}{(k+1)^q}$ with $q>\frac{2\tau-1}{\tau-1}$ and $\varepsilon_0>0$. Substituting $\varepsilon_k = \frac{\varepsilon_0}{(k+1)^q}$ in \eqref{KLineqrate2case2} and applying Lemma \ref{lemma-orderofseries}(ii), we have that, for any $\theta_{\tau}\leq\vartheta<1$,
\begin{equation*}
\begin{aligned}
\Upsilon_k^{\frac{\vartheta}{1-\vartheta}}
&\leq \beta_1\left(\Upsilon_{k-1}-\Upsilon_k\right)
{\textstyle+ \frac{\beta_2\sqrt{\varepsilon_0}}{k^{\frac{q}{2}}}
+ \frac{\beta_3e_1^{1-\frac{1}{\tau}}}{k^{\Xi(\tau,q)}}
+ \frac{\beta_4e_2^{\frac{\vartheta}{1-\vartheta}}}
{k^{\frac{\vartheta}{1-\vartheta}\left(\frac{q}{2}-1\right)}}
+ \frac{\beta_5e_3^{\frac{\vartheta}{1-\vartheta}} }
{k^{\frac{\vartheta}{1-\vartheta}\left[\Xi(\tau,q)-1\right]}}} \\
&\leq \beta_1\left(\Upsilon_{k-1}-\Upsilon_k\right)
+ \bar{\beta}_6\,k^{-\psi_2(q,\tau,\vartheta)},
\end{aligned}
\end{equation*}
where $\Xi(\tau,q):=\frac{\tau-1}{\tau}(q-1)$, $\bar{\beta}_6:=\beta_2\sqrt{\varepsilon_0}+\beta_3e_1^{1-\frac{1}{\tau}}
+\beta_4e_2^{\frac{\vartheta}{1-\vartheta}}+\beta_5e_3^{\frac{\vartheta}{1-\vartheta}}$, and $\psi_2(q,\tau,\vartheta)$ is defined in \eqref{defpsi2}. Then, using this relation and \eqref{ine-rate-large-expo-2} with $\vartheta$ in place of $\theta_{\tau}$, we obtain that
\begin{equation*}
\begin{aligned}
\beta_1\left(\Upsilon_{k-1}-\Upsilon_k\right)
+ {\textstyle\frac{\bar{\beta}_6}{k^{\psi_2(q,\tau,\vartheta)}}}
&\geq \Upsilon_k^{\frac{\vartheta}{1-\vartheta}}
\geq {\textstyle\frac{\vartheta\xi^{\frac{2\vartheta-1}{1-\vartheta}}}{(1-\vartheta) k}\left(\Upsilon_{k}-\xi k^{-\frac{1-\vartheta}{2\vartheta-1}}\right)
+ \left(\xi k^{-\frac{1-\vartheta}{2\vartheta-1}}\right)^{\frac{\vartheta}{1-\vartheta}}},
\end{aligned}
\end{equation*}
which further implies that
\begin{equation*}
\begin{aligned}
{\textstyle\left(\frac{\vartheta\xi^{\frac{2\vartheta-1}{1-\vartheta}}}{(1-\vartheta)k}+\beta_1\right)} \Upsilon_k
\leq \beta_1\Upsilon_{k-1}
+ \bar{\beta}_6\,k^{-\psi_2(q,\tau,\vartheta)}
+ \bar{\beta}_7\,k^{-\frac{\vartheta}{2\vartheta-1}}
\leq \beta_1\Upsilon_{k-1}
+ (\bar{\beta}_6+\bar{\beta}_7)\,k^{-\psi_3(q,\tau,\vartheta)},
\end{aligned}
\end{equation*}
where $\bar{\beta}_7:=\frac{2\vartheta-1}{1-\vartheta} \xi^{\frac{\vartheta}{1-\vartheta}}$ and $\psi_3(q,\tau,\vartheta)$ is defined in \eqref{defpsi3}. Dividing both sides of the above inequality by $\frac{\vartheta\xi^{\frac{2\vartheta-1}{1-\vartheta}}}{(1-\vartheta) k}+\beta_1$ and rearranging the terms, we then get that
\begin{equation*}
\begin{aligned}
\Upsilon_k
&\leq {\textstyle\left(\frac{\vartheta\xi^{\frac{2\vartheta-1}{1-\vartheta}}}{(1-\vartheta) k}+\beta_1\right)^{-1}}\beta_1\Upsilon_{k-1}
+ {\textstyle\left(\frac{\vartheta\xi^{\frac{2\vartheta-1}{1-\vartheta}}}{(1-\vartheta) k}+\beta_1\right)^{-1}}(\bar{\beta}_6+\bar{\beta}_7)\,k^{-\psi_3(q,\tau,\vartheta)}  \\
&\leq {\textstyle\left(1-\frac{\bar{\beta}_8}{\bar{\beta}_8+k}\right)}\Upsilon_{k-1}
+ \beta_1^{-1}(\bar{\beta}_6+\bar{\beta}_7)\,k^{-\psi_3(q,\tau,\vartheta)},
\end{aligned}
\end{equation*}
where $\bar{\beta}_8:=\frac{\vartheta\xi^{\frac{2\vartheta-1}{1-\vartheta}}}{(1-\vartheta)\beta_1}$. Notably, the above relation holds for any $\theta_{\tau}\leq\vartheta<1$ and the exponent $\psi_3(q,\tau,\vartheta)$ determines the final convergence rate. Thus, by applying Lemma \ref{lem-closed-psi3}, we can select the optimal $\vartheta^{(q,\tau,*)}$ to achieve the best rate exponent
$\psi_3^{(q,\tau,*)}
:=\max\limits_{\vartheta}\left\{\psi_3(q,\tau,\vartheta):\theta_{\tau}\leq\vartheta<1\right\}$.
Furthermore, note that $\xi>0$ can be an arbitrary positive scalar. Thus, one can select a $\xi>0$ such that $\bar{\beta}_8>\psi_3^{(q,\tau,*)}-1$. Then, by applying Lemma \ref{lemma-convrate}(ii), we conclude that $\Upsilon_k \leq \mathcal{O}\left(k^{-(\psi_3^{(q,\tau,*)}-1)}\right)$ for all sufficiently large $k$. Using this result along with Lemma \ref{lem-closed-psi3} and $\|\bm{x}^k-\bm{x}^*\|\leq\Upsilon_k$, we can derive the desired result in statement (iii)(b).
\end{proof}

\subsection{Further discussions on the KL exponent of $\Phi_{\tau}$}\label{sec-KLpofun}

Our convergence analysis above relies on the KL property and exponent of the auxiliary function $\Phi_{\tau}$ in \eqref{defpofun} for some $\tau>1$. It is worth noting that the function $t\mapsto\tau^{-1}t^{\tau}$ is semialgebraic for any rational $\tau>1$. Consequently, when the objective function $F$ in \eqref{defFun} is semialgebraic, $\Phi_{\tau}$ with any rational $\tau>1$ (which is now the sum of two semialgebraic functions) is also semialgebraic and hence it is a KL function; see, for example,
\citep[Section 4.3]{abrs2010proximal} and \citep[Section 2]{bdl2007the}. Moreover, we can rigorously establish the connection between the KL exponents of $F$ and $\Phi_{\tau}$ in the following proposition. This exponent formula is closely related to the rule for block-separable sums in \citep[Theorem 3.3]{lp2017calculus}, which assumes that each block function has an exponent in $(0,1)$ and is continuous on its subdifferential domain. For the particular sum $F(\bm{x})+t^{\tau}/\tau$, the following direct proof also covers $\theta_F=0$ and does not require continuity of $F$ on $\mathrm{dom}\,\partial F$.

\begin{proposition}\label{prop-klexpo}
Suppose that $F$ defined in \eqref{defFun} has the KL property at $\tilde{\bm{x}}\in{\rm dom}\,\partial F$ with an exponent $\theta_F\in[0,1)$. Then, for any $\tau>1$, the auxiliary function $\Phi_{\tau}$ defined in \eqref{defpofun} has the KL property at $(\tilde{\bm{x}},0)$ with an exponent $\theta_{\tau}:=\max\left\{\theta_F,\,\frac{\tau-1}{\tau}\right\}$.
\end{proposition}
\begin{proof}
By the hypothesis, there exist $a,\mu,\nu$ such that for any $\bm{x}$ satisfying $\bm{x}\in\mathrm{dom}\,\partial F$, $\|\bm{x}-\tilde{\bm{x}}\|\leq\mu$ and
$F(\tilde{\bm{x}})<F(\bm{x})<F(\tilde{\bm{x}})+\nu$, the following holds:
\begin{equation}\label{KLineq-add}
\mathrm{dist}(0, \,\partial F(\bm{x}))
\geq a\left(F(\bm{x})-F(\tilde{\bm{x}})\right)^{\theta_F}.
\end{equation}
Without loss of generality, we further assume that $\mu,\nu\leq1$. In the following, let $\theta_{\tau}:=\max\left\{\theta_F,\,\frac{\tau-1}{\tau}\right\}$ and consider any $(\bm{x},t)$ satisfying $\bm{x}\in\mathrm{dom}\,\partial F$, $\|\bm{x}-\tilde{\bm{x}}\|\leq\mu$, $0\leq t\leq\mu$, and $\Phi_{\tau}(\tilde{\bm{x}}, 0)<\Phi_{\tau}(\bm{x}, t)<\Phi_{\tau}(\tilde{\bm{x}}, 0)+\nu$. Note that, for such $(\bm{x},t)$, we have that
\begin{equation*}
F(\bm{x})
\leq \Phi_{\tau}(\bm{x}, t)
< \Phi_{\tau}(\tilde{\bm{x}}, 0) + \nu
= F(\tilde{\bm{x}}) + \nu.
\end{equation*}
Thus, \eqref{KLineq-add} holds for these $\bm{x}$ provided that $F(\bm{x})>F(\tilde{\bm{x}})$ in addition.

We next claim that, for any $\tau>1$, it holds that
\begin{equation}\label{KLineq-add-new}
\left[\mathrm{dist}(0, \,\partial F(\bm{x}))\right]^{\frac{1}{\theta_{\tau}}}
\geq a^{\frac{1}{\theta_{\tau}}}\left(F(\bm{x})-F(\tilde{\bm{x}})\right).
\end{equation}
Notably, if $F(\bm{x}) \leq F(\tilde{\bm{x}})$, \eqref{KLineq-add-new} holds trivially. We then consider the case where $F(\tilde{\bm{x}})<F(\bm{x})<F(\tilde{\bm{x}})+\nu$ with $\nu\leq1$. When $\theta_F=0$, it follows from \eqref{KLineq-add} that
\begin{equation*}
\left[\mathrm{dist}(0, \,\partial F(\bm{x}))\right]^{\frac{1}{\theta_{\tau}}} \geq a^{\frac{1}{\theta_{\tau}}} \geq a^{\frac{1}{\theta_{\tau}}}\left(F(\bm{x})-F(\tilde{\bm{x}})\right),
\end{equation*}
where the last inequality follows from $0<F(\bm{x})-F(\tilde{\bm{x}})\leq1$. Hence, \eqref{KLineq-add-new} holds. Moreover, when $\theta_F\in(0,1)$, it follows from \eqref{KLineq-add} that
\begin{equation*}
\left[\mathrm{dist}(0, \,\partial F(\bm{x}))\right]^{\frac{1}{\theta_{\tau}}}
\geq a^{\frac{1}{\theta_{\tau}}}\left(F(\bm{x})-F(\tilde{\bm{x}})\right)^{\frac{\theta_F}{\theta_{\tau}}}
\geq a^{\frac{1}{\theta_{\tau}}}\left(F(\bm{x})-F(\tilde{\bm{x}})\right),
\end{equation*}
where the last inequality follows from $0<\frac{\theta_F}{\theta_{\tau}}\leq1$ and $0<F(\bm{x})-F(\tilde{\bm{x}})\leq1$. Thus, the claim is proven for all cases.

Now, set $\widehat{a}:=\min\{a^{\frac{1}{\theta_{\tau}}},\,\tau\}$. By the claim above, we have $\inf\limits_{\bm{\xi}\in\partial F(\bm{x})}\|\bm{\xi}\|^{\frac{1}{\theta_{\tau}}}\geq \widehat{a}\left(F(\bm{x})-F(\tilde{\bm{x}})\right)$.
Then, we see that
\begin{equation*}
\begin{aligned}
\left[\,\operatorname{dist}\left(0, \,\partial \Phi_\tau(\bm{x},t)\right)\,\right]^{\frac{1}{\theta_{\tau}}}
&= \left[\,\inf_{\bm{\xi}\in\partial F(\bm{x})}\left\|\left(\bm{\xi},\,t^{\tau-1}\right)\right\|\,\right]^{\frac{1}{\theta_{\tau}}}
\geq \left[\,\inf_{\bm{\xi}\in\partial F(\bm{x})} b_0^{\theta_{\tau}}\left(\|\bm{\xi}\|^{\frac{1}{\theta_{\tau}}}
+\left(t^{\tau-1}\right)^{\frac{1}{\theta_{\tau}}}\right)^{\theta_{\tau}}\,\right]^{\frac{1}{\theta_{\tau}}} \\
&= b_0\left(t^{\tau-1}\right)^{\frac{1}{\theta_{\tau}}}
+ b_0\inf_{\bm{\xi}\in\partial F(\bm{x})}\|\bm{\xi}\|^{\frac{1}{\theta_{\tau}}}
\geq b_0\,t^\tau
+ b_0\,\widehat{a}
\left(F(\bm{x})-F(\tilde{\bm{x}})\right)  \\
&= b_0\widehat{a}\left(F(\bm{x})-F(\tilde{\bm{x}})+{\textstyle\frac{1}{\tau}}t^\tau\right)
+ b_0\left(1-{\textstyle\frac{\widehat{a}}{\tau}}\right)t^\tau  \\
&\geq b_1\left(\Phi_\tau(\bm{x}, t)-\Phi_\tau(\tilde{\bm{x}}, 0)\right) > 0,
\end{aligned}
\end{equation*}
where the first inequality follows from Lemma \ref{normineq1}, and the last inequality follows by setting $b_1:=b_0\widehat{a}$ and using $\widehat{a}\leq\tau$. From the above relation, we obtain the desired result.
\end{proof}

Note that, for $t>0$, the KL inequality with exponent $\theta_{\tau}$ holds trivially for $\Phi_{\tau}$, since $\operatorname{dist}(0,\partial\Phi_{\tau}(\bm{x},t))\geq t^{\tau-1}$. Combining this observation with Proposition \ref{prop-klexpo}, we see that if the objective function $F$ defined in \eqref{defFun} is a KL function with exponent $\theta_F$, then, for any $\tau>1$, the auxiliary function $\Phi_{\tau}$ defined in \eqref{defpofun} is also a KL function with exponent
$\theta_{\tau}:=\max\left\{\theta_F,\,\frac{\tau-1}{\tau}\right\}$. This result establishes an explicit connection between the KL geometries of the original objective $F$ and the auxiliary function $\Phi_{\tau}$. Building on this connection, we can choose $\tau$ appropriately and apply Theorem \ref{thm-convrate} to derive convergence-rate estimates for the iVPGSA directly in terms of the classical KL exponent $\theta_F$ of $F$.

\begin{theorem}\label{thm-convrate-Fexpo}
Suppose that Assumptions \ref{assumA}, \ref{assumB}, \ref{assumC}, \ref{assumD} hold, there exist $\underline{\gamma}, \,\overline{\gamma}>0$ such that $\frac{L_h}{2}<\underline{\gamma}\leq\overline{\gamma}$ and $\underline{\gamma}I_n \preceq H_k \preceq \overline{\gamma}I_n$ for all $k\geq0$, the summable error conditions in \eqref{epssumcond} hold, and $F$ defined in \eqref{defFun} is a KL function with an exponent $\theta_{F}\in[0,1)$. Let $\{\bm{x}^k\}$ be the sequence generated by the iVPGSA in Algorithm \ref{algo-iVPGSA} and let $\bm{x}^*$ be the limit point. Let $\varepsilon_0>0$. Then, the following statements hold.
\begin{itemize}
\item[{\rm (i)}] If $\theta_{F}=0$, the following results hold for all sufficiently large $k$.
    \begin{itemize}
    \item[(a)] For $\varepsilon_k=\varepsilon_0\zeta^k$ with $0<\zeta<1$, there exists a $\varrho_1\in(0,1)$ such that $\|\bm{x}^k-\bm{x}^*\|\leq\mathcal{O}\big(\varrho_1^k\big)$.

    \item[(b)] For $\varepsilon_k=\frac{\varepsilon_0}{(k+1)^q}$ with $q>2$, it holds that $\|\bm{x}^k-\bm{x}^*\| \leq\mathcal{O}\left(k^{-\left(\frac{q}{2}-1\right)}\right)$.
    \end{itemize}

\item[{\rm (ii)}] If $\theta_{F}\in\left(0,\frac{1}{2}\right]$, the following results hold for all sufficiently large $k$.
    \begin{itemize}
    \item[(a)] For $\varepsilon_k=\varepsilon_0\zeta^k$ with $0<\zeta<1$, there exists a $\varrho_2\in(0,1)$ such that $\|\bm{x}^k-\bm{x}^*\|\leq\mathcal{O}\big(\varrho_2^k\big)$.

    \item[(b)] For $\varepsilon_k=\frac{\varepsilon_0}{(k+1)^q}$ with $q>2$, it holds that $\|\bm{x}^k-\bm{x}^*\| \leq\mathcal{O}\left(k^{-\left(\frac{q}{2}-1\right)}\right)$.
    \end{itemize}

\item[{\rm (iii)}] If $\theta_{F}\in\left(\frac{1}{2}, 1\right)$, the following results hold for all sufficiently large $k$.
    \begin{itemize}
    \item[(a)] For $\varepsilon_k=\varepsilon_0\zeta^k$ with $0<\zeta<1$, it holds that
        $\|\bm{x}^k-\bm{x}^*\|\leq \mathcal{O}\left(k^{-\frac{1-\theta_{F}}{2\theta_{F}-1}}\right)$.

    \item[(b)] For $\varepsilon_k=\frac{\varepsilon_0}{(k+1)^q}$ with $q>2$, it holds that
        \begin{equation*}
        \hspace{-0.5cm}
        {\small \|\bm{x}^k-\bm{x}^*\|
        \leq\left\{\begin{aligned}
        &\mathcal{O}\left(k^{-\left(\frac{q}{2}-1\right)}\right),
        &&\text{if}\quad q<\frac{2\theta_{F}}{2\theta_{F}-1},  \\
        &\mathcal{O}\left(k^{-\frac{1-\theta_{F}}{2\theta_{F}-1}}\right),
        &&\text{otherwise}.
        \end{aligned}\right.}
        \end{equation*}
    \end{itemize}
\end{itemize}
\end{theorem}
\begin{proof}
First, we know from Proposition \ref{prop-klexpo} that $\Phi_{\tau}$ with any $\tau>1$ is also a KL function with an exponent
\begin{equation*}
\theta_{\tau}:=\max\left\{\theta_F,\,\frac{\tau-1}{\tau}\right\}
=\left\{\begin{aligned}
&\theta_F, &&\text{if}~~1<\tau\leq\textstyle{\frac{1}{1-\theta_F}}, \\
&\textstyle{\frac{\tau-1}{\tau}}, &&\text{otherwise}.
\end{aligned}\right.
\end{equation*}

\textit{Statement (i)}. Suppose that $\theta_{F}=0$. We consider the following two cases.

We first consider $\varepsilon_k=\varepsilon_0\zeta^k$ with $0<\zeta<1$. In this case, we consider setting $\tau=2$, and hence $\theta_{\tau}=\frac{\tau-1}{\tau}=\frac{1}{2}$. Then, we can apply Theorem \ref{thm-convrate}(ii)(a) to obtain the desired result.

We next consider $\varepsilon_k=\frac{\varepsilon_0}{(k+1)^q}$ with $q>2$. In this case, we consider setting $\tau:=2+\frac{2}{q-2}>2$, and hence $\theta_{\tau}=\frac{\tau-1}{\tau}\in\left(\frac{1}{2},1\right)$. Moreover, we have that $q=\frac{2(\tau-1)}{\tau-2}=\frac{2\theta_{\tau}}{2\theta_{\tau}-1}>\frac{2\tau-1}{\tau-1}$. Then, we can apply Theorem \ref{thm-convrate}(iii)(b2), together with $\frac{1-\theta_{\tau}}{2\theta_{\tau}-1}=\frac{q}{2}-1$, to obtain the desired result.

\textit{Statement (ii)}. Suppose that $\theta_{F}\in\left(0, \frac{1}{2}\right]$. We consider the following two cases.

We first consider $\varepsilon_k=\varepsilon_0\zeta^k$ with $0<\zeta<1$. In this case, we consider setting $\tau:=\frac{1}{1-\theta_F}\in(1,2]$, and hence $\theta_{\tau}=\theta_F\in\left(0, \frac{1}{2}\right]$. Then, we can apply Theorem \ref{thm-convrate}(ii)(a) to obtain the desired result.

We next consider $\varepsilon_k=\frac{\varepsilon_0}{(k+1)^q}$ with $q>2$. In this case, we consider setting $\tau:=2+\frac{2}{q-2}>2\geq\frac{1}{1-\theta_F}$, and hence $\theta_{\tau}=\frac{\tau-1}{\tau}\in\left(\frac{1}{2},1\right)$. Moreover, we have $q=\frac{2(\tau-1)}{\tau-2}=\frac{2\theta_{\tau}}{2\theta_{\tau}-1}>\frac{2\tau-1}{\tau-1}$. Then, we apply Theorem \ref{thm-convrate}(iii)(b2), together with $\frac{1-\theta_{\tau}}{2\theta_{\tau}-1}=\frac{q}{2}-1$, to get the desired result.

\textit{Statement (iii)}. Suppose that $\theta_{F}\in\left(\frac{1}{2},1\right)$. We consider the following two cases.

We first consider $\varepsilon_k=\varepsilon_0\zeta^k$ with $0<\zeta<1$. In this case, we consider setting $\tau:=\frac{1}{1-\theta_F}>2$, and hence $\theta_{\tau}=\theta_F\in\left(\frac{1}{2},1\right)$. Then, we can apply Theorem \ref{thm-convrate}(iii)(a) to obtain the desired result.

We next consider $\varepsilon_k=\frac{\varepsilon_0}{(k+1)^q}$ with $q>2$. In this case, we consider setting $\tau:=2+\frac{2}{q-2}$, and hence
\begin{equation*}
{\small\theta_{\tau}
=\left\{\begin{aligned}
&\theta_F, &&\text{if}~~\tau:=2+\frac{2}{q-2}\leq\frac{1}{1-\theta_F}, \\
&\frac{\tau-1}{\tau}, &&\text{otherwise},
\end{aligned}\right.
=\left\{\begin{aligned}
&\theta_F, &&\text{if}~~q\geq\frac{2\theta_F}{2\theta_F-1}, \\
&\frac{\tau-1}{\tau}, &&\text{otherwise}.
\end{aligned}\right.}
\end{equation*}
Moreover, we have $q=\frac{2(\tau-1)}{\tau-2}>\frac{2\tau-1}{\tau-1}$. Then, we can apply Theorem \ref{thm-convrate}(iii)(b2) to obtain the desired result.
\end{proof}

\section{Numerical experiments}\label{sec-num}

In this section, we conduct some numerical experiments to evaluate the performance of our iVPGSA in Algorithm \ref{algo-iVPGSA} for solving the $\ell_1/\ell_2$ Lasso problem \eqref{prob_l12lasso} and the constrained $\ell_1/\ell_2$ sparse optimization problem \eqref{prob_l12cso}. All experiments are run in {\sc Matlab} R2023a on a PC with an Intel processor i7-12700K@3.60GHz (with 12 cores and 20 threads) and 64GB of RAM, equipped with a Windows OS.

\subsection{The $\ell_1/\ell_2$ Lasso problem}\label{sec-num-l12lasso}

In this subsection, we consider the $\ell_1/\ell_2$ Lasso problem (see, e.g., \citep{lszz2022proximal,zl2022first}):
\begin{equation}\label{prob_l12lasso}
\min\limits_{\bm{x}\in\mathbb{R}^n}~~\frac{\lambda\|\bm{x}\|_1
+ \frac{1}{2}\|A\bm{x}-\bm{b}\|^2}{\|\bm{x}\|} \qquad
\mathrm{s.t.} \qquad \bm{\alpha} \leq \bm{x} \leq \bm{\beta},
\end{equation}
where $A\in \mathbb{R}^{m\times n}$, $\bm{b}\in\mathbb{R}^m$, $\lambda>0$ is the regularization parameter, and $\bm{\alpha}, \bm{\beta}\in\mathbb{R}^n$ are the given lower and upper bounds, respectively. To apply our iVPGSA in Algorithm \ref{algo-iVPGSA} for solving problem \eqref{prob_l12lasso}, we consider the following reformulation:
\begin{equation*}
\textstyle \min\limits_{\bm{x}\in\Omega}~~
F_{\mathrm{lasso}}(\bm{x}):=
\frac{\overbrace{\lambda\|\bm{x}\|_1
+ \iota_{[\bm{\alpha},\,\bm{\beta}]}(\bm{x})
+ \frac{1}{2}\|A\bm{x}-\bm{b}\|^2}^{f(\bm{x})} ~~
+ \,\overbrace{0}^{h(\bm{x})}}{\underbrace{\|\bm{x}\|}_{g(\bm{x})}},
\end{equation*}
where $\Omega:=\mathbb{R}^n\backslash\{0\}$ and $\iota_{[\bm{\alpha},\,\bm{\beta}]}$ denotes the indicator function on the set $\left\{\bm{x}\in\mathbb{R}^n : \bm{\alpha}\leq\bm{x}\leq\bm{\beta}\right\}$. As mentioned in Introduction, the above reformulation conforms to the form of \eqref{mainpro} and satisfies Assumption \ref{assumA}, and hence our iVPGSA is applicable. We will choose $H_k=\gamma_k I$ and the associated subproblem at the $k$-th iteration ($k\geq0$) is given as follows:
\begin{equation}\label{subprob_l12lasso}
\min\limits_{\bm{x}\in\mathbb{R}^n} ~~
r(\bm{x}) + \frac{1}{2}\|A\bm{x}-\bm{b}\|^2
- \langle c_k\bm{y}^{k},\,\bm{x}-\bm{x}^k\rangle
+ \frac{\gamma_k}{2} \Vert \bm{x}-\bm{x}^{k}\Vert^2,
\end{equation}
where $r(\bm{x}):=\lambda\|\bm{x}\|_1+\iota_{[\bm{\alpha},\,\bm{\beta}]}(\bm{x})$, $c_k=\frac{f(\bm{x}^k)+h(\bm{x}^k)}{g(\bm{x}^k)}$ and $\bm{y}^k\in\partial\|\bm{x}^k\|$.

\subsubsection{A dual semi-smooth Newton method for solving the subproblem \eqref{subprob_l12lasso}}\label{sec_l12lasso_ssn}

We next discuss how to efficiently solve \eqref{subprob_l12lasso} via a dual semi-smooth Newton ({\sc Ssn}) method to find a point $\bm{x}^{k+1}$ associated with an error pair $(\Delta^k,\,\delta_k)$ satisfying the inexact condition \eqref{inexcondH-x} and the error criterion \eqref{stopcritH-x}. Specifically, we solve the following equivalent problem:
\begin{equation}\label{subprobref_l12lasso}
\begin{aligned}
&\min\limits_{\bm{x}\in\mathbb{R}^n,\,\bm{y}\in\mathbb{R}^m}~r(\bm{x})
- \langle c_k\bm{y}^{k},\,\bm{x}\rangle+\frac{1}{2}\|\bm{y}-\bm{b}\|^2
+\frac{\gamma_k}{2}\|\bm{x}-\bm{x}^k\|^2 \\
&\quad~~\,\mathrm{s.t.}\qquad A\bm{x} = \bm{y}.
\end{aligned}
\end{equation}
By some manipulations, one can show that the dual problem of problem \eqref{subprobref_l12lasso} can be equivalently given by (in a minimization form)
\begin{equation}\label{subprobdual_l12lasso}
\hspace{-2mm}
\min\limits_{\bm{z}\in\mathbb{R}^m}
\left\{\begin{aligned}
&\Psi_k^{\text{lasso}}(\bm{z}):=
{\textstyle\frac{1}{2}\|\bm{z}\|^2 + \langle\bm{z},\,\bm{b}\rangle
-\lambda\left\|\texttt{prox}_{\gamma_k^{-1}r}\!
\left(\bm{v}_k(\bm{z})\right)\right\|_1}  \\[3pt]
&~~{\textstyle-\,\frac{\gamma_k}{2} \left\|\texttt{prox}_{\gamma_k^{-1}r}\!
\left(\bm{v}_k(\bm{z})\right)
- \bm{v}_k(\bm{z})\right\|^2
+ \frac{\gamma_k}{2}\left\|\bm{v}_k(\bm{z})\right\|^2
- \frac{\gamma_k}{2} \|\bm{x}^k\|^2}
\end{aligned}\right\},
\end{equation}
where $\bm{z}\in\mathbb{R}^m$ is the dual variable and
$\bm{v}_k(\bm{z}) := \bm{x}^k + \gamma_k^{-1}\left(c_k\bm{y}^{k}
-A^{\top}\bm{z}\right)$. From the property of the Moreau envelope of $\gamma_k^{-1}r$ (see, e.g., \citep[Proposition 12.29]{bc2011convex}), we see that $\Psi_k^{\text{lasso}}$ is strongly convex and continuously differentiable with the gradient
\begin{equation*}
\nabla\Psi_k^{\text{lasso}}(\bm{z})
= - A\texttt{prox}_{\gamma_k^{-1}r}\left(\bm{v}_k(\bm{z})\right)
+ \bm{z} + \bm{b}.
\end{equation*}
Note that the proximal mapping $\texttt{prox}_{\gamma_k^{-1}r}(\cdot)$ can be easily computed as follows:
\begin{equation*}
\texttt{prox}_{\gamma_k^{-1}r}\left(\bm{u}\right)
= \min\left\{\max\left\{\mathrm{sign}(\bm{u})
\odot\max\left\{|\bm{u}|-\gamma_k^{-1}\lambda, \,0\right\},\,\bm{\alpha}\right\},\,\bm{\beta}\right\}, \quad \forall\,\bm{u}\in\mathbb{R}^{n},
\end{equation*}
where $\odot$ denotes the entrywise multiplication between two vectors. Then, from the first-order optimality condition, solving problem \eqref{subprobdual_l12lasso} is equivalent to solving the following non-smooth equation:
\begin{equation}\label{subdualequa_l12lasso}
\nabla \Psi_k^{\text{lasso}}(\bm{z})=0.
\end{equation}
In view of the nice properties of $\texttt{prox}_{\gamma_k^{-1}r}$, we then follow \citep{lst2018highly,twst2020sparse,zts2023learning} to apply a globally convergent and locally superlinearly convergent {\sc Ssn} method to solve \eqref{subdualequa_l12lasso}. To this end, we define a multifunction $\widehat{\partial}^2 \Psi_k^{\text{lasso}}: \mathbb{R}^m \rightrightarrows \mathbb{R}^{m \times m}$ as follows:
\begin{equation*}
\widehat{\partial}^2\Psi_k^{\text{lasso}}(\bm{z}) := I + \gamma_k^{-1}A\,\partial\texttt{prox}_{\gamma_k^{-1}r}\left(\bm{v}_k(\bm{z})\right)A^{\top},
\end{equation*}
where $\partial\texttt{prox}_{\gamma_k^{-1}r}\left(\bm{v}_k(\bm{z})\right)$ is the Clarke subdifferential of the Lipschitz continuous mapping $\texttt{prox}_{\gamma_k^{-1}r}$ at $\bm{v}_k(\bm{z})$, defined as follows:
\begin{equation*}
{\small\partial\texttt{prox}_{\gamma_k^{-1}r}(\bm{u})
:= \left\{\mathrm{Diag}(\bm{d}) \,:\, \bm{d}\in\mathbb{R}^n, ~d_i\in
\left\{\begin{aligned}
&\{0\},   &&  \mathrm{if}~~|u_i|<\gamma_k^{-1}\lambda~~\mathrm{or}~~\hat{u}_i\notin(\alpha_i,\beta_i),    \\[3pt]
&[0,\,1], &&  \mathrm{if}~~|u_i|=\gamma_k^{-1}\lambda~~\mathrm{and}~~\hat{u}_i\in(\alpha_i,\beta_i),  \\[3pt]
&\{1\},   &&  \mathrm{if}~~|u_i|>\gamma_k^{-1}\lambda~~\mathrm{and}~~\hat{u}_i\in(\alpha_i,\beta_i),
\end{aligned}\right.~~\right\},}
\end{equation*}
where $\hat{u}_i=\mathrm{sign}\left(u_i\right)\max\left\{\left|u_i\right|-\gamma_k^{-1}\lambda,\,0\right\}$.
It is clear that all elements in $\widehat{\partial}^2\Psi_k^{\text{lasso}}(\bm{z})$ are positive definite. We are now ready to present the {\sc Ssn} method for solving equation \eqref{subdualequa_l12lasso} in Algorithm \ref{algo:SSN} and refer readers to \citep[Theorem 3.6]{lst2018highly} for its convergence.

\begin{algorithm}[htb!]
\caption{A semi-smooth Newton ({\sc Ssn}) method for solving equation \eqref{subdualequa_l12lasso}}\label{algo:SSN}
 	
\textbf{Initialization:} Choose $\bar{\eta}\in(0,1)$, $\tau_{\mathrm{N}}\in(0,1]$, $\mu\in(0,1/2)$, $\delta\in(0,1)$, and an initial point $\bm{z}^{k,0}\in\mathbb{R}^m$. Set $t=0$. Repeat until a termination criterion is met. \vspace{-1mm}
\begin{itemize}[leftmargin=1.6cm]
\item[\textbf{Step 1.}] Compute $\nabla\Psi_k^{\text{lasso}}(\bm{z}^{k,t})$ and select an element $H^{k,t}\in\widehat{\partial}^2\Psi_k^{\text{lasso}}(\bm{z}^{k,t})$. Solve the linear system $H^{k,t}\bm{d}=-\nabla\Psi_k^{\text{lasso}}(\bm{z}^{k,t})$ nearly exactly by the (sparse) Cholesky factorization with forward and backward substitutions, \textit{or} approximately by the preconditioned conjugate gradient method to find $\bm{d}^{k,t}$ such that $\big\|H^{k,t}\bm{d}^{k,t} + \nabla\Psi_k^{\text{lasso}}(\bm{z}^{k,t})\big\| \leq \min\big(\bar{\eta}, \,\|\nabla\Psi_k^{\text{lasso}}(\bm{z}^{k,t})\|^{1+\tau_{\mathrm{N}}}\big)$.
	
\item[\textbf{Step 2.}] (\textbf{Inexact line search}) Find a step size $\alpha_t:=\delta^{i_t}$, where $i_t$ is the smallest nonnegative integer $i$ for which $\Psi_k^{\text{lasso}}(\bm{z}^{k,t} + \delta^i\bm{d}^{k,t})
    \leq \Psi_k^{\text{lasso}}(\bm{z}^{k,t}) + \mu\, \delta^{i}\langle\nabla\Psi_k^{\text{lasso}}(\bm{z}^{k,t}), \,\bm{d}^{k,t}\rangle$.

\item[\textbf{Step 3.}] Set $\bm{z}^{k,t+1} = \bm{z}^{k,t} + \alpha_t\bm{d}^{k,t}$, $t=t+1$, and go to \textbf{Step 1}.
\end{itemize}
\end{algorithm}

We next show that our inexact condition \eqref{inexcondH-x} and the associated error criterion \eqref{stopcritH-x} can be met through some appropriate manipulations based on the dual sequence generated by the {\sc Ssn} method. Specifically, at the $k$-th iteration, we apply the {\sc Ssn} method for solving equation \eqref{subdualequa_l12lasso}, which generates a dual sequence $\{\bm{z}^{k,t}\}$. We then have the following proposition, whose proof is relegated to Appendix \ref{apd-pro-l12lasso}.

\begin{proposition}\label{pro-scnew-l12lasso}
Let $\bm{w}^{k,t}:=\texttt{prox}_{\gamma_k^{-1}r}
\left(\bm{x}^k+\gamma_k^{-1}\left(c_k\bm{y}^k-A^{\top}\bm{z}^{k,t}\right)\right)$ and $\bm{e}^{k,t}:=\nabla\Psi_k^{\text{lasso}}\big(\bm{z}^{k,t}\big)$. Then, we have that
\begin{equation*}
- A^\top \bm{e}^{k,t} \in \partial r(\bm{w}^{k,t}) + A^\top (A\bm{w}^{k,t}-\bm{b})
- c_k\bm{y}^k  + \gamma_k(\bm{w}^{k,t}-\bm{x}^k).
\end{equation*}
Thus, if $(\bm{w}^{k,t}, \bm{e}^{k,t})$ satisfies
\begin{equation}\label{checkcond_l12lasso}
\|A^\top\bm{e}^{k,t}\|^2 + |\langle A^\top\bm{e}^{k,t},\,\bm{w}^{k,t}-\bm{x}^{k} \rangle| \leq \varepsilon_k\,\|\bm{w}^{k,t}\|,
\end{equation}
then the error criterion \eqref{stopcritH-x} holds for $\bm{x}^{k+1}:=\bm{w}^{k,t}$, $\Delta^k:=-A^{\top}\bm{e}^{k,t}$ and $\delta_k:=0$.
\end{proposition}

From Proposition \ref{pro-scnew-l12lasso}, we see that our inexact condition \eqref{inexcondH-x} and its associated error criterion \eqref{stopcritH-x} are indeed verifiable and can be satisfied as long as inequality \eqref{checkcond_l12lasso} holds. Indeed, when the dual sequence $\{\bm{z}^{k,t}\}$ generated by the {\sc Ssn} method is convergent, we have that $\bm{e}^{k,t}\to0$ and $\{\bm{w}^{k,t}\}$ is convergent to the optimal solution $\bm{x}^{k,*}$ of the $k$-th subproblem \eqref{subprob_l12lasso}. Consequently, the left-hand-side term in \eqref{checkcond_l12lasso} approaches zero, i.e., $\|A^{\top}\bm{e}^{k,t}\|^2 +
|\langle A^{\top} \bm{e}^{k,t}, \,\bm{w}^{k,t}-\bm{x}^{k}\rangle|\to0$. On the other hand, based on the discussions following \eqref{suffdessubopt}, under the nonnegativity of $f+h$ on $\mathrm{dom}\,f$ (by Assumption \ref{assumA}4) and $H_k\succ\frac{L_h}{2}I_n$, if $\bm{x}^k\in\Omega\cap\mathrm{dom}\,f$ is not the solution of the $k$-th subproblem and $c_k>0$, we have that $\|\bm{x}^{k,*}\|>0$. Thus, the right-hand-side term in \eqref{checkcond_l12lasso} converges to a positive value. Therefore, inequality \eqref{checkcond_l12lasso} must hold after finitely many iterations.

\subsubsection{Comparison results}\label{sec-comp-l12lasso}

We compare iVPGSA with the proximal-gradient-subgradient algorithm with backtracked extrapolation (denoted by PGSA\_BE) proposed in \citep{lszz2022proximal} for solving problem \eqref{prob_l12lasso}. For iVPGSA, we set $\varepsilon_k=\max\left\{(k+1)^{-2.01},\,10^{-8}\right\}$ and $\gamma_k=\max\left\{(k+1)^{-\frac{1}{2}},\,0.01\right\}$, where the floor $10^{-8}$ is used in the numerical tests to avoid oversolving the inner subproblems. Moreover, for the {\sc Ssn} in Algorithm \ref{algo:SSN}, we set $\mu=10^{-4}$, $\delta=0.5$, $\bar{\eta}=10^{-3}$ and $\tau_{\mathrm{N}}=0.2$. We will initialize {\sc Ssn} with the origin at the first outer iteration and then employ a \textit{warm-start} strategy thereafter. Specifically, at each outer iteration, we initialize {\sc Ssn} with the approximate solution obtained by the {\sc Ssn} method in the previous outer iteration. Finally, for the PGSA\_BE, we use the parameter settings recommended
in \citep{lszz2022proximal}.

We initialize two algorithms with a point $\bm{x}^0$ obtained by applying the fast iterative shrinkage-thresholding algorithm (FISTA) with backtracking \citep{bt2009a} for solving the classical $\ell_1$ Lasso problem: $\min\big\{\lambda\|\bm{x}\|_1+\frac{1}{2}\|A\bm{x}-\bm{b}\|^2\big\}$, using 200 iterations. Finally, we terminate them when the number of iterations reaches 50000 or the following holds for 3 consecutive iterations:
\begin{equation*}
\max\left\{\frac{\|\bm{x}^k-\bm{x}^{k-1}\|}{1+\|\bm{x}^k\|},\, \frac{|F_{\text{lasso}}(\bm{x}^k)-F_{\text{lasso}}(\bm{x}^{k-1})|}{1+|F_{\text{lasso}}(\bm{x}^k)|}\right\} < 10^{-7}
~~ \mbox{or} ~~
\frac{|F_{\text{lasso}}(\bm{x}^k)-F_{\text{lasso}}(\bm{x}^{k-1})|}{1+|F_{\text{lasso}}(\bm{x}^k)|} < 10^{-10}.
\end{equation*}

In our experiments, we consider $\lambda \in \{0.001, 0.01, 0.1, 1\}$ and $(m,n,s)=(100i,1000i,20i)$ for $i\in\{5,10,15,20,25\}$. For each triple $(m, n, s)$, we randomly generate a trial as follows. First, we generate a matrix $A \in \mathbb{R}^{m\times n}$ with i.i.d. standard Gaussian entries. We then choose a subset $\mathcal{S}\subset\{1, \cdots, n\}$ of size $s$ uniformly at random and generate an $s$-sparse vector $\bm{x}_{\text{orig}}\in\mathbb{R}^{n}$, which has i.i.d. standard Gaussian entries on $\mathcal{S}$ and zeros on the complement set $\mathcal{S}^c$. Finally, we generate the vector $\bm{b} \in \mathbb{R}^{m}$ by setting $\bm{b}=A\bm{x}_{\text{orig}} + 0.01\cdot\widehat{\bm{n}}$, where $\widehat{\bm{n}}\in\mathbb{R}^m$ is a random vector with i.i.d. standard Gaussian entries. We then set $\bm{\alpha} = -5\times\textbf{1}_n$ and $\bm{\beta} = 5\times\textbf{1}_n$, where $\textbf{1}_n$ denotes the $n$-dimensional vector with all entries being $1$.

The average computational results for each triple $(m,n,s)$ across 10 instances are presented in Table \ref{Tab:l1ol2lasso}. From the results, one can observe that our iVPGSA exhibits superior numerical performance compared to PGSA\_BE, particularly when the regularization parameter $\lambda$ is small or the problem size is large. For example, in cases where $\lambda\in\{0.01,0.001\}$ and $n\geq10000$, our iVPGSA is consistently at least 4 times faster than PGSA\_BE, while achieving comparable or even better objective function values.

We also test two algorithms with two instances $(A,\bm{b})$ obtained from the data sets \texttt{mpg7} and \texttt{gisette} in the UCI data repository \citep{UCIdata}. Here, the \texttt{mpg7} data set is an expanded version of its original version, with the last digit signifying that an order 7 polynomial is used to generate the basis functions; see \citep[Section 4.1]{lst2018highly} for more details. For \texttt{mpg7}, the matrix $A$ is of size $392\times3432$ with $\lambda_{\max}(A^{\top}A)\approx1.28\times10^4$. For \texttt{gisette}, the matrix $A$ is of size $1000\times4971$ with $\lambda_{\max}(A^{\top}A)\approx3.36\times10^6$. For each data set, we choose $\lambda=\lambda_c\|A^{\top}\bm{b}\|_{\infty}$ with $\lambda_c\in\{10^{-4},10^{-5}\}$. Figure \ref{Fig:L1oL2Lasso} presents the numerical results of iVPGSA and PGSA\_BE, where the objective value $F_{\text{lasso}}(\bm{x}^k)$ is plotted against the computational time. From the results, we see that PGSA\_BE exhibits a slower rate of reduction in the objective value, possibly due to the large Lipschitz constant associated with the real data. In contrast, our iVPGSA continues to demonstrate promising performance on both real data sets.

\begin{table}[ht]
\caption{The average computational results for each triple $(m,n,s)$ from 10 instances on the $\ell_1/\ell_2$ Lasso problem with $\lambda\in\left\{0.001,0.01,0.1,1\right\}$, where ``\texttt{obj}" denotes the objective function value, ``\texttt{iter}" denotes the number of iterations (the total number of the {\sc Ssn} iterations in iVPGSA is also reported in the bracket), ``\texttt{time}" denotes the computational time, and ``\texttt{t0}" denotes the computational time used to obtain an initial point by FISTA using 200 iterations.}\label{Tab:l1ol2lasso}
\centering \tabcolsep 5pt
\scalebox{0.8}{
\renewcommand\arraystretch{1.2}
\begin{tabular}{|c|c|clcc|clcc|}
\hline
$(m,n,s)$ & \texttt{method} & \texttt{obj} & \texttt{iter} & \texttt{time} & \texttt{t0} & \texttt{obj} & \texttt{iter} & \texttt{time} & \texttt{t0}   \\
\hline
&&\multicolumn{4}{c|}{$\lambda=0.001$} & \multicolumn{4}{c|}{$\lambda=0.01$}  \\
\hline
\multirow{2}{*}{(500,5000,100)}
&PGSA\_BE & 7.98e-03 & 36674 & 14.25 & 0.16 & 7.98e-02 & 9504 & 3.62 & 0.15 \\
&iVPGSA   & 7.98e-03 & 271 (808) & 8.90 & 0.16 & 7.98e-02 & 83 (359) & 4.65 & 0.15 \\
\hline
\multirow{2}{*}{(1000,10000,200)}
&PGSA\_BE & 1.14e-02 & 50000 & 375.83 & 3.02 & 1.13e-01 & 14022 & 105.30 & 3.02 \\
&iVPGSA   & 1.13e-02 & 252 (863) & 41.50 & 3.02 & 1.13e-01 & 62 (403) & 24.30 & 3.02 \\
\hline
\multirow{2}{*}{(1500,15000,300)}
&PGSA\_BE & 1.41e-02 & 50000 & 976.40 & 7.88 & 1.38e-01 & 18199 & 355.45 & 7.91 \\
&iVPGSA   & 1.38e-02 & 255 (926) & 103.42 & 7.88 & 1.38e-01 & 71 (456) & 65.39 & 7.91 \\
\hline
\multirow{2}{*}{(2000,20000,400)}
&PGSA\_BE & 1.67e-02 & 50000 & 1696.92 & 13.74 & 1.61e-01 & 21383 & 730.86 & 13.82 \\
&iVPGSA   & 1.61e-02 & 245 (936) & 204.38 & 13.74 & 1.61e-01 & 63 (479) & 133.89 & 13.82 \\
\hline
\multirow{2}{*}{(2500,25000,500)}
&PGSA\_BE & 1.88e-02 & 50000 & 2684.69 & 21.69 & 1.79e-01 & 24634 & 1320.26 & 21.70 \\
&iVPGSA   & 1.79e-02 & 241 (953) & 356.54 & 21.69 & 1.79e-01 & 61 (507) & 235.91 & 21.70 \\
\hline
&&\multicolumn{4}{c|}{$\lambda=0.1$} & \multicolumn{4}{c|}{$\lambda=1$}  \\
\hline
\multirow{2}{*}{(500,5000,100)}
&PGSA\_BE & 7.98e-01 & 2724 & 1.06 & 0.15 & 7.97e+00 & 991 & 0.40 & 0.16 \\
&iVPGSA   & 7.98e-01 & 71 (261) & 3.77 & 0.15 & 7.97e+00 & 9 (79) & 1.16 & 0.16 \\
\hline
\multirow{2}{*}{(1000,10000,200)}
&PGSA\_BE & 1.13e+00 & 3873 & 29.08 & 3.02 & 1.13e+01 & 1299 & 9.77 & 3.03 \\
&iVPGSA   & 1.13e+00 & 33 (310) & 23.13 & 3.02 & 1.13e+01 & 14 (125) & 8.66 & 3.03 \\
\hline
\multirow{2}{*}{(1500,15000,300)}
&PGSA\_BE & 1.38e+00 & 4975 & 97.18 & 7.91 & 1.38e+01 & 1652 & 32.31 & 7.90 \\
&iVPGSA   & 1.38e+00 & 45 (398) & 70.91 & 7.91 & 1.38e+01 & 12 (155) & 26.66 & 7.90 \\
\hline
\multirow{2}{*}{(2000,20000,400)}
&PGSA\_BE & 1.61e+00 & 5734 & 196.09 & 13.87 & 1.61e+01 & 1771 & 60.54 & 13.85 \\
&iVPGSA   & 1.61e+00 & 34 (427) & 145.19 & 13.87 & 1.61e+01 & 10 (179) & 57.18 & 13.85 \\
\hline
\multirow{2}{*}{(2500,25000,500)}
&PGSA\_BE & 1.79e+00 & 6574 & 353.06 & 21.80 & 1.79e+01 & 1984 & 106.32 & 21.73 \\
&iVPGSA   & 1.79e+00 & 30 (478) & 269.82 & 21.80 & 1.79e+01 & 12 (199) & 105.08 & 21.73 \\
\hline
\end{tabular}
}
\end{table}

\begin{figure}[ht]
\centering
\subfigure[\texttt{mpg7}, where $A$ is of size $392\times3432$ with $\lambda_{\max}(A^{\top}A)\approx1.28\times10^4$]{
\includegraphics[width=7cm]{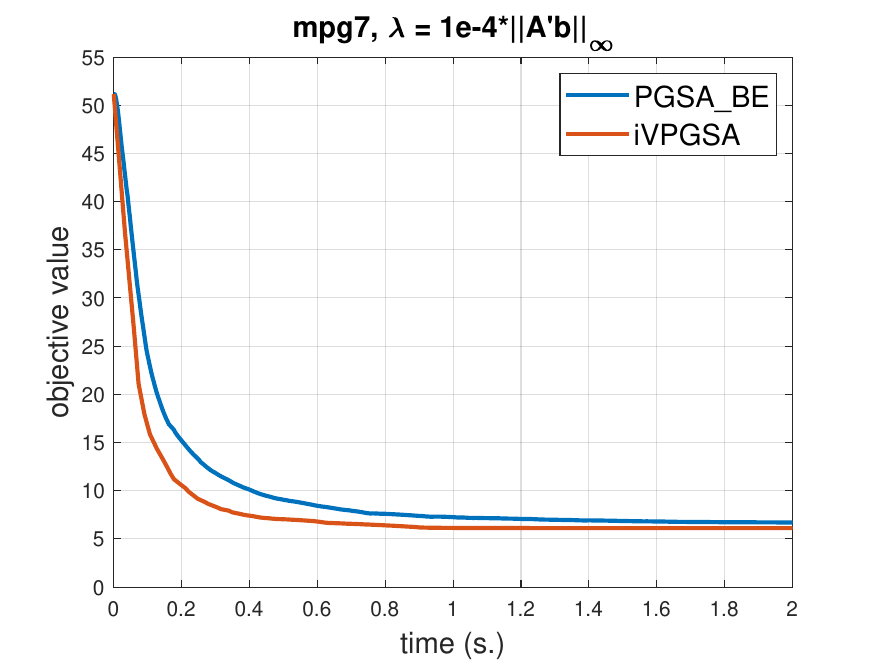}
\includegraphics[width=7cm]{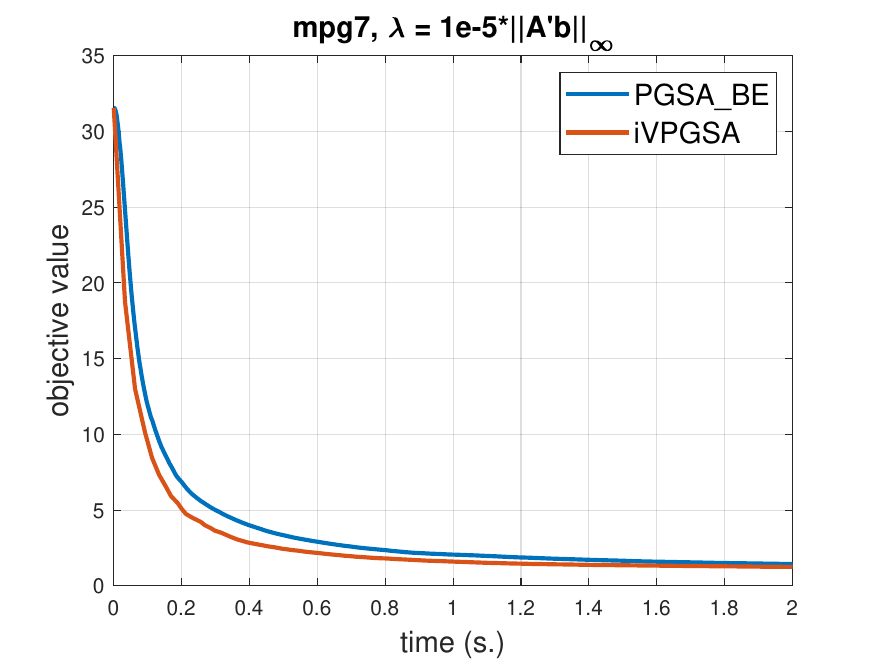}
}
\subfigure[\texttt{gisette}, where $A$ is of size $1000\times4971$ with $\lambda_{\max}(A^{\top}A)\approx3.36\times10^6$]{
\includegraphics[width=7cm]{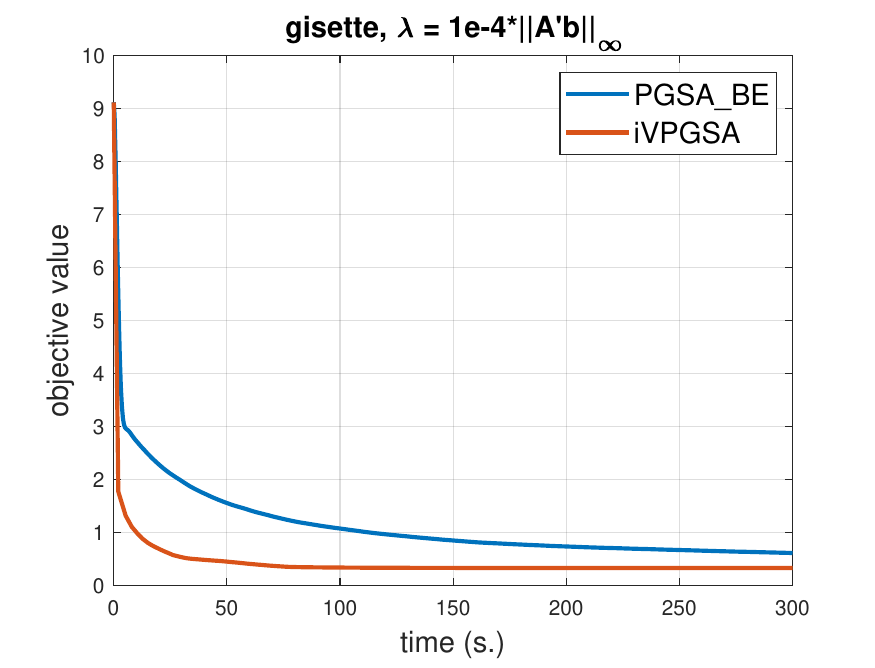}
\includegraphics[width=7cm]{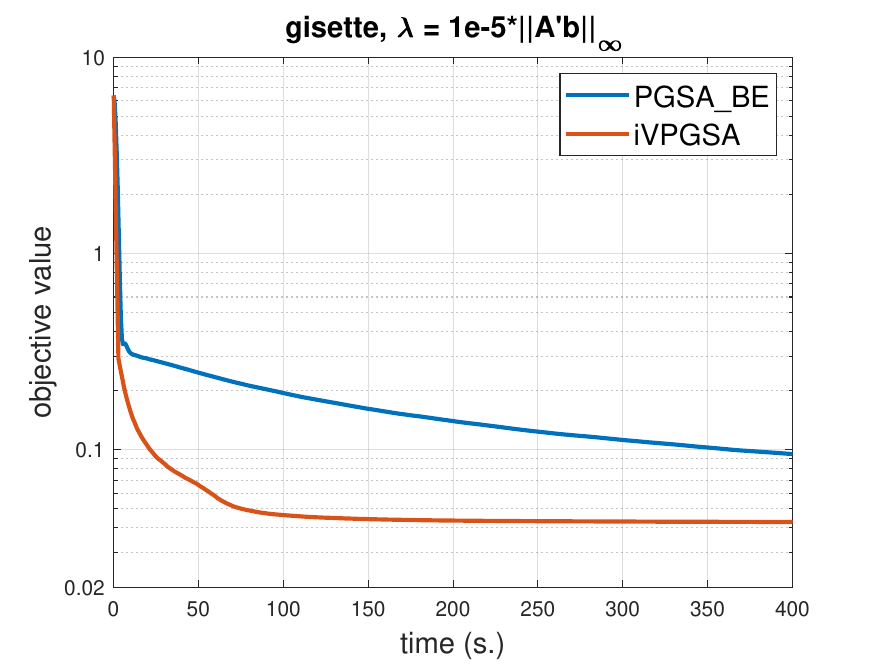}
}
\caption{Numerical results of PGSA\_BE and iVPGSA for solving the $\ell_1/\ell_2$ Lasso problem on \texttt{mpg7} and \texttt{gisette} from the UCI data repository.}\label{Fig:L1oL2Lasso}
\end{figure}

\begin{remark}[\textbf{Comments on difference between iVPGSA and PGSA\_BE}]\label{rek-diff}
By allowing the inexact minimization of the subproblem, our iVPGSA has more flexibility in practical implementations when it is applied to a specific problem. Indeed, for problem \eqref{prob_l12lasso}, the PGSA\_BE developed in \citep{lszz2022proximal} is applied to the reformulation shown on the left below so that the subproblem admits an easy-to-compute solution, while our iVPGSA can be applied to the reformulation shown on the right below, as implemented in our experiments.
\begin{equation*}
{\small\begin{aligned}
\quad
&\hspace{3cm}{\text{PGSA\_BE}} \\
&\hspace{3.5cm}\Downarrow \\
&\min\limits_{\bm{x}\in\Omega}~
\frac{\overbrace{\lambda\|\bm{x}\|_1
+ \iota_{[\bm{\alpha},\,\bm{\beta}]}(\bm{x})}^{f(\bm{x})} ~~
+ \,\overbrace{ \frac{1}{2}\|A\bm{x}-\bm{b}\|^2}^{h(\bm{x})}}{\underbrace{\|\bm{x}\|}_{g(\bm{x})}}
\end{aligned}
\quad~~ vs ~~\quad
\begin{aligned}
&\hspace{3.7cm}{\text{iVPGSA}} \\
&\hspace{4.1cm}\Downarrow \\
&\min\limits_{\bm{x}\in\Omega}~
\frac{\overbrace{\lambda\|\bm{x}\|_1
+ \iota_{[\bm{\alpha},\,\bm{\beta}]}(\bm{x})
+ \frac{1}{2}\|A\bm{x}-\bm{b}\|^2}^{f(\bm{x})} ~~
+ \,\overbrace{0}^{h(\bm{x})}}{\underbrace{\|\bm{x}\|}_{g(\bm{x})}}.
\end{aligned}}	
\end{equation*}
These two reformulations lead to entirely distinct algorithmic frameworks with significantly different numerical performances. Since PGSA\_BE also approximates the smooth convex part $\frac{1}{2}\|A\bm{x}-\bm{b}\|^2$ by a quadratic majorant, its performance could be heavily influenced by the Lipschitz constant $L_h:=\lambda_{\max}(A^{\top}A)$. This sensitivity to the Lipschitz constant is indeed a well-recognized limitation of first-order methods and may lead to the inferior performance of PGSA\_BE for solving large-scale problems, as observed in Table \ref{Tab:l1ol2lasso} and Figure \ref{Fig:L1oL2Lasso}. This is because, for a randomly generated matrix $A$, the largest eigenvalue of $A^{\top}A$ tends to increase with the problem size, and for the data sets \texttt{mpg7} and \texttt{gisette}, the largest eigenvalues of $A^{\top}A$ are around $1.28\times10^4$ and $3.36\times10^6$, respectively. In contrast, our iVPGSA can sidestep the majorization of the least squares term and choose instead to solve the subproblem via a highly efficient second-order {\sc Ssn} method. The results in Table \ref{Tab:l1ol2lasso} and Figure \ref{Fig:L1oL2Lasso} illustrate the promising performance of this implementation strategy. Moreover, we should emphasize that these achievements are made possible by allowing the inexact minimization of the subproblem. This highlights the motivation for developing an efficient inexact algorithmic framework in this work.
\end{remark}

\subsection{The constrained $\ell_1/\ell_2$ sparse optimization problem}\label{sec-num-l12cso}

In this section, we consider the constrained $\ell_1/\ell_2$ sparse optimization problem with Gaussian noise, as studied in \citep{zyp2021analysis}:
\begin{equation}\label{prob_l12cso}
\min\limits_{\bm{x}\in\mathbb{R}^n}
~~\frac{\|\bm{x}\|_1}{\|\bm{x}\|}
\qquad \mathrm{s.t.} \qquad
\|A\bm{x}-\bm{b}\|\leq\sigma,
\end{equation}
where $\sigma>0$, $A\in\mathbb{R}^{m\times n}$ has full row rank, and $\bm{b}\in\mathbb{R}^m$ satisfies $\|\bm{b}\|>\sigma$.
To apply our iVPGSA in Algorithm \ref{algo-iVPGSA} for solving problem \eqref{prob_l12cso}, we consider the following reformulation:
\begin{equation*}
\min\limits_{\bm{x}\in\Omega}~~F_{\text{con}}(\bm{x}):=
\frac{\overbrace{\|\bm{x}\|_1+\iota_{\sigma}(A\bm{x}-\bm{b})}^{f(\bm{x})}
~~+\, \overbrace{0}^{h(\bm{x})}}{\underbrace{\|\bm{x}\|}_{g(\bm{x})}},
\end{equation*}
where $\Omega:=\mathbb{R}^n\backslash\{0\}$ and $\iota_{\sigma}$ denotes the indicator function on the set $\left\{\bm{y}\in\mathbb{R}^m:\|\bm{y}\|\leq\sigma\right\}$. It is easy to see that the above reformulation conforms to the form of \eqref{mainpro} and satisfies Assumptions \ref{assumA}1--4. The level-boundedness requirement in Assumption \ref{assumA}5 can be easily ensured by imposing a box constraint $\|\bm{x}\|_{\infty}\leq M$ with a sufficiently large $M>0$ in problem \eqref{prob_l12cso}. However, to remain consistent with the formulation studied in \citep{zyp2021analysis} and facilitate the direct application of the algorithm developed therein, we omit this box constraint for simplicity. Moreover, to apply our iVPGSA, we choose $H_k=\gamma_k(I+A^{\top}A)$ and the associated subproblem at the $k$-th iteration ($k\geq0$) is given as follows:
\begin{equation}\label{subprob_l12cso}
\min\limits_{\bm{x}\in\mathbb{R}^n}~~\|\bm{x}\|_1
+ \iota_{\sigma}(A\bm{x}-\bm{b})
- \langle c_k\bm{y}^{k},\,\bm{x}-\bm{x}^k\rangle
+ \frac{\gamma_k}{2}\|\bm{x}-\bm{x}^{k}\|^2
+ \frac{\gamma_k}{2}\|A\bm{x} - A\bm{x}^{k}\|^2,
\end{equation}
where $c_k=\frac{f(\bm{x}^k)+h(\bm{x}^k)}{g(\bm{x}^k)}$ and $\bm{y}^k\in\partial\|\bm{x}^k\|$.

In this example, the metric $H_k=\gamma_k(I+A^{\top}A)$ exploits the linear-operator structure and leads to a structured dual semismooth equation derived below. Then, the subproblem \eqref{subprob_l12cso} can be solved by the {\sc Ssn} method up to the accuracy required by the verifiable criterion \eqref{checkcond_l12cos}. Overall, the variable metric yields a subproblem suitable for a second-order solver, while the error criterion controls the accuracy required by the outer iVPGSA iteration.

\subsubsection{A dual semi-smooth Newton method for solving the subproblem \eqref{subprob_l12cso}}\label{sec_l12cso_ssn}

Similar to the implementations as described in subsection \ref{sec_l12lasso_ssn}, we discuss how to efficiently solve \eqref{subprob_l12cso} via a dual {\sc Ssn} method to find a point $\bm{x}^{k+1}$ associated with an error pair $(\Delta^k,\,\delta_k)$ satisfying the inexact condition \eqref{inexcondH-x} and the error criterion \eqref{stopcritH-x}. Specifically, we are dedicated to solving the following equivalent reformulation:
\begin{equation}\label{subprobref_l12cos}
\begin{aligned}
&\min\limits_{\bm{x}\in\mathbb{R}^n,\,\bm{y}\in\mathbb{R}^m}~
\|\bm{x}\|_1
{\textstyle + \frac{\gamma_k}{2}\left\|\bm{x}-\big(\bm{x}^k+\gamma_k^{-1}c_k\bm{y}^{k}\big)\right\|^2
+ \iota_{\sigma}(\bm{y})
+ \frac{\gamma_k}{2}\left\|\bm{y} - \big(A\bm{x}^k-\bm{b}\big)\right\|^2}  \\
&\quad~~\,\mathrm{s.t.}\qquad A\bm{x} - \bm{y} = \bm{b}.
\end{aligned}
\end{equation}
For notational simplicity, let $\bm{s}^k:=\bm{x}^k+\gamma_k^{-1}c_k\bm{y}^{k}$ and $\bm{b}^k:=A\bm{x}^k-\bm{b}$. By some manipulations, it can be shown that the dual problem of \eqref{subprobref_l12cos} admits the following equivalent minimization form:
\begin{equation}\label{dualprobleml12con}
\hspace{-2mm}
\min_{\bm{z}\in\mathbb{R}^m}~
\left\{\begin{aligned}
\Psi_k^{\text{con}}(\bm{z})
&:= {\textstyle\langle\bm{z}, \bm{b}\rangle
+\frac{\gamma_k}{2}\left\|\bm{s}^k-\gamma_k^{-1} A^{\top} \bm{z}\right\|^2
-\left\|\texttt{prox}_{\gamma_k^{-1}\|\cdot\|_1}\left(\bm{s}^k-\gamma_k^{-1} A^{\top}\bm{z}\right)\right\|_1}  \\
&\qquad {\textstyle-\,\frac{\gamma_k}{2}\left\|\texttt{prox}_{\gamma_k^{-1}\|\cdot\|_1}\left(\bm{s}^k-\gamma_k^{-1} A^{\top} \bm{z}\right)
-\left(\bm{s}^k-\gamma_k^{-1}A^{\top}\bm{z}\right)\right\|^2} \\
&\qquad {\textstyle+\,\frac{\gamma_k}{2}\left\|\bm{b}^k+\gamma_k^{-1}\bm{z}\right\|^2
-\frac{\gamma_k}{2}\left\|\Pi_{\sigma}\left(\bm{b}^k+\gamma_k^{-1}\bm{z}\right)
-\left(\bm{b}^k+\gamma_k^{-1}\bm{z}\right)\right\|^2}   \\[3pt]
&\qquad {\textstyle-\,\frac{\gamma_k}{2}\|\bm{s}^k\|^2
-\frac{\gamma_k}{2}\|\bm{b}^k\|^2}
\end{aligned}\right\},
\end{equation}
where $\bm{z}\in\mathbb{R}^m$ is the dual variable and $\Pi_{\sigma}$ is the projection operator over $\left\{\bm{y} \in \mathbb{R}^m:\|\bm{y}\|\leq\sigma\right\}$. From the property of the Moreau envelope of $\gamma_k^{-1}\|\cdot\|_1$ and $\gamma_k^{-1}\iota_{\sigma}$ (see, e.g.,
\citep[Proposition 12.29]{bc2011convex}), we see that $\Psi_k^{\text{con}}$ is convex and continuously differentiable with the gradient:
\begin{equation*}
\nabla \Psi_k^{\text{con}}(\bm{z})
= -A \texttt{prox}_{\gamma_k^{-1}\|\cdot\|_1}{\textstyle\left(\bm{s}^k-\gamma_k^{-1} A^{\top} \bm{z}\right)}
+ \Pi_{\sigma}{\textstyle\left(\bm{b}^k+\gamma_k^{-1} \bm{z}\right)} + \bm{b}.
\end{equation*}
Thus, from the first-order optimality condition, solving problem \eqref{dualprobleml12con} is equivalent to solving the following non-smooth equation:
\begin{equation}\label{subdualequa_l12con}
\nabla \Psi_k^{\text{con}}(\bm{z})=0.
\end{equation}
Similar to the implementations in subsection \ref{sec_l12lasso_ssn}, to apply the {\sc Ssn} method, we define a multifunction $\widehat{\partial}^2\Psi_k^{\text{con}}:\mathbb{R}^m \rightrightarrows \mathbb{R}^{m \times m}$ as follows:
\begin{equation*}
\widehat{\partial}^2\Psi_k^{\text{con}}(\bm{z}) := \gamma_k^{-1}A\partial\texttt{prox}_{\gamma_k^{-1}\|\cdot\|_1}
{\textstyle\left(\bm{s}^k-\gamma_k^{-1}A^{\top}\bm{z}\right)}A^{\top}
+ \gamma_k^{-1}\partial\Pi_{\sigma}
{\textstyle\left(\bm{b}^k+\gamma_k^{-1} \bm{z}\right)},
\end{equation*}
where $\partial\texttt{prox}_{\gamma_k^{-1}\|\cdot\|_1}\!
\left(\bm{s}^k-\gamma_k^{-1}A^{\top}\bm{z}\right)$ is the Clarke subdifferential of the Lipschitz continuous mapping $\texttt{prox}_{\gamma_k^{-1}\|\cdot\|_1}(\cdot)$ at $\bm{s}^k-\gamma_k^{-1}A^{\top}\bm{z}$, defined as follows:
\begin{equation*}
\partial\texttt{prox}_{\alpha^{-1}\|\cdot\|_1}(\bm{u})
:= \left\{\mathrm{Diag}(\bm{d}) \,:\, \bm{d}\in\mathbb{R}^n, ~d_i\in
\left\{\begin{aligned}
&\{1\}, && \mathrm{if}~~|u_i|>\alpha^{-1}, \\[3pt]
&[0,1], && \mathrm{if}~~|u_i|=\alpha^{-1}, \\[3pt]
&\{0\}, && \mathrm{if}~~|u_i|<\alpha^{-1},
\end{aligned}\right.~~\right\},
\end{equation*}
and $\partial\Pi_{\sigma}\left(\bm{b}^k+\gamma_k^{-1}\bm{z}\right)$ is the Clarke subdifferential of the Lipschitz continuous mapping $\Pi_{\sigma}(\cdot)$ at $\bm{b}^k+\gamma_k^{-1}\bm{z}$, defined as follows (see, e.g., \citep[Remark 3.1]{zzst2020efficient}):
\begin{equation*}
\partial\Pi_{\sigma}(\bm{u})
=\left\{\begin{aligned}
&\{I\}, &&~~\mathrm{if}~~\|\bm{u}\| < \sigma, \\
&\Big\{I - {\textstyle\frac{t}{\sigma^2}\bm{u}\bm{u}^{\top}} \mid t\in[0,1]\Big\}, &&~~ \mathrm{if}~~\|\bm{u}\| = \sigma, \\
&\Big\{ {\textstyle\frac{\sigma}{\|\bm{u}\|}\big(I - \frac{1}{\|\bm{u}\|^2}\bm{u}\bm{u}^{\top}\big)} \Big\}, &&~~ \mathrm{if}~~\|\bm{u}\| > \sigma.
\end{aligned}\right.
\end{equation*}
Note that the elements in $\widehat{\partial}^2\Psi_k^{\text{con}}(\bm{z})$ may only be positive semidefinite. We then need to incorporate a small adaptive regularization term when solving the linear system in the {\sc Ssn} method. We present the whole iterative framework in Algorithm \ref{algo:SSNcon} and refer readers to \citep[Theorems 3.4 and 3.5]{zst2010newton} for its convergence results.

\begin{algorithm}[htb!]
\caption{A semi-smooth Newton ({\sc Ssn}) method for solving equation \eqref{subdualequa_l12con}}\label{algo:SSNcon}
 	
\textbf{Initialization:} Choose $\bar{\eta}\in(0,1)$, $\gamma\in(0,1]$, $\mu\in(0,1/2)$, $\delta\in(0,1)$, $\tau_1,\tau_2\in(0,1)$, and an initial point $\bm{z}^{k,0}\in\mathbb{R}^m$. Set $t=0$. Repeat until a termination criterion is met. \vspace{-1mm}
\begin{itemize}[leftmargin=1.6cm]
\item[\textbf{Step 1.}] Compute $\nabla\Psi_k^{\text{con}}(\bm{z}^{k,t})$, select an element $H^{k,t}\in\widehat{\partial}^2\Psi_k^{\text{con}}(\bm{z}^{k,t})$, and let $\nu_t:=\tau_1\min\big\{\tau_2,\,\|\nabla\Psi_k^{\text{con}}(\bm{z}^{k,t})\|\big\}$. Solve the linear system $\left(H^{k,t}+\nu_t I_m\right)\bm{d} = -\nabla\Psi_k^{\text{con}}(\bm{z}^{k,t})$ nearly exactly by the (sparse) Cholesky factorization with forward and backward substitutions, \textit{or} approximately by the preconditioned conjugate gradient method to find $\bm{d}^{k,t}$ such that $\big\|\left(H^{k,t}+\nu_t I_m\right)\bm{d}^{k,t} + \nabla\Psi_k^{\text{con}}(\bm{z}^{k,t})\big\|
    \leq \min\big(\bar{\eta}, \,\|\nabla\Psi_k^{\text{con}}(\bm{z}^{k,t})\|^{1+\gamma}\big)$.
	
\item[\textbf{Step 2.}] (\textbf{Inexact line search}) Find a step size $\alpha_t:=\delta^{i_t}$, where $i_t$ is the smallest nonnegative integer $i$ for which $\Psi_k^{\text{con}}(\bm{z}^{k,t} + \delta^i\bm{d}^{k,t})
    \leq \Psi_k^{\text{con}}(\bm{z}^{k,t}) + \mu \delta^{i}\langle\nabla\Psi_k^{\text{con}}(\bm{z}^{k,t}), \,\bm{d}^{k,t}\rangle$.

\item[\textbf{Step 3.}] Set $\bm{z}^{k,t+1} = \bm{z}^{k,t} + \alpha_t\bm{d}^{k,t}$, $t=t+1$, and go to \textbf{Step 1}.
\end{itemize}
\end{algorithm}

We next show that our inexact condition \eqref{inexcondH-x} and error criterion \eqref{stopcritH-x} can be achieved through some appropriate manipulations based on the dual sequence generated by the {\sc Ssn} method in Algorithm \ref{algo:SSNcon}. To this end, we first assume that a strictly feasible point $\bm{x}^{\text{feas}}$ satisfying $\|A\bm{x}^{\text{feas}}-\bm{b}\|<\sigma$ is available on hand. Indeed, such a point can be simply obtained by setting $\bm{x}^{\text{feas}}:=A^\dag\bm{b}$. Then, at the $k$-th iteration, we apply the {\sc Ssn} method for solving equation \eqref{subdualequa_l12con}, which generates a dual sequence $\{\bm{z}^{k,t}\}$. Let
\begin{equation*}
\bm{w}^{k,t}:=\texttt{prox}_{\gamma_k^{-1}\|\cdot\|_1}
\big(\bm{x}^k+\gamma_k^{-1}c_k\bm{y}^k-\gamma_k^{-1}A^{\top}\bm{z}^{k,t}\big)
\quad \mbox{and} \quad
\bm{e}^{k,t}:= \nabla\Psi_k^{\text{con}}\big(\bm{z}^{k,t}\big).
\end{equation*}
Note that the terminating approximate primal solution $\bm{w}^{k,t}$ may not be exactly feasible to the subproblem \eqref{subprob_l12cso} and thus the inexact condition \eqref{inexcondH-x} and error criterion \eqref{stopcritH-x} cannot be verified at $\bm{w}^{k,t}$ in this scenario. Therefore, we further adopt a retraction strategy with the aid of the strictly feasible point $\bm{x}^{\text{feas}}$. Specifically, we define
\begin{equation*}
\widetilde{\bm{w}}^{k,t}:=\rho_{k,t}\bm{w}^{k,t} + (1-\rho_{k,t})\bm{x}^{\text{feas}}
~~\text{with}~~
\rho_{k,t}:=\left\{
\begin{aligned}
&1, &&\|A\bm{w}^{k,t}-\bm{b}\|\leq\sigma,\\
&{\textstyle\frac{\sigma-\|A\bm{x}^{\text{feas}}-\bm{b}\|}{\|A\bm{w}^{k,t}-\bm{b}\|-\|A\bm{x}^{\text{feas}}-\bm{b}\|}}, &&\|A\bm{w}^{k,t}-\bm{b}\|>\sigma.\\
\end{aligned}\right.
\end{equation*}
Then, one can easily verify that $\widetilde{\bm{w}}^{k,t}$ is a feasible point of the subproblem \eqref{subprob_l12cso}, namely, it satisfies that $\|A\widetilde{\bm{w}}^{k,t}-\bm{b}\|\leq\sigma$. Thus, $\widetilde{\bm{w}}^{k,t}$ is a feasible candidate that can be used in the verifications of the inexact condition \eqref{inexcondH-x} and error criterion \eqref{stopcritH-x}. Indeed, we have the following proposition, whose proof is relegated to Appendix \ref{apd-pro-l12cos}.

\begin{proposition}\label{pro-scnew-l12con}
Let
\begin{equation*}
\begin{aligned}
\Delta^{k,t}
&:= -\gamma_kA^{\top}\bm{e}^{k,t}
+ \gamma_k(\widetilde{\bm{w}}^{k,t}-\bm{w}^{k,t})
+ \gamma_k A^{\top}A(\widetilde{\bm{w}}^{k,t}-\bm{w}^{k,t}), \\
\bm{d}_1^{k,t}
&:= \gamma_k\left[\big({\bm{x}}^{k}+\gamma_k^{-1} c_k\bm{y}^k-\gamma_k^{-1}A^{\top} {\bm{z}}^{k,t}\big)-\bm{w}^{k,t}\right], \\
\bm{d}_2^{k,t}
&:= \gamma_k\left[\big(A{\bm{x}}^k-\bm{b}+\gamma_k^{-1}\bm{z}^{k,t}\big)
- \Pi_{\sigma}\big(A\bm{x}^k-\bm{b}+\gamma_k^{-1}\bm{z}^{k,t}\big)\right], \\
\delta^1_{k,t}
&:= \|\widetilde{\bm{w}}^{k,t}\|_1 - \|\bm{w}^{k,t}\|_1
- \langle{\bm{d}}_1^{k,t},\,\widetilde{\bm{w}}^{k,t}-\bm{w}^{k,t}\rangle, \\
\delta_{k,t}^2
&:= \big|\langle \bm{e}^{k,t} - A(\widetilde{\bm{w}}^{k,t}-\bm{w}^{k,t}), \,\bm{d}_2^{k,t} \rangle\big|.
\end{aligned}
\end{equation*}
Then, we have that
\begin{equation*}
\Delta^{k,t}
\in\partial_{\delta_{k,t}^1+\delta_{k,t}^2}
\Big(\|\cdot\|_1+\iota_{\sigma}(A\cdot-\bm{b})\Big)(\widetilde{\bm{w}}^{k,t})
- c_k\bm{y}^k + \gamma_k\big(\widetilde{\bm{w}}^{k,t}-\bm{x}^{k}\big)
+ \gamma_kA^{\top}A\big(\widetilde{\bm{w}}^{k,t}-\bm{x}^{k}\big).
\end{equation*}
If $(\widetilde{\bm{w}}^{k,t},\Delta^{k,t},\delta^1_{k,t},\delta^2_{k,t})$ satisfies
\begin{equation}\label{checkcond_l12cos}
\|\Delta^{k,t}\|^2
+ |\langle\Delta^{k,t}, \,\widetilde{\bm{w}}^{k,t}-\bm{x}^{k}\rangle|
+ \delta^1_{k,t} + \delta^2_{k,t}
\leq \varepsilon_k\,\|\widetilde{\bm{w}}^{k,t}\|,
\end{equation}
then the error criterion \eqref{stopcritH-x} holds for $\bm{x}^{k+1}:=\widetilde{\bm{w}}^{k,t}$, $\Delta^k:=\Delta^{k,t}$, and $\delta_k:=\delta^1_{k,t} + \delta^2_{k,t}$.
\end{proposition}

Similar to discussions at the end of subsection \ref{sec_l12lasso_ssn}, one can also see from Proposition \ref{pro-scnew-l12con} that our inexact condition \eqref{inexcondH-x} and error criterion \eqref{stopcritH-x} are indeed verifiable and can be satisfied as long as inequality \eqref{checkcond_l12cos} holds. Indeed, when the dual sequence $\{\bm{z}^{k,t}\}$ generated by the {\sc Ssn} method is convergent under proper conditions (see \citep[Theorems 3.4 and 3.5]{zst2010newton}), we have that $\bm{e}^{k,t}\to0$, $\{\bm{w}^{k,t}\}$ is convergent to the optimal solution $\bm{x}^{k,*}$ of the $k$-th subproblem \eqref{subprob_l12cso}, and $\widetilde{\bm{w}}^{k,t}-\bm{w}^{k,t}\to0$. Consequently, the left-hand-side term in \eqref{checkcond_l12cos} approaches zero, i.e., $\|\Delta^{k,t}\|^2 + |\langle\Delta^{k,t}, \,\widetilde{\bm{w}}^{k,t}-\bm{x}^{k}\rangle|
+ \delta^1_{k,t} + \delta^2_{k,t}\to0$. On the other hand, based on the discussions following \eqref{suffdessubopt}, under the nonnegativity of $f+h$ on $\mathrm{dom}\,f$ (by Assumption \ref{assumA}4) and $H_k\succ\frac{L_h}{2}I_n$, if $\bm{x}^k\in\Omega\cap\mathrm{dom}\,f$ is not the solution of the $k$-th subproblem and $c_k>0$, we have that $\|\bm{x}^{k,*}\|>0$. Thus, the right-hand-side term in \eqref{checkcond_l12cos} converges to a positive value. Therefore, \eqref{checkcond_l12cos} must hold after finitely many iterations.

\subsubsection{Comparison results}

We evaluate the performance of our iVPGSA, using the same parameter settings as described in the first paragraph of subsection \ref{sec-comp-l12lasso} with additional parameters $\tau_1$, $\tau_2$ for the {\sc Ssn} in Algorithm \ref{algo:SSNcon} being set to $\tau_1=0.99$ and $\tau_2=10^{-6}$. We also compare our method with a moving-balls-approximation based algorithm (denoted by MBA) proposed in \citep{zyp2021analysis}. For MBA, we use default parameter settings provided in the accompanying codes\footnote{The {\sc Matlab} codes of the MBA for solving \eqref{prob_l12cso} are available at \url{https://www.polyu.edu.hk/ama/profile/pong/MBA_l1vl2/}.}.

We initialize two methods with a point $\bm{x}^0$ generated as follows: we first apply the popular solver SPGL1 (version 2.1) \citep{vf2009probing} for solving the classical constrained $\ell_1$ sparse optimization problem: $\min\big\{\|\bm{x}\|_1:\|A \bm{x}-\bm{b}\|\leq\sigma\big\}$, using 200 iterations to obtain an approximate solution $\bm{x}_{\text{spgl1}}$. Since $\bm{x}_{\text{spgl1}}$ may violate the constraint slightly, we then adopt the same retraction strategy as described before Proposition \ref{pro-scnew-l12con} to retract $\bm{x}_{\text{spgl1}}$ and then set the resulting point as $\bm{x}^0$. Additionally, we terminate both methods when the following holds for 3 consecutive iterations:
\begin{equation*}
\frac{|F_{\text{con}}(\bm{x}^k)-F_{\text{con}}(\bm{x}^{k-1})|}{1+|F_{\text{con}}(\bm{x}^k)|} < 10^{-7}.
\end{equation*}
We also terminate our iVPGSA when the total number of {\sc Ssn} iterations reaches 1000, and terminate the MBA when its number of iterations reaches 30000.

In the following experiments, we set $(m,n,s)=(100i,1000i,20i)$ for $i\in\{5,10,15,20,25\}$. For each triple $(m, n, s)$, a random instance is generated using the same way as in subsection \ref{sec-comp-l12lasso}. Moreover, we choose $\sigma=\texttt{nf}\cdot\|0.01\cdot\widehat{\bm{n}}\|$ with $\texttt{nf}\in\{1.2,\,2\}$ (the first choice is also used in \citep[Section 7.3]{zyp2021analysis}), and present the average computational results for each triple $(m, n, s)$ across 10 instances in Table \ref{Table:l12cos}. From the results, one can see that our iVPGSA significantly outperforms MBA in both objective function value and computational efficiency.

Finally, we test two methods with two instances $(A,\bm{b})$ obtained from the data sets \texttt{mpg7} and \texttt{gisette} in the UCI data repository. For each data set, we choose $\sigma=\sigma_c\|\bm{b}\|$ with $\sigma_c\in\{0.05, 0.1\}$. The numerical results are reported in Figure \ref{Fig:l12cso}, where we plot the objective value $F_{\text{con}}(\bm{x}^k)$ against the computational time. The results further demonstrate the encouraging performance of our iVPGSA on real data sets.

\begin{table}[ht]
\caption{The average computational results for each triple $(m, n, s)$ from 10 instances on the constrained $\ell_1/\ell_2$ sparse optimization problem with $\sigma=\texttt{nf}\cdot\|0.01\cdot\widehat{\bm{n}}\|$ with $\texttt{nf}\in\{1.2,\,2\}$. In the table, ``\texttt{obj}" denotes the objective function value, ``\texttt{feas}" denotes the violation of the constraint $\|A\bm{x}^k-\bm{b}\|-\sigma$, ``\texttt{iter}" denotes the number of iterations (the total number of the {\sc Ssn} iterations in iVPGSA is also reported in the bracket), ``\texttt{time}" denotes the computational time, and ``\texttt{t0}" denotes the computational time used to obtain an initial point by SPGL1 using 200 iterations.}\label{Table:l12cos}
\centering \tabcolsep 6pt
\scalebox{0.9}
{\renewcommand\arraystretch{1}
\begin{tabular}[!]{l|c|l|llllll}
\toprule
\texttt{nf} & $(m,n,s)$ & {\tt method} & \texttt{obj} & \texttt{feas}
& \texttt{iter} & \texttt{time} & \texttt{t0} \\
\midrule
\multirow{10}{*}{1.2}&(500,5000,100)
&MBA     & 9.27e+00 & -6.27e-05 & 14504 & 29.57 & 0.11 \\
&&iVPGSA & 9.14e+00 & -9.99e-05 & 46 (540) & 3.19 & 0.11 \\[3pt]
&(1000,10000,200)
&MBA     & 1.38e+01 & -2.15e-04 & 15465 & 192.74 & 1.31 \\
&&iVPGSA & 1.30e+01 & 8.00e-09 & 48 (557) & 17.88 & 1.31 \\[3pt]
&(1500,15000,300)
&MBA     & 1.90e+01 & -4.12e-03 & 21408 & 594.46 & 3.42 \\
&&iVPGSA & 1.68e+01 & -9.03e-06 & 62 (759) & 58.75 & 3.42 \\[3pt]
&(2000,20000,400)
&MBA     & 2.18e+01 & -5.81e-03 & 18584 & 852.60 & 6.09 \\
&&iVPGSA & 1.87e+01 & -6.94e-06 & 57 (685) & 106.72 & 6.09 \\[3pt]
&(2500,25000,500)
&MBA     & 2.70e+01 & -1.26e-02 & 24237 & 1681.98 & 9.56 \\
&&iVPGSA & 2.22e+01 & -2.53e-05 & 71 (838) & 220.74 & 9.56 \\[3pt]
\midrule
\multirow{10}{*}{2}&(500,5000,100)
&MBA     & 9.19e+00 & -1.31e-04 & 10849 & 22.35 & 0.11 \\
&&iVPGSA & 8.95e+00 & 8.26e-09 & 49 (443) & 2.60 & 0.11 \\[3pt]
&(1000,10000,200)
&MBA     & 1.21e+01 & -3.60e-04 & 6397 & 78.61 & 1.32 \\
&&iVPGSA & 1.20e+01 & -1.67e-05 & 26 (270) & 8.93 & 1.32 \\[3pt]
&(1500,15000,300)
&MBA     & 1.48e+01 & -2.46e-03 & 6412 & 181.31 & 3.24 \\
&&iVPGSA & 1.47e+01 & 1.94e-09 & 25 (304) & 25.64 & 3.24 \\[3pt]
&(2000,20000,400)
&MBA     & 1.96e+01 & -2.05e-03 & 16111 & 741.80 & 6.14 \\
&&iVPGSA & 1.83e+01 & -8.12e-05 & 48 (601) & 91.73 & 6.14 \\[3pt]
&(2500,25000,500)
&MBA     & 2.24e+01 & -3.40e-03 & 15361 & 1059.93 & 9.68 \\
&&iVPGSA & 2.05e+01 & -1.74e-05 & 49 (589) & 148.76 & 9.68 \\[3pt]
\bottomrule
\end{tabular}
}
\end{table}

\begin{figure}[ht]
\centering
\subfigure[\texttt{mpg7}, where $A$ is of size $392\times3432$ with $\lambda_{\max}(A^{\top}A)\approx1.28\times10^4$]{
\includegraphics[width=7cm]{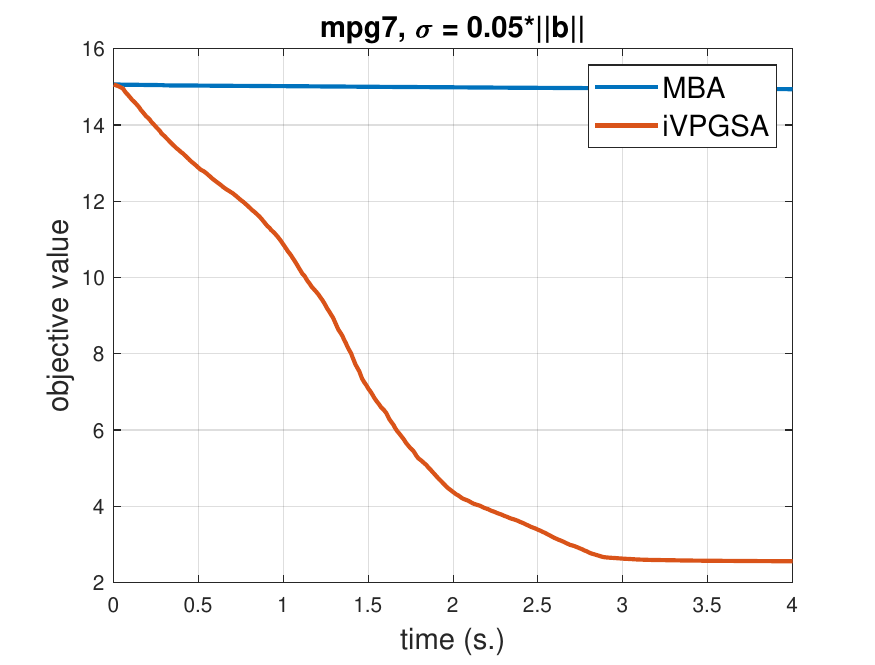}
\includegraphics[width=7cm]{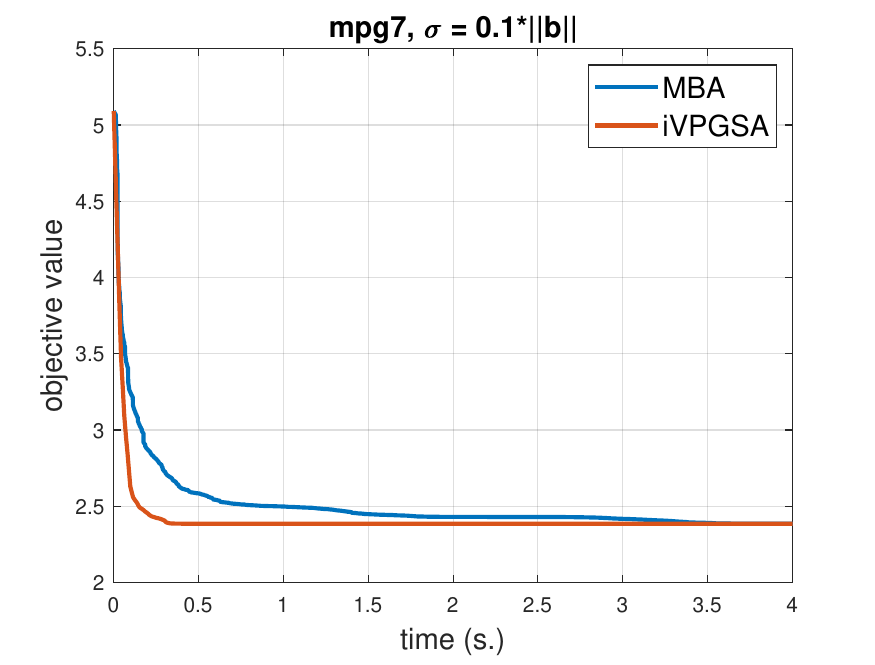}
}
\subfigure[\texttt{gisette}, where $A$ is of size $1000\times4971$ with $\lambda_{\max}(A^{\top}A)\approx3.36\times10^6$]{
\includegraphics[width=7cm]{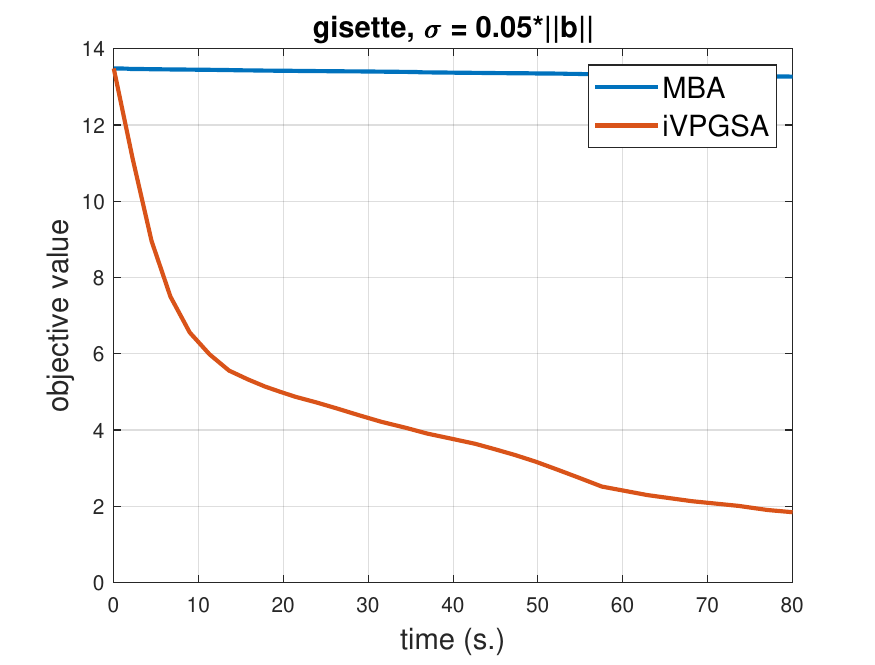}
\includegraphics[width=7cm]{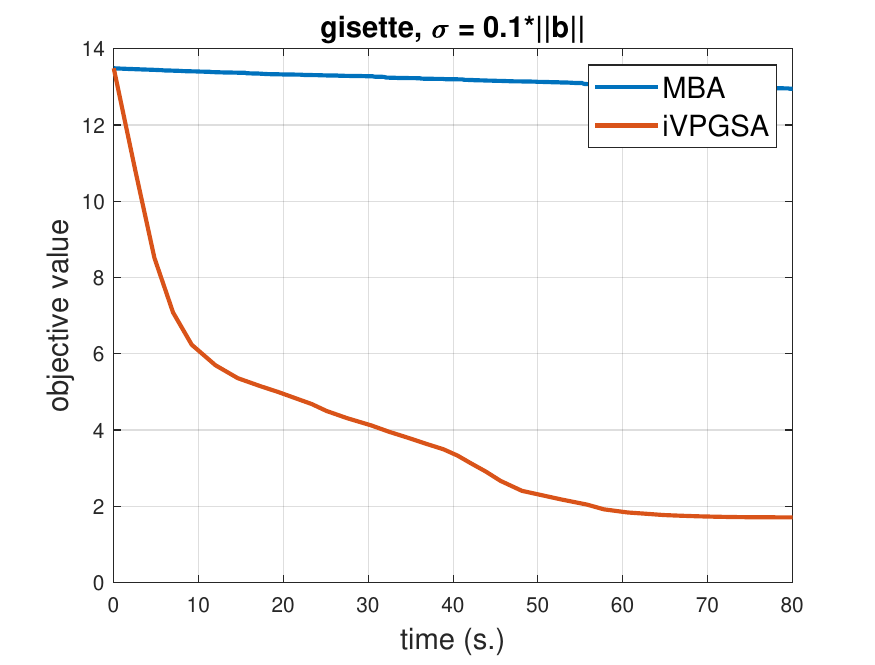}
}
\caption{Numerical results of MBA and iVPGSA for solving the constrained $\ell_1/\ell_2$ sparse optimization problem on \texttt{mpg7} and \texttt{gisette} from the UCI data repository.}\label{Fig:l12cso}
\end{figure}

\section{Conclusions}\label{sec-conc}

In this paper, we developed an inexact variable-metric proximal gradient-subgradient algorithm (iVPGSA) for solving a class of fractional optimization problems of the form \eqref{mainpro}. By combining a flexible error criterion with a variable metric, the proposed method allows approximate subproblem solutions while exploiting problem structure in practical implementations. Moreover, we conducted a comprehensive convergence analysis of iVPGSA, including subsequential convergence, global convergence of the entire sequence, and convergence-rate estimates. In particular, we developed a KL-based analysis framework, relying only on the classical KL property and its associated exponent, to handle the accumulated errors caused by inexact subproblem solutions without requiring a strict sufficient descent property. Our results clarify how the original KL geometry and the error decay jointly influence the convergence rate. The rate analysis may also provide useful insights for studying other inexact methods with similar perturbed descent relations. Finally, numerical experiments demonstrated the computational efficiency and advantages of our approach.

\appendix

\section{Proof of Lemma \ref{lemma-convrate}(i)}\label{apd-lem-convrate}
\begin{proof}
Let $a_k:=k^p u_k-\frac{d}{1-\rho}$. Then, we have that
\begin{equation*}
\begin{aligned}
a_{k+1}
&={\textstyle (k+1)^p u_{k+1}-\frac{d}{1-\rho}
~\leq~ k^p\left(1+\frac{1}{k}\right)^p\left[\rho u_k+\frac{d}{k^p}\right]-\frac{d}{1-\rho}} \\
&={\textstyle \rho k^p u_k\left(1+\frac{1}{k}\right)^p+d\left(1+\frac{1}{k}\right)^p
-\frac{d}{1-\rho}} \\
&={\textstyle\rho\left(a_k+\frac{d}{1-\rho}\right) \left(1+\frac{1}{k}\right)^p +d\left(1+\frac{1}{k}\right)^p-\frac{d}{1-\rho}} \\
&= {\textstyle a_{k}\left[\rho\left(1+\frac{1}{k}\right)^{p}\right]
+ \frac{d}{1-\rho}\left(1+\frac{1}{k}\right)^{p}
- \frac{d}{1-\rho}}.
\end{aligned}
\end{equation*}
Set $r_k:=\frac{d}{1-\rho}\left[\left(1+\frac{1}{k}\right)^p-1\right]\geq0$. Using the Taylor expansion $\left(1+\frac{1}{k}\right)^p=1+\frac{p}{k}+o\left(\frac{1}{k}\right)$, we obtain
\begin{equation*}
r_k=\frac{dp}{1-\rho}\cdot\frac{1}{k}+o\left(\frac{1}{k}\right)\to0.
\end{equation*}
Fix $\tilde{\rho}\in(\rho,1)$. Since $\rho\left(1+\frac{1}{k}\right)^p\to\rho$, there exists an integer $K\geq1$ such that $\rho\left(1+\frac{1}{k}\right)^p\leq\tilde{\rho}$ for all $k\geq K$. Now, let $a_k^+:=\max\{a_k,0\}$. The above estimate, together with $a_k\leq a_k^+$ and $r_k\geq0$, gives
\begin{equation*}
0\leq a_{k+1}^+
\leq\rho\left(1+\frac{1}{k}\right)^p a_k^++r_k
\leq\tilde{\rho}a_k^++r_k,\quad\forall\,k\geq K.
\end{equation*}
Since $\tilde{\rho}\in(0,1)$ and $r_k\to0$, applying \citep[Lemma 3 in Section 2.2]{p1987introduction} yields $a_k^+\to0$. Hence, $\limsup_{k\to\infty}a_k\leq0$ because $a_k\leq a_k^+$. Thus, for any $\epsilon>0$, we have that $a_k:=k^pu_k-\frac{d}{1-\rho}\leq\frac{\epsilon}{1-\rho}$ and hence $u_k \leq \frac{d+\epsilon}{1-\rho}\cdot\frac{1}{k^p}$ for all sufficiently large $k$.
\end{proof}

\section{Verification of the choices on $\{\varepsilon_k\}$ in Remark \ref{remark-epsk}}\label{apd-remark-sumcond}

\textit{Case (i)}. $\varepsilon_k=\varepsilon_0 \zeta^k$ with $0<\zeta<1$. In this case, it is easy to verify that the first two conditions in \eqref{epssumcond} hold. We next consider the third condition in \eqref{epssumcond}. Note that
\begin{equation*}
\left(\sum_{i=k}^{\infty}\varepsilon_i\right)^{1-\frac{1}{\tau}}
= \left(\frac{\varepsilon_0 \zeta^k}{1-\zeta}\right)^{1-\frac{1}{\tau}}
= \left(\frac{\varepsilon_0}{1-\zeta}\right)^{1-\frac{1}{\tau}} \zeta^{\left(1-\frac{1}{\tau}\right)k},
\end{equation*}
and for a given $\tau>1$, we have $\zeta^{1-\frac{1}{\tau}}\in\left(0,1\right)$ for any $0<\zeta<1$. Then, we see that
\begin{equation*}
\sum_{k=0}^{\infty}\left(\sum^{\infty}_{i=k}\varepsilon_i\right)^{1-\frac{1}{\tau}}
= \left(\frac{\varepsilon_0}{1-\zeta}\right)^{1-\frac{1}{\tau}}
\sum_{k=0}^{\infty}\zeta^{\left(1-\frac{1}{\tau}\right)k} < +\infty.
\end{equation*}
Thus, the third condition in \eqref{epssumcond} holds.

\textit{Case (ii)}. $\varepsilon_k=\frac{\varepsilon_0}{(k+1)^q}$ with $q>\frac{2\tau-1}{\tau-1}$. In this case, since $q>\frac{2\tau-1}{\tau-1}=2+\frac{1}{\tau-1}$, it is easy to verify that the first two conditions in \eqref{epssumcond} hold. We next consider the third condition. For every $k\geq1$,
\begin{equation*}
\left(\sum_{i=k}^{\infty}\varepsilon_i\right)^{1-\frac{1}{\tau}}
=\left(\sum_{i=k}^{\infty} \frac{\varepsilon_0}{(i+1)^{q}}\right)^{1-\frac{1}{\tau}}
\leq\left(\sum_{i=k}^{\infty}\int_i^{i+1}\frac{\varepsilon_0}{t^{q}}\,{\rm d}t\right)^{1-\frac{1}{\tau}}
=\left(\frac{\varepsilon_0}{q-1}\right)^{1-\frac{1}{\tau}} \cdot \frac{1}{k^{(q-1)(1-\frac{1}{\tau})}},
\end{equation*}
and for a given $\tau>1$, we have $(q-1)(1-\frac{1}{\tau})>1$ for any $q>\frac{2\tau-1}{\tau-1}$. Then, we see that
\begin{equation*}
\sum_{k=0}^{\infty}\left(\sum^{\infty}_{i=k}\varepsilon_i\right)^{1-\frac{1}{\tau}}
\leq \left(\sum_{i=0}^{\infty}\varepsilon_i\right)^{1-\frac{1}{\tau}}
+ \left(\frac{\varepsilon_0}{q-1}\right)^{1-\frac{1}{\tau}}
\sum_{k=1}^{\infty}\frac{1}{k^{(q-1)(1-\frac{1}{\tau})}}
< +\infty.
\end{equation*}
Thus, the third condition in \eqref{epssumcond} holds.

\section{Proofs for Section \ref{sec-num}}

\subsection{Proof of Proposition \ref{pro-scnew-l12lasso}}\label{apd-pro-l12lasso}
\begin{proof}
From the definition of $\bm{w}^{k,t}$ and the property of the proximal mapping $\texttt{prox}_{\gamma_k^{-1}r}$, we have
\begin{equation*}
0 \in \gamma_k^{-1}\partial r(\bm{w}^{k,t}) + \bm{w}^{k,t}
- (\bm{x}^k+\gamma_k^{-1}(c_k\bm{y}^k-A^{\top}\bm{z}^{k,t})),
\end{equation*}
which, together with $\bm{e}^{k,t}:=\nabla\Psi_k^{\text{lasso}}(\bm{z}^{k,t})
= -A\bm{w}^{k,t}+\bm{z}^{k,t}+\bm{b}$, deduces that
\begin{equation*}
\begin{aligned}
&0 \in \partial r(\bm{w}^{k,t}) + \gamma_k(\bm{w}^{k,t}-\bm{x}^k) - c_k\bm{y}^k + A^\top(\bm{e}^{k,t} + A\bm{w}^{k,t} - \bm{b}), \\
&\Longrightarrow ~~- A^\top \bm{e}^{k,t} \in \partial r(\bm{w}^{k,t}) + A^\top (A\bm{w}^{k,t}-\bm{b})
- c_k\bm{y}^k + \gamma_k(\bm{w}^{k,t}-\bm{x}^k).
\end{aligned}
\end{equation*}
It is then easy to see that $\bm{w}^{k,t}$ and the error pair $(-A^{\top}\bm{e}^{k,t}, \,0)$ satisfy condition \eqref{inexcondH-x}. Using this relation, we can obtain the desired results.
\end{proof}

\subsection{Proof of Proposition \ref{pro-scnew-l12con}}
\label{apd-pro-l12cos}

\begin{proof}
From the definition of $\bm{w}^{k,t}$ and the property of the proximal mapping $\texttt{prox}_{\gamma_k^{-1}\|\cdot\|_1}$, we have
\begin{equation*}
0 \in \partial \|\bm{w}^{k,t}\|_1+\gamma_k\left[\bm{w}^{k,t}-\left({\bm{x}}^{k}+\gamma_k^{-1} c_k\bm{y}^k-\gamma_k^{-1} A^{\top} {\bm{z}}^{k,t}\right)\right].
\end{equation*}
Let $\bm{d}_1^{k,t}:=\gamma_k\left[\big({\bm{x}}^{k}+\gamma_k^{-1} c_k\bm{y}^k-\gamma_k^{-1} A^{\top}\bm{z}^{k,t}\big)-\bm{w}^{k,t}\right]$. Clearly, $\bm{d}_1^{k,t} \in \partial\|\bm{w}^{k,t}\|_1$. Then, for any $\bm{u}\in\mathbb{R}^n$, we have that
\begin{equation*}
\begin{aligned}
\quad \|\bm{u}\|_1
&\geq \|\bm{w}^{k,t}\|_1 + \langle{\bm{d}}_1^{k,t}, \,\bm{u}-\bm{w}^{k,t}\rangle \\
&\geq \|\widetilde{\bm{w}}^{k,t}\|_1 + \langle{\bm{d}}_1^{k,t}, \,\bm{u}-\widetilde{\bm{w}}^{k,t}\rangle
- \big( \|\widetilde{\bm{w}}^{k,t}\|_1 - \|\bm{w}^{k,t}\|_1
- \langle{\bm{d}}_1^{k,t},\,\widetilde{\bm{w}}^{k,t}-\bm{w}^{k,t}\rangle\big),
\end{aligned}
\end{equation*}
which implies that $\bm{d}_1^{k,t} \in \partial_{\delta_{k,t}^1} \|\widetilde{\bm{w}}^{k,t}\|_1$ with $\delta^1_{k,t} := \|\widetilde{\bm{w}}^{k,t}\|_1 - \|\bm{w}^{k,t}\|_1
- \langle{\bm{d}}_1^{k,t},\,\widetilde{\bm{w}}^{k,t}-\bm{w}^{k,t}\rangle\geq0$ (due to the convexity of $\|\cdot\|_1$ and $\bm{d}_1^{k,t} \in \partial\|\bm{w}^{k,t}\|_1$). Thus, we have
\begin{equation}\label{d1partialg}
0 \in \partial_{\delta_{k,t}^1} \|\widetilde{\bm{w}}^{k,t}\|_1
- c_k\bm{y}^k + A^{\top}\bm{z}^{k,t} + \gamma_k(\bm{w}^{k,t}-\bm{x}^{k}).
\end{equation}
On the other hand, we see from the property of projection onto the convex set $\{\bm{x} \in \mathbb{R}^{m}:\|\bm{x}\|\leq \sigma\}$ that, for any $\bm{u}$ satisfying $\|\bm{u}\|\leq\sigma$,
\begin{equation*}
\big\langle \bm{u}-\Pi_{\sigma}\big(A \bm{x}^k-\bm{b}+\gamma_k^{-1}\bm{z}^{k,t}\big), \,\big(A{\bm{x}}^k-\bm{b}+\gamma_k^{-1}\bm{z}^{k,t}\big)
- \Pi_{\sigma}\big(A\bm{x}^k-\bm{b}+\gamma_k^{-1}\bm{z}^{k,t}\big)\big\rangle
\leq 0.
\end{equation*}
Let $\bm{d}_2^{k,t}:=\gamma_k\left[\big(A{\bm{x}}^k-\bm{b}+\gamma_k^{-1}\bm{z}^{k,t}\big)
- \Pi_{\sigma}\big(A\bm{x}^k-\bm{b}+\gamma_k^{-1}\bm{z}^{k,t}\big)\right]$ and recall the definition of $\bm{e}^{k, t}$ that $\Pi_{\sigma}\big(A\bm{x}^k-\bm{b}+\gamma_k^{-1}\bm{z}^{k,t}\big)=A \bm{w}^{k,t}-\bm{b}+\bm{e}^{k,t}$. Then, we can see from the above relation that
\begin{equation*}
\begin{aligned}
0
\geq \big\langle \bm{u} - \big(A\bm{w}^{k,t}-\bm{b}+\bm{e}^{k,t}\big),
\,\bm{d}_2^{k,t}\big\rangle
= \langle \bm{u}-(A\widetilde{\bm{w}}^{k,t}-\bm{b}), \,\bm{d}_2^{k,t}\rangle
- \langle \bm{e}^{k,t} - A(\widetilde{\bm{w}}^{k,t}-\bm{w}^{k,t}), \,\bm{d}_2^{k,t}\rangle,
\end{aligned}
\end{equation*}
and hence
\begin{equation*}
\langle \bm{u}-(A\widetilde{\bm{w}}^{k,t}-\bm{b}), \,\bm{d}_2^{k,t}\rangle
\leq \langle \bm{e}^{k,t} - A(\widetilde{\bm{w}}^{k,t}-\bm{w}^{k,t}), \,\bm{d}_2^{k,t}\rangle
\leq \big|\langle \bm{e}^{k,t} - A(\widetilde{\bm{w}}^{k,t}-\bm{w}^{k,t}), \,\bm{d}_2^{k,t}\rangle\big|.
\end{equation*}
This, together with $\|A\widetilde{\bm{w}}^{k,t}-\bm{b}\|\leq\sigma$, implies that $\bm{d}_2^{k,t} \in \partial_{\delta_{k,t}^2}\iota_{\sigma}(A \widetilde{\bm{w}}^{k,t}-\bm{b})$ with $\delta_{k,t}^2:=|\langle \bm{e}^{k,t} - A(\widetilde{\bm{w}}^{k,t}-\bm{w}^{k,t}), \,\bm{d}_2^{k,t}\rangle|$. Thus, we have
\begin{equation}\label{d2partial-ballsigma}
-\gamma_k\bm{e}^{k,t}
\in \partial_{\delta_{k,t}^2}\iota_{\sigma}(A\widetilde{\bm{w}}^{k,t}-\bm{b})
- \bm{z}^{k,t} + \gamma_kA(\bm{w}^{k,t}-\bm{x}^{k}).
\end{equation}
Now, combining \eqref{d1partialg} and \eqref{d2partial-ballsigma}, we can obtain that
\begin{equation*}
-\gamma_k A^{\top}\bm{e}^{k,t}
\in \partial_{\delta_{k,t}^1}\|\widetilde{\bm{w}}^{k,t}\|_1
+ A^{\top}\partial_{\delta_{k,t}^2}\iota_{\sigma}(A\widetilde{\bm{w}}^{k,t}-\bm{b})
- c_k\bm{y}^k + \gamma_k\big(\bm{w}^{k,t}-\bm{x}^{k}\big)
+ \gamma_kA^{\top}A\big(\bm{w}^{k,t}-\bm{x}^{k}\big),
\end{equation*}
which further implies that
\begin{equation*}
\begin{aligned}
\Delta^{k,t}
&:=-\gamma_kA^{\top}\bm{e}^{k,t}
+ \gamma_k(\widetilde{\bm{w}}^{k,t}-\bm{w}^{k,t})
+ \gamma_k A^{\top}A(\widetilde{\bm{w}}^{k,t}-\bm{w}^{k,t}) \\
&\in \partial_{\delta_{k,t}^1}\|\widetilde{\bm{w}}^{k,t}\|_1
+ A^{\top}\partial_{\delta_{k,t}^2}\iota_{\sigma}(A\widetilde{\bm{w}}^{k,t}-\bm{b})
- c_k\bm{y}^k
+ \gamma_k\big(\widetilde{\bm{w}}^{k,t}-\bm{x}^{k}\big)
+ \gamma_kA^{\top}A\big(\widetilde{\bm{w}}^{k,t}-\bm{x}^{k}\big)  \\
&\subset\partial_{\delta_{k,t}}\big(\|\cdot\|_1+\iota_{\sigma}(A\cdot-\bm{b})\big)(\widetilde{\bm{w}}^{k,t})
- c_k\bm{y}^k + \gamma_k\big(\widetilde{\bm{w}}^{k,t}-\bm{x}^{k}\big)
+ \gamma_kA^{\top}A\big(\widetilde{\bm{w}}^{k,t}-\bm{x}^{k}\big),
\end{aligned}
\end{equation*}
where $\delta_{k,t} := \delta^1_{k,t} + \delta^2_{k,t}$ and the last inclusion can be verified by the definition of the $\delta$-subdifferential of $\|\cdot\|_1+\iota_{\sigma}(A\cdot-\bm{b})$ at $\widetilde{\bm{w}}^{k,t}$. Then, one can see from the above relation that the point $\widetilde{\bm{w}}^{k,t}$ associated with the error pair $(\Delta^{k,t}, \,\delta_{k,t})$ satisfies condition \eqref{inexcondH-x}. Using this relation, we can readily obtain the desired results.
\end{proof}



\bibliographystyle{plain}
\bibliography{Ref_iVPGSA}

@article{zll2026parameterized,
  title={{Parameterized proximal-gradient algorithms for {$L_{1}/L_{2}$} sparse signal recovery}},
  author={{Zhang}, {N.} and {Liu}, {X.} and {Li}, {Q.}},
  journal={Applied and Computational Harmonic Analysis},
  volume={82},
  pages={101835},
  year={2026}
}

@article{zpl2024vmipg,
  title={A {VMiPG} method for composite optimization with nonsmooth term having no closed-form proximal mapping},
  author={Zhang, T. and Pan, S. and Liu, R.},
  journal={Journal of Scientific Computing},
  volume={101},
  pages={69},
  year={2024}
}

@article{s2017variable,
  title={The variable metric forward-backward splitting algorithm under mild differentiability assumptions},
  author={Salzo, S.},
  journal={SIAM Journal on Optimization},
  volume={27},
  number={4},
  pages={2153--2181},
  year={2017}
}

@article{bpr2016variable,
  title={A variable metric forward-backward method with extrapolation},
  author={Bonettini, S. and Porta, F. and Ruggiero, V.},
  journal={SIAM Journal on Scientific Computing},
  volume={38},
  number={4},
  pages={A2558--A2584},
  year={2016}
}

@article{n1983a,
  title={A method of solving a convex programming problem with convergence rate {$O(1/k^2)$}},
  author={Nesterov, Y.},
  journal={Soviet Mathematics Doklady},
  volume={27},
  number={2},
  pages={372--376},
  year={1983}
}

@book{n2004introductory,
  title={Introductory Lectures on Convex Optimization: A Basic Course},
  author={Nesterov, Y.},
  year={2004},
  publisher={Kluwer Academic Publishers},
  address={Boston}
}

@article{n2013gradient,
  title={Gradient methods for minimizing composite functions},
  author={Nesterov, Y.},
  journal={Mathematical Programming},
  volume={140},
  number={1},
  pages={125--161},
  year={2013}
}

@article{sun2021sequence,
  title={Sequence convergence of inexact nonconvex and nonsmooth algorithms with more realistic assumptions},
  author={Sun, T.},
  journal={Numerical Functional Analysis and Optimization},
  volume={42},
  number={2},
  pages={234--250},
  year={2021}
}

@article{lmq2023convergence,
  title={Convergence of random reshuffling under the {K}urdyka--{{\L}}ojasiewicz inequality},
  author={Li, X. and Milzarek, A. and Qiu, J.},
  journal={SIAM Journal on Optimization},
  volume={33},
  number={2},
  pages={1092--1120},
  year={2023},
  publisher={SIAM}
}

@article{qmlm2024kl,
  title={A {KL}-based analysis framework with applications to non-descent optimization methods},
  author={Qiu, J. and Ma, B. and Li, X. and Milzarek, A.},
  journal={arXiv preprint arXiv:2406.02273},
  year={2024}
}

@article{hl2023convergence,
  title={The convergence properties of infeasible inexact proximal alternating linearized minimization},
  author={Hu, Y. and Liu, X.},
  journal={Science China Mathematics},
  volume={66},
  number={10},
  pages={2385--2410},
  year={2023}
}

@article{UCIdata,
  title={The {UCI} machine learning repository},
  author={Kelly, M. and Longjohn, R. and Nottingham, K.},
  journal={https://archive.ics.uci.edu},
  year={}
}

@article{zzst2020efficient,
  title={An efficient {H}essian based algorithm for solving large-scale sparse group {L}asso problems},
  author={Zhang, Y. and Zhang, N. and Sun, D. and Toh, K.-C.},
  journal={Mathematical Programming},
  volume={179},
  number={1},
  pages={223--263},
  year={2020}
}

@article{ab2009on,
  title={On the convergence of the proximal algorithm for nonsmooth functions involving analytic features},
  author={Attouch, H. and Bolte, J.},
  journal={Mathematical Programming},
  volume={116},
  number={1},
  pages={5--16},
  year={2009}
}

@article{abrs2010proximal,
  title={Proximal alternating minimization and projection methods for nonconvex problems: {A}n approach based on the {K}urdyka-{{\L}}ojasiewicz inequality},
  author={Attouch, H. and Bolte, J. and Redont, P. and Soubeyran, A.},
  journal={Mathematics of Operations Research},
  volume={35},
  number={2},
  pages={438--457},
  year={2010}
}

@article{abs2013convergence,
  title={Convergence of descent methods for semi-algebraic and tame problems: {P}roximal algorithms, forward--backward splitting, and regularized {G}auss--{S}eidel methods},
  author={Attouch, H. and Bolte, J. and Svaiter, B.F.},
  journal={Mathematical Programming},
  volume={137},
  number={1},
  pages={91--129},
  year={2013}
}

@book{bc2011convex,
  title={Convex Analysis and Monotone Operator Theory in Hilbert Spaces},
  author={Bauschke, H.H. and Combettes, P.L.},
  volume={408},
  year={2011},
  publisher={Springer}
}

@article{bt2009a,
  title={A fast iterative shrinkage-thresholding algorithm for linear inverse problems},
  author={Beck, A. and Teboulle, M.},
  journal={SIAM Journal on Imaging Sciences},
  volume={2},
  number={1},
  pages={183--202},
  year={2009}
}

@article{bdl2007the,
  title={The {{\L}}ojasiewicz inequality for nonsmooth subanalytic functions with applications to subgradient dynamical systems},
  author={Bolte, J. and Daniilidis, A. and Lewis, A.},
  journal={SIAM Journal on Optimization},
  volume={17},
  number={4},
  pages={1205--1223},
  year={2007}
}

@article{bst2014proximal,
  title={Proximal alternating linearized minimization for nonconvex and nonsmooth problems},
  author={Bolte, J. and Sabach, S. and {Teboulle}, M.},
  journal={Mathematical Programming},
  volume={146},
  number={1},
  pages={459--494},
  year={2014}
}

@article{bc2017proximal,
  title={Proximal-gradient algorithms for fractional programming},
  author={Bo{\c{t}}, R.I. and Csetnek, E.R.},
  journal={Optimization},
  volume={66},
  number={8},
  pages={1383--1396},
  year={2017},
  publisher={Taylor \& Francis}
}

@article{bdl2022extrapolated,
  title={Extrapolated proximal subgradient algorithms for nonconvex and nonsmooth fractional programs},
  author={Bo{\c{t}}, R.I. and Dao, M.N. and Li, G.},
  journal={Mathematics of Operations Research},
  volume={47},
  number={3},
  pages={2415--2443},
  year={2022},
  publisher={INFORMS}
}

@article{bdl2023inertial,
  title={Inertial proximal block coordinate method for a class of nonsmooth sum-of-ratios optimization problems},
  author={Bo{\c{t}}, R.I. and Dao, M.N. and Li, G.},
  journal={SIAM Journal on Optimization},
  volume={33},
  number={2},
  pages={361--393},
  year={2023},
  publisher={SIAM}
}

@article{blt2023full,
  title={Full splitting algorithms for fractional programs with structured numerators and denominators},
  author={Bo{\c{t}}, R.I. and Li, G. and Tao, M.},
  journal={SIAM Journal on Optimization},
  volume={35},
  number={4},
  pages={2623--2653},
  year={2025},
  publisher={SIAM}
}

@book{cp2021modern,
  title={Modern Nonconvex Nondifferentiable Optimization},
  author={Cui, Y. and Pang, J.-S.},
  year={2021},
  publisher={SIAM}
}

@article{dinkelbach1967nonlinear,
  title={On nonlinear fractional programming},
  author={Dinkelbach, W.},
  journal={Management Science},
  volume={13},
  number={7},
  pages={492--498},
  year={1967},
  publisher={INFORMS}
}

@article{elx2013method,
  title={A method for finding structured sparse solutions to nonnegative least squares problems with applications},
  author={Esser, E. and Lou, Y. and Xin, J.},
  journal={SIAM Journal on Imaging Sciences},
  volume={6},
  number={4},
  pages={2010--2046},
  year={2013},
  publisher={SIAM}
}

@article{fgp2015splitting,
  title={Splitting methods with variable metric for {K}urdyka--{{\L}}ojasiewicz functions and general convergence rates},
  author={Frankel, P. and Garrigos, G. and Peypouquet, J.},
  journal={Journal of Optimization Theory and Applications},
  volume={165},
  number={3},
  pages={874--900},
  year={2015}
}

@book{hl1993convex,
  title={Convex Analysis and Minimization Algorithms II: Advanced Theory and Bundle Methods},
  author={Hiriart-Urruty, J.B. and Lemar{\'e}chal, C.},
  year={1993},
  publisher={Springer}
}

@article{jlsw2012image,
  title={Image deconvolution using a characterization of sharp images in wavelet domain},
  author={Ji, H. and Li, J. and Shen, Z. and Wang, K.},
  journal={Applied and Computational Harmonic Analysis},
  volume={32},
  number={2},
  pages={295--304},
  year={2012},
  publisher={Elsevier}
}

@article{lp2017calculus,
  title={Calculus of the exponent of {K}urdyka--{{\L}}ojasiewicz inequality and its applications to linear convergence of first--order methods},
  author={Li, G. and Pong, T.K.},
  journal={Foundations of Computational Mathematics},
  volume={18},
  number={5},
  pages={1199--1232},
  year={2018}
}

@article{lszz2022proximal,
  title={A proximal algorithm with backtracked extrapolation for a class of structured fractional programming},
  author={Li, Q. and Shen, L. and Zhang, N. and Zhou, J.},
  journal={Applied and Computational Harmonic Analysis},
  volume={56},
  pages={98--122},
  year={2022},
  publisher={Elsevier}
}

@article{lst2018highly,
  title={A highly efficient semismooth {N}ewton augmented {L}agrangian method for solving {L}asso problems},
  author={Li, X. and Sun, D. and Toh, K.-C.},
  journal={SIAM Journal on Optimization},
  volume={28},
  number={1},
  pages={433--458},
  year={2018},
  publisher={SIAM}
}

@book{p1987introduction,
  title={Introduction to Optimization},
  author={Polyak, B.T.},
  year={1987},
  publisher={Optimization Software Inc.},
  address={New York}
}

@article{rwdl2019scale,
  title={A scale-invariant approach for sparse signal recovery},
  author={Rahimi, Y. and Wang, C. and Dong, H. and Lou, Y.},
  journal={SIAM Journal on Scientific Computing},
  volume={41},
  number={6},
  pages={A3649--A3672},
  year={2019},
  publisher={SIAM}
}

@article{rpdcp2015euclid,
  title={Euclid in a taxicab: {S}parse blind deconvolution with smoothed $\ell_1/\ell_2$ regularization},
  author={Repetti, A. and Pham, M.Q. and Duval, L. and Chouzenoux, E. and Pesquet, J.-C.},
  journal={IEEE Signal Processing Letters},
  volume={22},
  number={5},
  pages={539--543},
  year={2015}
}

@book{r1970convex,
  title={Convex Analysis},
  author={Rockafellar, R.T.},
  year={1970},
  publisher={Princeton University Press},
  address={Princeton}
}

@book{rw1998variational,
  title={Variational Analysis},
  author={Rockafellar, R.T. and Wets, R.J.-B.},
  year={1998},
  publisher={Springer}
}

@article{twst2020sparse,
  title={A sparse semismooth {N}ewton based proximal majorization-minimization algorithm for nonconvex square-root-loss regression problems},
  author={Tang, P. and Wang, C. and Sun, D. and Toh, K.-C.},
  journal={Journal of Machine Learning Research},
  volume={21},
  number={226},
  pages={1--38},
  year={2020}
}

@article{tao2022minimization,
  title={Minimization of ${L}_1$ Over ${L}_2$ for Sparse Signal Recovery with Convergence Guarantee},
  author={Tao, M.},
  journal={SIAM Journal on Scientific Computing},
  volume={44},
  number={2},
  pages={A770--A797},
  year={2022},
  publisher={SIAM}
}

@article{tz2023study,
  title={Study on ${L}_1$ over ${L}_2$ Minimization for Nonnegative Signal Recovery},
  author={Tao, M. and Zhang, X.-P.},
  journal={Journal of Scientific Computing},
  volume={95},
  number={3},
  pages={94},
  year={2023}
}

@article{tzx2024partly,
  title={On Partly Smoothness, Activity Identification and Faster Algorithms of ${L}_1$ over ${L}_2$ Minimization},
  author={Tao, M. and Zhang, X.-P. and Xia, Z.-H.},
  journal={IEEE Transactions on Signal Processing},
  volume={72},
  pages={2874--2889},
  year={2024}
}

@article{vf2009probing,
  title={Probing the {P}areto frontier for basis pursuit solutions},
  author={Van Den Berg, E. and Friedlander, M. P.},
  journal={SIAM Journal on Scientific Computing},
  volume={31},
  number={2},
  pages={890--912},
  year={2009},
  publisher={SIAM}
}

@article{wtcnl2022minimizing,
  title={Minimizing ${L}_1$ over ${L}_2$ norms on the gradient},
  author={Wang, C. and Tao, M. and Chuah, C.-N. and Nagy, J.G. and Lou, Y.},
  journal={Inverse Problems},
  volume={38},
  number={6},
  pages={065011},
  year={2022},
  publisher={IOP Publishing}
}

@article{wtnl2021limited,
  title={Limited-angle {CT} reconstruction via the ${L}_1/{L}_2$ minimization},
  author={Wang, C. and Tao, M. and Nagy, J.G. and Lou, Y.},
  journal={SIAM Journal on Imaging Sciences},
  volume={14},
  number={2},
  pages={749--777},
  year={2021},
  publisher={SIAM}
}

@article{wcp2017a,
  title={A proximal difference-of-convex algorithm with extrapolation},
  author={Wen, B. and Chen, X. and Pong, T.K.},
  journal={Computational Optimization and Applications},
  volume={69},
  number={2},
  pages={297--324},
  year={2018}
}

@article{y2024proximal,
  title={Proximal gradient method with extrapolation and line search for a class of nonconvex and nonsmooth problems},
  author={Yang, L.},
  journal={Journal of Optimization Theory and Applications},
  volume={200},
  number={1},
  pages={68--103},
  year={2024}
}

@article{yex2014ratio,
  title={Ratio and difference of $\ell_1 $ and $\ell_2$ norms and sparse representation with coherent dictionaries},
  author={Yin, P. and Esser, E. and Xin, J.},
  journal={Communications in Information and Systems},
  volume={14},
  number={2},
  pages={87--109},
  year={2014},
  publisher={International Press of Boston}
}

@inproceedings{yuan2023coordinate,
  title={Coordinate descent methods for fractional minimization},
  author={Yuan, G.},
  booktitle={Proceedings of the 40th International Conference on Machine Learning},
  pages={40488--40518},
  year={2023},
  organization={PMLR}
}

@inproceedings{yuan2025admm,
  title={{ADMM} for structured fractional minimization},
  author={Yuan, G.},
  booktitle={International Conference on Learning Representations},
  year={2025}
}

@article{zyp2021analysis,
  title={Analysis and algorithms for some compressed sensing models based on {L}1/{L}2 minimization},
  author={Zeng, L. and Yu, P. and Pong, T.K.},
  journal={SIAM Journal on Optimization},
  volume={31},
  number={2},
  pages={1576--1603},
  year={2021},
  publisher={SIAM}
}

@article{zts2023learning,
  title={Learning graph {L}aplacian with {MCP}},
  author={Zhang, Y. and Toh, K.-C. and Sun, D.},
  journal={Optimization Methods and Software},
  volume = {39},
  number = {3},
  pages = {569--600},
  year = {2024},
  publisher={Taylor \& Francis}
}

@article{zl2022first,
  title={First-order algorithms for a class of fractional optimization problems},
  author={Zhang, N. and Li, Q.},
  journal={SIAM Journal on Optimization},
  volume={32},
  number={1},
  pages={100--129},
  year={2022}
}

@article{zst2010newton,
  title={A {N}ewton-{CG} augmented {L}agrangian method for semidefinite programming},
  author={Zhao, X.-Y. and Sun, D. and Toh, K.-C.},
  journal={SIAM Journal on Optimization},
  volume={20},
  number={4},
  pages={1737--1765},
  year={2010}
}

@article{zzl2023equivalent,
  title={An equivalent reformulation and {multiproximity} gradient algorithms for a class of nonsmooth fractional programming},
  author={Zhou, J. and Zhang, N. and Li, Q.},
  journal={Mathematics of Operations Research},
  volume={50},
  number={4},
  pages={3010--3038},
  year={2025},
  doi={10.1287/moor.2024.0457}
}

\end{document}